\documentclass[amssymb,amsfonts,refcheck,11pt,verbatim,righttag]{amsart}

\usepackage{amssymb,mathrsfs,hyperref}

\usepackage{graphicx,color,tikz,caption,subcaption,tabularx}
\RequirePackage{yhmath} 

\numberwithin{equation}{section}
\theoremstyle{plain}
\usepackage[utf8]{inputenc}  
\usepackage[T1]{fontenc}

\usepackage{lmodern}
\usepackage{bm}
\usepackage[english]{babel}

\usepackage{enumitem}
\usepackage{float}
\usepackage{indentfirst}
\usepackage{mathtools}
\usepackage{dsfont} 
\usepackage{epigraph}
\usepackage{hyperref}
\usepackage{url}
\hypersetup{colorlinks,citecolor=blue,filecolor=blue,linkcolor=blue,urlcolor=blue,pdfauthor=author}
\usepackage{verbatim}
\usepackage{setspace}
\newtheorem{theorem}{Theorem}[section]

\newtheorem{proposition}[theorem]{Proposition}

\newtheorem{lemma}[theorem]{Lemma}

\theoremstyle{definition}
\newtheorem{definition}[theorem]{Definition}

\theoremstyle{remark}
\newtheorem{remark}[theorem]{Remark}

\DeclareMathOperator\supp{supp}

\allowdisplaybreaks

\setlist[itemize]{label=\textbullet}

\begin{document}

\title{Variational principles for Hausdorff and packing dimensions of fractal percolation on self-affine~sponges}

\author{Julien Barral}
\address{Laboratoire de G\'eom\'etrie, Analyse et Applications, CNRS, UMR 7539, Universit\'e Sorbonne Paris Nord, CNRS, UMR 7539,  F-93430, Villetaneuse, France}
\email{barral@math.univ-paris13.fr}
\author{Guilhem Brunet}
\address{Laboratoire de G\'eom\'etrie, Analyse et Applications, CNRS, UMR 7539, Universit\'e Sorbonne Paris Nord, CNRS, UMR 7539,  F-93430, Villetaneuse, France}
\email{brunet@math.univ-paris13.fr}


\date{}

\begin{abstract} We establish variational principles for the Hausdorff and packing dimensions of a class of statistically self-affine sponges, including in particular fractal percolation sets obtained from Barański and Gatzouras-Lalley carpets and sponges. Our first step is to compute the Hausdorff and packing dimensions of non-degenerate inhomogeneous Mandelbrot measures supported on the associated random limit sets. This is not a straightforward combination of the existing approaches for the deterministic inhomogeneous Bernoulli  measures  and the Mandelbrot measures on random Sierpi\'nski sponges; it reveals new structural features. The variational principles rely on a specific subclass of inhomogeneous Mandelbrot measures, which are connected to localized digit frequencies in the underlying coding space. This connection makes it possible to construct effective coverings of the random limit set, leading to sharp upper bounds for its Hausdorff and packing dimensions.
\end{abstract}

\maketitle

\raggedbottom
\pagestyle{plain}

\section{Introduction}\label{introduction}

Let $\{f_i\}_{i\in\mathcal I}$ be an iterated function system (IFS) consisting of a non-empty and finite collection of strictly contracting  maps of the Euclidean space $\mathbb{R}^d$ $(d\ge 1)$. According to Hutchinson \cite{Hutchinson}, there exists a unique non empty compact set $K$ such that 
\begin{equation}\label{Hut}
K = \bigcup_{i\in\mathcal{I}} f_i(K),
\end{equation}
called the attractor of the IFS. We assume that the maps $f_i$ have no common fixed points, so that $K$ is nontrivial, and that they are affine maps $x\mapsto A_ix+t_i$, so that $K$ is self-affine. Also, we assume that the $A_i$ are invertible. Associated to $\{f_i\}_{i\in\mathcal I}$ are the Borel probability measures $\mu$ obeying a self-affinity relation 
\begin{equation}\label{SAE}
\mu=\sum_{i\in\mathcal I}p_i\, \mu\circ f_i^{-1},
\end{equation}
where  $(p_i)_{i\in\mathcal I}$ is a probability vector. If $\nu$  stands for the Bernoulli product measure $\otimes_{k=1}^\infty(\sum_{i\in\mathcal I}p_i\delta_i)$ on $\Sigma=\mathcal I^{\mathbb{N}^+}$ endowed with the $\sigma$-algebra generated by cylinders, the unique self-affine Borel probability measure $\mu$ obeying \eqref{SAE}  is the pushforward $\pi_*\nu$ of $\nu$ by the coding map from $\Sigma$ to $K$ defined as
$$
\pi : \boldsymbol{i}=i_1i_2\cdots  \in \Sigma \longmapsto \lim_{k\to +\infty} f_{i_1}\circ \cdots\circ f_{i_k}(0).
$$
The dimension theory of such sets and measures is an active area of research.~A fundamental result by Falconer \cite{ref28} states that if the linear parts $A_i$, $i\in\mathcal I$, have   operator norms $<\frac{1}{3}$ (this bound can be relaxed to $<1/2$~\cite{Solomyak}), then for $\mathcal{L}^{d \#\mathcal I}$-almost every choice of $(t_i)_{i\in\mathcal I}$ ($\mathcal{L}$ denotes the $1$-dimensional Lebesgue measure), $\dim_H K=\dim_B K=\min(d,\dim_a(K))$, where  $\dim_H$ and $\dim_B$ denote the Hausdorff and box-counting dimensions, and $\dim_a(K)$ is the affinity dimension of  $K$ defined thanks to  the singular values of the elements of the semigroup $S$ generated by $\{A_i:\, i\in\mathcal I\}$. The counterpart to $\dim_a(K)$ for the measure $\pi_*\nu$ is the  Lyapunov dimension $\dim_L(\nu,T)$ of $\nu$~\cite{JPS}, where $T$ is the shift operation on $\Sigma$. This dimension is also defined for any $T$-invariant probability measure~$\eta$, and  expressed in terms of the entropy of $\eta$ and the Lyapunov exponents of the system $(f_i)_{i \in\mathcal{I}}$ as seen from~$\eta$. K\"aenm\"aki~\cite{Kaenmaki} showed that for some $T$-ergodic probability measure $\eta$ one has $\min(\dim_L(\eta,T),d)=\min (\dim_aK,d)$ and for $\mathcal{L}^{d \#\mathcal I}$-a.e.  $(t_i)_{i\in\mathcal I}$,  $\min(\dim_L(\eta,T),d)=\dim_H(\pi_*\eta)$ (Hausdorff and packing dimensions of a measure are defined in Section~\ref{appendix}); thus $\dim_H K=\sup\{\dim_H (\pi_*\rho):\, \rho \text{ is }T\text{-invariant}\}$. On the other hand, the set of exceptions~to the validity of the formula $\dim_H K=\min(\dim_aK,d)$ contains classical self-affine~sets such as self-affine Sierpi\'nski carpets and sponges \cite{ref3,ref12,ref18} and their generalizations \cite{ref10,ref1,ref5,Kol}. Though we will focus on such sponges,  we continue our overview of the positive results known about the validity of  $\dim_H= \min(\dim_aK,d)$; this will naturally lead to introduce the starting point of our study, namely a result by Das and Simmons~in~\cite{ref5}.

Considerable progress has been made over the past fifteen years in developing checkable  sufficient conditions on the IFS under which  $\dim_HK=\min (\dim_a K, d)$ and $\dim_H(\pi_*\nu)=\min(\dim_L(\nu,T),d)$, and possible variational principles relating these quantities.  A first breakthrough was made by Hochman when $d=1$~\cite{Hochman1}: he replaced the classical open set condition (OSC)\footnote{Recall that $\{f_i\}_{i\in\mathcal I}$ satisfies the OSC  if there exists a non-empty open set $U$ such that the sets $f_i(U)$, $i\in\mathcal I$, are pairwise disjoint and all included in $U$, and the strong OSC  if, moreover, $U$ can be chosen so that $K\cap U\neq\emptyset$. It satisfies the ESC if it generates a free semi-group and there exists $\epsilon>0$ such that for all $k\ge 1$ and all $i_1\cdots i_k\neq j_1\cdots j_k$ in $\mathcal I^k$, $\| f_{i_1}\circ \cdots\circ f_{i_k}- f_{j_1}\circ \cdots\circ f_{j_k}\|\ge\epsilon^k$ (in the self-similar case, this can be weakened to hold only for infinitely many $k$).}~\cite{Moran,Hutchinson} with  the much weaker so-called exponential separation condition (ESC), and he used  ideas from additive combinatorics to  show the desired equalities. He later extended his result to higher dimensional self-similar systems, by adding some natural assumptions, in particular an irreducibility property for the semi-group~$S$~\cite{Hochman2}. In the planar self-affine setting, B\'ar\'any, Hochman and Rapaport \cite{ref2} obtained  $\dim_H(\pi_*\nu)=\min(\dim_L(\nu,T),d)$ under the assumptions that $S$ is strongly irreducible (no finite union of nontrivial subspaces of $\mathbb{R}^d$ is invariant by $S$), proximal, and that the strong OSC (SOSC) holds.  Subsequently, Hochman and Rapaport \cite{ref40} relaxed the SOSC to the~ESC, and  Rapaport \cite{Rapaport1} extended the result to $d=3$ under the SOSC. Once $\dim_H(\pi_*\nu)=\min(\dim_L(\nu,T),d)$ is obtained, it is combined with results by Morris and Shmerkin ($d=2$) and Morris and Sert ($d\ge 3$) \cite{MSh,MSert}, which state that the Lyapunov dimension of the K\"aenm\"aki measure  is the supremum of  those of Bernoulli product measures associated to subsystems obtained by iterating  the original IFS. This leads to the conclusion $\dim_HK=\min (\dim_a K, d)$.~When the semi-group $S$ preserves  a nontrivial linear subspace, the formulas are known to hold under the ESC, subject to  restrictions in specific planar situations (B\'ar\'any, Rams and Simon \cite{BRS,ref39}, B\'ar\'any, Hochman and Rapaport~\cite{ref31,ref40}).  They also hold for any $d\ge 2$ when the $A_i$ are diagonal and the maps $f_i$, restricted to each principal direction, define an IFS satisfying the~ESC, provided some additional mild condition are satisfied (Rapaport~\cite{Rapaport2}).

Still in the diagonal case, for $d\ge3$, Das and Simmons~\cite{ref5} investigated self-affine Gatzouras–Lalley sponges (see the definition below), for which the restrictions of the maps $f_i$ to some principal subspaces (i.e. subspaces generated by finitely many principal directions) form a self-affine IFS with exact overlaps. Such overlaps typically imply that $\dim_H K<\dim_B K< \min(\dim_aK,d)$. They exhibited examples for which $\dim_H K>\sup\{\dim_H (\pi_*\rho):\, \rho\text{ is }T\text{-invariant}\}=\sup\{\dim_H(\pi_*\nu):\, \nu\text{ Bernoulli}\}$, in sharp contrast to the Gatzouras-Lalley carpets for which the three last quantities are equal. This phenomenon raises the natural question of identifying a class of measures, related to the construction of  $K$, over which a variational principle for $\dim_H K$ could be based. In~\cite{ref5}, a class of inhomogeneous Bernoulli measures is proposed  (see the discussion before Theorem~\ref{dimKd}), but the corresponding variational principle has not been yet established.

In this paper we prove variational principles for the Hausdorff and packing dimensions of a class of statistically self-affine sponges including some random versions of self-affine Gatzouras-Lalley sponges; this covers the deterministic case, for which  the variational principle associated to $\dim_H K$ differs from that considered in \cite{ref5}. Beyond the problem raised by Das and Simmons, our motivation also stems from the fact that, for the type of randomization we consider—namely, a fractal percolation on $K$—the studies of the Hausdorff dimension of random statistically self-affine Sierpiński carpets \cite{ref9} and sponges~\cite{ref6} suggest that the richer geometric structure of Gatzouras–Lalley sponges is likely to give rise to new phenomena and developments. We base our study on the random counterpart of inhomogeneous Bernoulli measures, namely inhomogeneous Mandelbrot measures. Determining the Hausdorff and packing dimensions of such a measure indeed is not simply a matter of combining formulas and  techniques from the deterministic inhomogeneous case and the study of homogeneous Mandelbrot measures on random Sierpi\'nski sponges; rather, it uncovers new structural features. The variational principles rely on a natural connection between a certain subclass of these measures and sequences of localized digit frequencies associated with points in the coding space. This relation enables the construction of suitable coverings, which in turn yield sharp upper bounds on the dimensions.

Let us start with the Hausdorff dimension in the planar case.

\subsection{The planar case. Statistically self-affine Gatzouras-Lalley and Bara\'nski carpets}\label{planarcase} We assume that up to conjugation of $\{f_i\}_{i\in\mathcal I}$ by an affine map, there are families $(a_{i})_{i\in\mathcal I}\in (0,1)^{\mathcal I}$, $(b_i)_{i\in\mathcal I}\in (0,1)^{\mathcal I}$ and $({t}_{i})_{i\in\mathcal I}\in (\mathbb R_+^2)^{\mathcal I}$ such that  for all $i\in\mathcal I$, 
$
f_{i} : x\in\mathbb{R}^2\mapsto \mathrm{diag}(a_i,b_i) x+t_i$ and $f_i([0,1]^2)\subset [0,1]^2$.

Recall that the attractor $K$ of $\{f_i\}_{i\in\mathcal I}$ is then called a \textit{Bara\'nski carpet} if the sets $f_i((0,1)^2)$, $i\in\mathcal I$, are pairwise disjoint sub-rectangles of $(0,1)^2$, and for each of the principal axes, for all $(i,j)\in\mathcal I^2$, the orthogonal projections of $f_i((0,1)^2)$ and $f_j((0,1)^2)$ on this axis are either disjoint or equal intervals.  It is a \textit{Gatzouras-Lalley carpet} if, up to a conjugation of $\{f_i\}_{i\in\mathcal I}$ by the symmetry with respect to the first bisector,  the sets $f_i((0,1)^2)$, $i\in\mathcal I$, are pairwise disjoint sub-rectangles of $(0,1)^2$, stretched in the horizontal direction ($b_i<a_i)$, and for all $(i,j)\in\mathcal I^2$, the orthogonal projections of $f_i((0,1)^2)$ and $f_j((0,1)^2)$ on the first principal axis are either disjoint or equal intervals. When  there are integers $m_1,m_2\ge 2$ such that the $f_i((0,1)^2)$  take the form $(\frac{k_i}{m_1},\frac{k_i+1}{m_1})\times (\frac{\ell_i}{m_2},\frac{\ell_i+1}{m_2})$ for some $(k_i,\ell_i)\in\mathbb{N}^2$ and are pairwise disjoint, $K$ is a \textit{Sierpi\'nski} carpet.

Gathering Gatzouras-Lalley and Bara\'nski results, which generalise those by Bedford~\cite{ref3} and McMullen~\cite{ref12} for Sierpi\'nski carpets, one has the following variational principle. 
\begin{theorem}[{\cite[Theorem 5.3]{ref10},\cite[Theorem A]{ref1}}]\label{thmHDBGL} If $K$ is a Gatzouras-Lalley or a Bara\'nski carpet, then 
$$
\dim_H K=\max\left\{\dim_H(\mu):\, \text{$\mu$ is a self-affine measure supported on $K$}\right \}.
$$
\end{theorem}
If $K$ is a Sierpi\'nski carpet, the maximum is uniquely attained \cite{ref18} (also there is a closed-form expression for $\dim_HK$ \cite{ref3,ref12}), but it may not be the case otherwise~\cite{BFNL}.

Let us now describe the randomization of the previous models considered in this paper.

\noindent
\textbf{Random statistically self-affine Bara\'nski and Gatzouras-Lalley carpets.} Let $\mathbb{N}^+$ denote the set of positive integers. Denote by $\mathcal I^*=\bigcup_{n\ge 0}\mathcal {I}^n$, the set of finite words over the alphabet~$\mathcal {I}$; $\mathcal {I}^0$ contains the empty word denoted by $\boldsymbol{\epsilon}$. The set $\mathcal I^*$ and  the symbolic space $\mathcal{I}^{\mathbb{N}^+}$ made of the infinite words over $\mathcal I$ will be also denoted by $\Sigma^*$ and $\Sigma$ respectively. The concatenation of a finite word $u\in\mathcal I^*$ with a finite or infinite word $v\in\mathcal I^*\cup \mathcal I^{\mathbb{N}^+}$ is denoted by $u\cdot v$.  For each $w\in\mathcal{I}^*$, denote by $[w]$ the cylinder generated by $w$, that is the set of infinite words over $\mathcal I$ having $w$ as prefix; also denote by $|w|$ the length of $w\in\mathcal I^*\cup \mathcal{I}^{\mathbb{N}^+}$. If $\boldsymbol{i}\in\Sigma$ and $n\in\mathbb{N}^+$, $\boldsymbol{i}_{|n}=\boldsymbol{i}_1\cdots \boldsymbol{i}_n$ and $\boldsymbol{i}_{|0}=\boldsymbol{\epsilon}$. The set $\Sigma$ is endowed with the $\sigma$-algebra $\mathcal C$ generated by the cylinders, which is also the Borel $\sigma$-algebra  once  $\Sigma$ has been endowed with the standard distance $\mathrm{d}(\boldsymbol{i},\boldsymbol{i'})=\exp(-|\boldsymbol{i}\land\boldsymbol{i'}|)$, where $\boldsymbol{i}\land\boldsymbol{i'}$ is the longest common prefix of $\boldsymbol{i}$ and $\boldsymbol{i'}$. The shift operation on $\Sigma$ is denoted by $T$.

\noindent
\textbf{\textit{Construction of the random attractor and Mandelbrot measures.}} Fix a Gatzouras-Lalley or a Bara\'nski carpet $K$ as defined above.  Consider a random subset $\mathcal I_\omega$ of $\mathcal{I}$ such that $\mathbb{E}(\# \mathcal{I}_{\omega} )>1$. This is equivalent to considering $C = (c_{i})_{i \in \mathcal{I}}$, a random vector taking values in $\left\{0,1\right\}^{\mathcal{I}}$ such that $\mathbb{E} \left(\sum_{i\in\mathcal{I}} c_i\right)> 1$ and to setting $\mathcal I_\omega=\{i\in\mathcal{I} :\, c_i(\omega)=1\}$. Without loss of generality we assume that $\mathbb{P}(c_i=1)>0$ for all $i\in\mathcal I$, for otherwise we could reduce $\mathcal I$. 

We  are going to construct a random carpet $K_\omega\subset K$ as the image by $\pi$ of the boundary $\Sigma_\omega$ of a non degenerate Galton-Watson tree included in $\Sigma^*$. This will follow a fractal percolation process, or random curdling, according to Mandelbrot procedure \cite{ref25,Mand-Orsay} (see also \cite{GMW,DekGri,Fal,RamSi3,FJ,ref2,SimonOrgov} for studies of geometric and topological properties of statistically self-similar sets obtained by percolation on self-similar sets, and their projections).  The Hausdorff dimension of these sets will be studied using  the pushforward  by $\pi$ on $K_\omega$ of so-called Mandelbrot measures supported on $\Sigma_\omega$. 
To get such a Mandelbrot measure consider, simultaneously with $C$,  a random vector $W = (W_{i})_{i \in \mathcal{I}}$ taking values in  $\mathbb{R}_+^\mathcal{I}$ and satisfying the following properties: 
\begin{align*}
\mathbb{E} \Big (\sum_{i \in \mathcal{I}} W_{i}\Big )= 1,\ \mathbb{P} \Big(\sum\limits_{i \neq i'} W_{i} W_{i'} = 0\Big ) < 1,\ 
\left\{W_{i} > 0\right\} \subset \left\{c_{i} = 1\right\} \text{ a.s. }\forall \, i \in \mathcal{I}.
\end{align*}
The first property guaranties a mass conservation  in the mean in the process to follow, the second one ensures that the limit measure is not a Dirac mass, while the third one ensures that its topological support is included in $\Sigma_\omega$. 

Let  $(C(v),W(v))_{v\in \mathcal{I}^*}$ be a sequence of independent copies of $(C,W)$ and  $(\Omega,\mathcal F,\mathbb P)$ the probability space over which these random variables are defined, and simply denote $(C(\boldsymbol{\epsilon}),W(\boldsymbol{\epsilon}))$ by $(C,W)$. In particular, almost surely, for all $v\in  \mathcal{I}^*$ and $i \in \mathcal{I}$, one has 
$
\left\{W_{i}(v) > 0 \right\} \subset \left\{c_{i}(v) = 1 \right\}.
$
For all $\omega\in\Omega$ and $n\ge 0$  set 
\begin{align}\label{Sigman}
\Sigma_{\omega,n} &= \left\{\boldsymbol{i} \in \Sigma: \ c_{i_m}(\boldsymbol{i}_{|m-1}) (\omega)= 1 \ \text{for all }1\le m \le n\right\},\\
\nonumber \text{and }\quad 
\Sigma_\omega &= \bigcap\limits_{n \geq 0} \Sigma_{\omega,n}.
\end{align}
Classical properties of Galton-Watson processes show that under our assumptions $\mathbb{E}(\# \mathcal I_\omega)>1$,  the set $\Sigma_\omega$ is the boundary of a supercritical Galton-Watson tree, so that  $\mathbb P(\Sigma_\omega \neq \emptyset) > 0$. Set 
\begin{equation*}\label{Kn}
K_\omega = \pi(\Sigma_\omega)=\bigcap_{n\ge 0} K_{\omega,n}, \text{ where }K_{\omega,n} = \pi(\Sigma_{\omega,n}).
\end{equation*}

Now we define the Mandelbrot measure associated with $(W(v))_{v\in \Sigma^*}$. For $v \in \mathcal I^*$, $n \geq 0$ and $w=i_1\cdots i_n \in \mathcal{I}^n$, define $Q^{v}(w)=1$ if $n=0$ and 
\begin{equation}\label{Qvw}
Q^{v}(w)=W_{i_1}(v)W_{i_2}(v\cdot i_1) \cdots W_{i_n}(v\cdot i_1 \cdots i_{n-1})
\end{equation}
otherwise. We simply denote $Q^{\boldsymbol{\epsilon}}(w)$ by $Q(w)$, and set 
$$
Y_n(v) = \sum_{w\in \mathcal{I}^n} Q^{v}(w). 
$$

\noindent The sequence  $(Y_n(v),\sigma(W_{i}(v w) : i \in \mathcal{I}, \ w \in \bigcup_{k=0}^{n-1}\mathcal{I}^{k}))_{n \geq 0}$ is a non negative  martingale. Denote by  $Y(v)$ its almost sure limit. Since $\mathcal{I}^*$ is countable, the random variables $Y(v)$, $v\in \mathcal I^*$, are almost surely defined simultaneously. Moreover, they obey the recursion relation $Y(v)=\sum_{i\in\mathcal I}W_i(v)Y(vi)$, so that one can define almost surely over the cylinders of $\Sigma$ the mapping 
$$
\nu_\omega: [v]\mapsto Q(v)Y(v),
$$
which extends uniquely to a measure on $(\Sigma,\mathcal{B}(\Sigma))$, still denoted by $\nu_\omega$, or simply~$\nu$ when there is no ambiguity. This measure is almost surely the weak limit  of the sequence $(\nu_n)_{n \geq 0}$ defined by uniformly distributing (with respect to the uniform measure on $(\Sigma,\mathcal{B}(\Sigma))$), the mass $Q(w)$ over each cylinder $\left[w\right]$, $w \in \mathcal{I}^n$. By construction, the topological support of $\nu$ is included in $\Sigma_\omega$. Also, the random variables $Y(v)$, $v\in \mathcal I^*$, are identically distributed. Denote $Y(\boldsymbol{\epsilon})=\|\nu\|$ by $Y$. 

\noindent
\textbf{\textit{Non degeneracy.}} The measure~$\nu$ is not necessarily non degenerate, that is positive with positive probability. Let
\begin{equation}\label{phiW}
\phi_W : q \geq 0 \longmapsto \mathbb{E} \Big (\sum_{i \in \mathcal{I}} W_i^q\Big ) \text{ and }T_W=- \log \phi_W.
\end{equation}
$T_W$ is finite, continuous and concave  over $\left[0,1\right]$. Set
\begin{equation}\label{HW}
H(W)=T_W'(1^-) =-\phi_W'(1^-)= - \sum_{i \in \mathcal{I}} \mathbb{E}\left(W_i \log(W_i)\right).
\end{equation}
\begin{theorem}[\cite{ref37, ref26}]\label{thmKP}
The following assertions are equivalent :

(1) $\nu$ is not degenerate (i.e. $\mathbb P(\nu\neq 0)>0$); (2) $\mathbb{E}\left(Y\right)= 1$; (3) $H(W) > 0$.
\end{theorem}
It is not hard to prove that conditional on $\{\nu_.\neq 0\}$, $\mathrm{supp}(\nu_\omega)$ is almost surely equal to the set of those points $\boldsymbol{i}=i_1i_2\cdots$ of  $\Sigma$ such that $W_{i_n}(i_1\cdots i_{n-1}) > 0$ for all $n \geq 1$ (see \cite{ref2}). Also, $\supp(\nu_\omega) = \Sigma_\omega$ almost surely, if and only if $\mathbb{P}(c_{i} = 1) = \mathbb{P}(W_{i}>0)$ for all $i \in \mathcal{I}$. 

\noindent
\textbf{\textit{Symbolic Hausdorff dimension and entropy dimension of $\nu$.}} It will be interesting in our study  to consider the probability vector $p=(p_i)_{i\in\mathcal I}=\mathbb{E}(W)$ and define $\widetilde W_i=W_i/p_i$ if $p_i>0$ and $\widetilde W_i=1$ otherwise. Then, recalling that the entropy of $p$ is defined as $-\sum_{i\in\mathcal I}p_i\log(p_i)$, the quantity $H(W)$ satisfies 
\begin{equation}\label{HW}
H(W)=h(p)-\sum_{i\in\mathcal I}p_i\mathbb{E}(\widetilde W_i\log(\widetilde W_i))\le h(p)\le \log (\#\mathcal I),
\end{equation}
where $\mathbb{E}(\widetilde W_i\log(\widetilde W_i))\ge 0$ since $\mathbb{E}(\widetilde W_i)=1$ and $x\ge 0\mapsto x\log x$ is convex; also, the first inequality is strict except if $W=p$ a.s. When $H(W)>0$, Kahane and Peyri\`ere~\cite{ref26,K87} showed that conditional on $\nu\neq 0$, $\dim_H(\nu)=\dim_P(\nu)=H(W)$ ($\Sigma$ being endowed with the standard distance $\mathrm{d}$). This implies \cite{Heurteaux1998} that 
$$
\lim_{n\to +\infty} -\frac{1}{n}\sum_{w\in\mathcal I^n}\nu([w])\log(\nu([w]))=H(W),
$$
that is $H(W)$ is also the entropy dimension  $\dim_e(\nu)$ of $\nu$.

The measure $\pi_*\nu_\omega$ will be denoted by $\mu_\omega$ and called a Mandelbrot measure on $K_\omega$. Our first result is the following extension of Theorem~\ref{thmHDBGL} (the case of random Sierpi\'nski carpets was established in \cite{ref6}, in which case the supremum in \eqref{dimK2'} below is uniquely attained; the value of $\dim_HK_\omega$ had been obtained in~\cite{ref9}, as well as in~\cite{BenNasr} for special cases). 
\begin{theorem}\label{dimK2}
With probability~1, conditional on  $\{K_\cdot \neq \emptyset\}$,
\begin{equation}\label{dimK2'}
\dim_H(K_\omega) = \max\left\{{\dim}_H(\mu_\omega): \ \mu_\omega  \text{ is a  Mandelbrot measure supported on} \ K_\omega\right\}.
\end{equation}
\end{theorem}

\subsection{The higher dimensional case}\label{sechigherdim} We work in $\mathbb{R}^d$ ($d\ge 2$) and seek for an extension, in the random setting, of Das and Simmons~\cite{ref5} study of the Hausdorff dimension of a class of sponges which contains higher dimensional versions of self-affine Bara\'nski and Gatzouras-Lalley carpets. 

\noindent
\textbf{\textit{``Good'' sponges}} (\cite{ref5})\textbf{.} We assume that for each $i\in\mathcal{I}$, the linear part $A_i$ of $f_i$ is a diagonal matrix $\mathrm{diag}(a_{i,1},\ldots, a_{i,d})$ with $0<|a_{i,k}|<1$ for all $1\le k\le d$, 
and without loss of generality we assume that $f_i([0,1]^d)\subset [0,1]^d$ for all $i\in\mathcal{I}$. If $D\subset\{1,\ldots,d\}$ is non-empty, denote by $\pi^D$ the orthogonal projection from $\mathbb{R}^d$ to the subspace $\mathbb{R}^{D}$ generated by the coordinate axes indexed by the elements of $D$. 

Denoting by $P_{\mathcal I}$ the set of probability vectors $(p_i)_{i\in\mathcal I}$, for each ${p}\in P_{\mathcal I}$ and $1\le k\le d$, consider the Lyapunov exponent associated to $p$ in direction $k$, that is  
\begin{equation}\label{chip}
\chi_k({p})=-\sum_{i\in\mathcal I} p_i\log(|a_{i,k}|).
\end{equation} 
\begin{definition}\label{Goods}
According to \cite{ref5}, say that the attractor $K$ of the IFS $\{f_i\}_{i\in\mathcal I}$ is a \textit{good} sponge if, for each $p\in P_{\mathcal{I}}$ and $x\in\mathbb R_+$,  setting $D=D(p,x)=\{1\le k\le d: \chi_k(p)\le x\}$, for all $i,j\in \mathcal I$, either $f_i$ and $f_j$ overlap exactly on $\mathbb R^D$, that is $\pi^D\circ {f_i}_{|[0,1]^d}= \pi^D \circ  {f_j}_{|[0,1]^d}$, or $\pi^D\circ {f_i}((0,1)^d)\cap \pi^D\circ  {f_j}((0,1)^d)=\emptyset$.
\end{definition}

The class of good sponges  is a little more general than that of the sponges obeying the  \textit{separation of principal projections condition} (SPPC) considered by Fraser and Kolossv\'ary~\cite{Fra-Kol} and Kolossv\'ary~\cite{Kol} for the study of the Assouad and lower dimensions of the associated self-affine measures, as well as their $L^q$-spectrum. To get sponges satisfying the SPPC, in  Definition~\ref{Goods} one should require  in addition that the alternative between exact overlapping and disjointness holds for the orthogonal projections on all the spaces $\mathbb{R}^{D'}$ with $\emptyset \neq D'\subset D$. This prevents certain configurations where in restriction to some subspaces of dimension $\ge 2$ generated by principal axes the linear parts $A_i$ are similarities (in particular SPPC excludes many self-similar sets obeying the OSC).  However SPPC covers many natural examples, starting with Bara\'nski and Gatzouras-Lalley carpets and their higher dimensional versions. Gatzouras-Lalley sponges correspond to the case where there exists a permutation $\sigma\in \mathfrak{S}_d$ such that $|a_{i,\sigma_{k+1}}|<|a_{i,\sigma_{k}}|$ for all $i\in \mathcal I$ and $1\le k\le d-1$, and for all $1\le k\le d$, setting $D_k=\{\sigma_k,\ldots,\sigma_d\}$, for all $i,j\in\mathcal{I}$, either $f_i$ and $f_j$ overlap exactly on $\mathbb R^{D_k}$, or $\pi^{D_k}\circ {f_i}((0,1)^d)\cap \pi^{D_k}\circ  {f_j}((0,1)^d)=\emptyset$. Bara\'nski sponges correspond to the situation where for all $1\le k\le d$, for all $i,j\in\mathcal{I}$, either $f_i$ and $f_j$ overlap exactly on $\mathbb R^{\{k\}}$, or $\pi^{\{k\}}\circ {f_i}((0,1)^d)\cap \pi^{\{k\}}\circ  {f_j}((0,1)^d)=\emptyset$ (note that when $d=2$ both previous classes are slightly more general than in Section~\ref{planarcase}). Sierpi\'nski sponges are Bara\'nski sponges for which there are integers $m_1,\ldots,m_d\ge 2$ such that the linear parts of the $f_i$, $i\in\mathcal I$, are all equal to $\mathrm{diag}(m_1^{-1},\ldots,m_d^{-1})$ and their translation vector parts belong to $\prod_{k=1}^d m_k^{-1}\{0,\ldots, m_k-1\}$. Also, when $d\ge 3$, the class of sponges satisfying SPPC strictly contains the previous ones~\cite{Fra-Kol}.

\noindent
\textbf{\textit{The associated random attractor and  inhomogeneous Mandelbrot measures.}} Fix a good sponge $K$ as above.  As in dimension~$2$,  consider  a random vector $C=(c_i)_{i\in\mathcal{I}}\in\{0,1\}^{\mathcal I}$  such that $\mathbb{E}(\sum_{i\in\mathcal I}c_i)>1$ and $\mathbb{P}(c_i=1)>0$ for all $i\in\mathcal I$. Also,  consider a sequence $((C^{(n)},W^{(n)}))_{n\ge 1}$ of random vectors such that, for each $n\ge 1$, $C^{(n)}$ is distributed like $C$ and the random vector $W^{(n)}=(W^{(n)}_{i})_{i\in \mathcal I}\in\mathbb{R}_+^{\mathcal I}$ satisfies 
\begin{align*}
\mathbb{E} \Big (\sum_{i \in \mathcal{I}} W^{(n)}_{i}\Big )= 1\text{ and }
\left\{W^{(n)}_{i} > 0\right\} \subset \left\{c^{(n)}_{i} = 1\right\} \text{ a.s. }\forall \, i \in \mathcal{I}.
\end{align*}
Let $\big ((C(v),W(v))\big )_{v\in\Sigma^*}$ be a sequence of independent random vectors, such that for all $n\ge 1$ and $v\in\mathcal I^{n-1}$, $(C(v),W(v))$ is distributed like $(C^{(n)},W^{(n)})$. We also denote $(C(v),W(v))$ by $(C^{(n)}(v),W^{(n)}(v))$ when $v\in \mathcal I^{n-1}$. 

Then define $\Sigma_\omega$, $K_\omega$, $\nu_\omega$ and $\mu_\omega=\pi_*\nu_\omega$ exactly in the same way as in dimension 2. The measures $\nu_\omega$ and $\mu_\omega$ are called inhomogeneous Mandelbrot measures (IMM). Note that Mandelbrot measures (MMs) are IMMs, but this should not create any confusion.

\noindent
\textbf{\textit{Non degeneracy.}} One has the following sufficient condition for non degeneracy of $\nu$.

\begin{theorem}\label{sufcond}
If $
\liminf_{N\to +\infty} \frac{1}{N}\sum_{n=1}^N H(W^{(n)})>0\text{  and }\sum_{n\ge 1}\frac{\phi''_{W^{(n)}}(1^-)}{n^2}<+\infty
$,
then $\mathbb{E}(\|\nu\|)=1$.   Moreover, if for all $v\in \mathcal{I}^*$ one has $\mathbb{P}(W_i(v)>0)=\mathbb{P}(c_i(v)=1)$, then conditional on $K_\omega\neq\emptyset$, $\supp(\mu_\omega)=K_\omega$. 
\end{theorem}

\noindent
\textbf{\textit{Hausdorff dimension of $K_\omega$.}} When the components of $C^{(n)}$ are all equal to 1 and the $W^{(n)}$, ${n\ge 1}$, are deterministic, the limiting measure $\mu$ is a deterministic inhomogeneous Bernoulli  measure supported on $K$.  When, moreover, $d\ge 3$,  Das and Simmons exhibited examples for which $(W^{(n)})_{n\ge 1}$ is the restriction to $\mathbb{N}^+$ of a continuous function $(W^{(t)})_{t>0}$ such that $u\in\mathbb{R}\mapsto \mathbb{P}_{W^{(\exp(u))}}$ is periodic, and $\dim_H(\mu)>\sup\{\dim_H (\pi_*\rho):\, \rho\text{ is }T\text{-invariant}\}=\sup\{\dim_H(\rho):\, \rho\text{ is self-affine and $\supp(\rho)\subset K$}\}$, thus showing that a dimensional gap between the dynamical and Hausdorff dimensions of $K$ can occur. As a value for  $\dim_H K$ they proposed the supremum of the Hausdorff dimensions of such exponentially periodic Bernoulli measures supported on $K$. However, the proof of this variational principle presents a gap (\cite{ref5} p. 112, between the second and third term of the series of seven equalities and equivalents; personal communication with the authors), and whether it holds true or not remains an open question (see Remark~\ref{remdimH}). 

We will establish an alternative variational principle. Let
\begin{equation}\label{L}
\mathscr L=\left\{ (\ell_m)_{m\ge 1}\in {(\mathbb{N}^+)}^{\mathbb{N}^+}:\, \begin{cases}\ell_{m-1}<\ell_{m},\ \forall\, m\ge 1 \\ \ell_m=o(L_{m-1}=\ell_1+\cdots+\ell_{m-1}) \text{ as $m\to+\infty$}
\end{cases}\right\}.
\end{equation}
If $\ell\in\mathscr L$, say that an inhomogenous Mandelbrot measure  is of type $\ell$ if for all $m\ge 1$, all the $W^{(n)}$, $L_{m-1}< n\le L_m$, have the same law.

\begin{theorem} \label{dimKd}Let  $\ell\in\mathscr L$. 
With probability~1, conditional on $\{K_\cdot\neq \emptyset\}$,  one has 
\begin{equation*}\label{RDS0}
\dim_H K_\omega=\sup\{\dim_H(\mu_\omega): \, \text{$\mu_\omega$ is a non degenerate  IMM of type $\ell$ supported on $K_\omega$}\}.
\end{equation*}
\end{theorem}
We do not know whether the supremum in Theorem~\ref{dimKd} is attained in general. The possibility of a dimension gap established by Das and Simmons in the deterministic case for $d\ge 3$ naturally persists in the random case, in the sense that in general the supremum of the Hausdorff dimensions of Mandelbrot measures supported on $K_\omega$ is strictly smaller than that associated to inhomogeneous ones. This can be seen by considering a random perturbation of Das and Simmons example (see Section~\ref{epex}).

We also have the following extension of the result obtained in~\cite{ref6} for random Sierpi\'nski sponges. 

\begin{theorem}\label{dimKd''} If the linear parts of the affine maps $f_i$ are equal, with probability~1, conditional on $\{K_\cdot\neq \emptyset\}$, one has  
\begin{equation*}
\dim_H K_\omega=\max\{\dim_H(\mu_\omega): \, \text{$\mu_\omega$ is a non degenerate MM supported on $K_\omega$}\}.
\end{equation*} 
Moreover, the maximum is attained at a unique Mandelbrot measure. 
\end{theorem}

Below we describe our approach to get  the previous results. 

The variational principle established in Theorem~\ref{dimKd} relies on having sufficiently precise information about the Hausdorff dimensions of IMMs. To this end, we prove a general result—Theorem~\ref{thm-2.4}(2)—which provides both the Hausdorff and packing dimensions for a broad class of non-degenerate IMMs. This result is of independent interest. Since its precise formulation requires additional notation, we defer its full statement to Section~\ref{DM}. Nevertheless, we outline here the approach used to study these dimensions and contrast it with the method used in the deterministic case. In the latter case, the Hausdorff and packing dimensions of an inhomogeneous Bernoulli measure (IBM) $\mu$ on $K$ associated to a sequence of probability vectors $(p^{(n)})_{n\ge 1}$ can be obtained by studying $\mu$-almost everywhere the fluctuations of $\frac{\log(\mu(B(x,r)))}{\log(r)}$ as $r\to 0$ by $(i)$ replacing balls by sequences $(Q_N(z))_{N\ge 1}$ of almost cubes suitably chosen according to the behavior of the Lyapunov exponents associated with $\{A_i\}_{i\in\mathcal I}$ and $\nu$ and with sides comparable to the scale  $e^{-N}$ for large $N$, that is the collection $(\chi_k(N_k))_{1\le k\le d}$, where $\chi_k(n)=-\frac{1}{n} \int_\Sigma \sum_{j=1}^n \log(|a_{\boldsymbol{i}_j,k}|)\, \mathrm{d}\nu(\boldsymbol{i})$ and $N_k\chi_k(N_k)\sim N$; $(ii)$ exploiting the multiplicative structure of IBMs  and their orthogonal projections to principal subspaces (they are IBMs as well) to decompose the  logarithms of these masses as finitely many sums of independent random variables  to which applies  the strong law of large numbers for non identically distributed independent random variables (see~\cite{ref5} where this is done for $\dim_H(\mu)$ and $(p^{(n)})_{n\ge 1}$ being the restriction to $\mathbb{N}^+$ of an exponentially continuous and periodic function  $(p^{(t)})_{t>0}$; but the method is general). As a result, there is a sequence  $(S_N(\mu))_{N\ge 1}$ of sums of entropies of BMs and entropies of projections of BMs such that $\dim_H(\mu)= \liminf_{N\to\infty} N^{-1}S_N(\mu)$ and $\dim_P(\mu)= \limsup_{N\to\infty} N^{-1}S_N(\mu)$. In the random case, orthogonal projections on principal subspaces of an IMM $\mu=\mu_\omega$ are not IMMs in general, but  they keep multiplicative properties  in expectation. This is why a large deviation approach is substituted to the SLLN, via a fine control of the expectations of sequences of partition functions $\sum_{Q\in\mathcal F_N} \mu(Q)^q$ around the inverse temperature $q=1$ (using the terminology of thermodynamics), where $\mathcal F_N$ is a collection of parallepipeds  (with pairwise disjoint interiors) which  form a covering of  $K_\omega$,  and is determined by the Lyapunov exponents of~$\nu$  associated to successive scales $e^{-N}$ (now $p^{(n)}=\mathbb{E}(W^{(n)})$); the elements of $\mathcal F_N$ are far from being  all almost cubes, while it is the case with the so-called $L^q$-spectrum which is enough to tackle the case of  MM on Sierpi\'nski sponges. However, $\mu$-almost every point is asymptotically contained in an element of $\mathcal F_N$ which is  an almost cube; this makes it possible to get the desired dimensions from the asymptotic behavior of the functions $q\mapsto \mathbb{E} \big (\sum_{Q\in\mathcal F_N} \mu(Q)^q\big)$ near $1$  and concentration inequalities. This asymp\-totic behavior results from calculations which go far beyond those conducted in~\cite{ref6} to control the $L^q$-spectrum of MM on random Sierpi\'nski sponges, and which include new estimates for the $L^q$ norm of inhomogeneous Mandelbrot martingales taking into account the possible occurrence of many levels~$n$ in the cascade of multiplications defining $\nu$ and $\mu$, for which $H(W^{(n)})<0$. It turns out that as the scale $e^{-N}$ goes to 0, $\mathbb{E} \big (\sum_{Q\in\mathcal F_N} \mu(Q)^q\big)$  behaves as $O\big (e^{-(q-1)S_N(\mu)(1+o(1))}\big )$ as $q\to 1$, and this time $S_N(\mu)$ is the minimum of about $\left\lceil\big (\frac{1}{\chi_{\min}(\nu,N))}-\frac{1}{\chi_{\max}(\nu,N)}\big )N\right\rceil$ distinct sums of entropy dimensions of MMs and entropies of dimensions of projections of BMs,  where $\chi_{\min}(\nu,N)$ and $\chi_{\max}(\nu,N)$ are respectively the smallest and the biggest element  of $\{\chi_k(N_k)\}_{1\le k\le d}$.  Again,  setting $d_N(\mu)=N^{-1}S_N(\mu)$, one has $\dim_H(\mu)$ and $\dim_P(\mu)$ equal to $\liminf_{N\to\infty} d_N(\mu)$ and $\limsup_{N\to\infty} d_N(\mu)$ respectively. 

To get Theorem~\ref{dimKd}, we apply Theorem~\ref{thm-2.4}(2) to the subclass of IMMs $\mu_{\boldsymbol{p}}$ of type $\ell$ such that for all $n\ge 1$, $W^{(n)}$ is distributed like $\big (p_i^{(n)}\frac{\mathbf{1}_{\{c_i=1\}}}{\mathbb{P}(c_i=1)}\big )_{i\in\mathcal I}$, and  $\boldsymbol{p}=((p^{(n)})_{i\in\mathcal I})_{n\in\mathbb{N}^+}$ is a sequence of positive probability vectors defining a Bernoulli product measure of type $\ell$ fully supported on $\Sigma$; such a measure is almost surely fully supported on $\Sigma_\omega$ conditional on $\{\Sigma_\omega\neq\emptyset\}$. The validity of the variational principle  follows by proving that $\dim_H K_\omega$ is upper bounded by the supremum of the Hausdorff dimensions of these measures $\mu_{\boldsymbol{p}}$. To do so, as for random statistically self-affine Sierpi\'nski carpets or sponges, we  need to exhibit adapted coverings, in the spirit of the original Bedford's approach  \cite{ref3} to the Hausdorff dimension of Sierpi\'nski carpets, further developed for random Sierpi\'nski carpets and sponges in \cite{ref9,ref6}. It is where, instead of using the usual notion of frequency of digits on the coding space as in the aforementioned studies,  we use for each $\boldsymbol{i}\in\Sigma$  the sequence $(p(\boldsymbol{i},m))_{m\ge 1}$ of localized frequencies of digits obtained when one  considers, for all $m\in\mathbb{N}^+$,  the vector $p(\boldsymbol{i},m)$ of  the frequencies of the digits  $i$ of $\mathcal {I}$ in the finite subword $\boldsymbol{i}_{L_{m-1}+1}\cdots \boldsymbol{i}_{L_{m}}$ of length $\ell_m$. We first provide a new proof of the sharp upper bound for $\dim_H (\mu)$ when $\mu$ is a non degenerate IMM of type $\ell$, by using suitable collections of coverings. This  exploits the fact that one can control very well the asymptotic behavior of the localized frequencies for $\nu$-almost every $\boldsymbol{i}\in\Sigma_\omega$: $\big\|p(\boldsymbol{i},m)-\mathbb{E}(W^{(L_m)})\|_\infty$ converges to 0 as $m\to+\infty$. These coverings are made of collections of almost cubes of side lengths about $e^{-N}$, whose expected number is estimated from above in about $\left\lceil\big (\frac{1}{\chi_{\min}(\nu,N))}-\frac{1}{\chi_{\max}(\nu,N)}\big )N\right\rceil$ manners, and so by the infimum of the resulting estimates. This is where  the connection with $d_N(\mu)$ is made. Then, for each $\epsilon>0$, one selects a suitable subset $\mathscr P_\epsilon$ of sequences $\boldsymbol{p}=(p^{(n)}_i)_{i\in\mathcal I})_{n\in\mathbb{N}^+}$ of positive probability vectors, allowing most sequences $(p(\boldsymbol{i},m))_{m\ge 1}$ (for $\boldsymbol{i}\in\Sigma$) to be approximated,  up to~$\epsilon$, by some $(p^{(L_m)})_{m\ge 1}$ with $\boldsymbol{p}\in  \mathscr P_\epsilon$. Subsequently, for all $\epsilon>0$ and $\boldsymbol{p}\in \mathscr P_\epsilon$ one can construct suitable coverings—closely related to those previously used to estimate $\dim_H(\mu_{\boldsymbol{p}})$ from above—to cover the set of  points in $K_\omega$ that are  images under the coding map $\pi$ of points $\boldsymbol{i}\in\Sigma_\omega$ for which $(p(\boldsymbol{i},m))_{m\ge 1}$ is $\epsilon$-close to $(p^{(L_m)})_{m\ge 1}$.  This implies that this set has a Hausdorff dimension smaller than $\dim_H(\mu_{\boldsymbol{p}})+\eta(\epsilon)$, with $\lim_{\epsilon\to 0} \eta(\epsilon)=0$.  Finally, the  fact that $\mathscr P_\epsilon$ can be taken as an infinite product of finite sets whose cardinalities grow slowly (due to the property $\ell_m=o(L_{m-1})$), combined with the existence of $C(\epsilon)>0$ independent of $\boldsymbol{p}$ such that $d_N(\mu_{\boldsymbol{p}})$ depends on the $\lfloor C(\epsilon) N\rfloor $ first vectors $p^{(n)}$ only, yields the existence of some $K^\epsilon_\omega\subset K_\omega$ such that $\dim_HK^\epsilon_\omega\le \sup_{\boldsymbol{p}\in\mathscr P_\epsilon}\dim_H(\mu_{\boldsymbol{p}})+\eta(\epsilon)$ and $\lim_{\epsilon\to 0} \dim_H( K_\omega\setminus K^\epsilon_\omega)=0$. 
 
Theorem~\ref{dimK2} will be seen as a consequence of Theorem~\ref{dimKd}. Theorem~\ref{dimKd''} can be obtained as rather a direct generalisation of the already known results when $K_\omega$ is a random Sierpi\'nski sponge (\cite{ref6}). However, apart from the uniqueness of the MM of maximal Hausdorff dimension, we will show how  to naturally derive  the result from Theorem~\ref{dimKd}. 

\begin{remark}\label{remdimH}
 It remains open whether or not  it is possible to  use non degenerate  IMMs of exponentially continuous and periodic type (see Section~\ref{epex} for a precise definition) in the variational principle. The fact that continuity plus periodicity implies uniform continuity makes the restriction to the positive  integers of the associated process $(W^{(t)})_{t>0}$ easy to arbitrary approximate uniformly near $+\infty$ by some $({W'}^{(n)})_{n\ge 1}$ defining a non degenerate IMM of type~$\ell$, but it is the converse, possibly in a weaker sense,  and even if one replaces continuity and periodicity by the weaker uniform continuity, which is missing. 
\end{remark}

\noindent
\textbf{\textit{Packing dimension of $K_\omega$.}} Note that by statistical self-affinity,  $\dim_P K_\omega=\overline \dim_B K_\omega$ almost surely.  The following variational principle holds for the packing dimension of $K_\omega$. 
\begin{theorem} \label{dimPKd}
Fix $\ell\in\mathscr L$. With probability~1, conditional on $\{K_\cdot\neq \emptyset\}$,  one has 
\begin{align*}
\dim_P K_\omega =\sup\{\dim_P(\mu_\omega): \, \text{$\mu_\omega$ is a non degenerate  IMM of type $\ell$ supported on $K_\omega$}\}.
\end{align*}
Moreover, the supremum is attained in the deterministic case. 
\end{theorem}
The existence of $\dim_BK_\omega$ is known to hold in the deterministic case \cite{ref3,ref12,ref10,ref18,FengWang, ref1, Kol}, as well as for random Sierpi\'nski carpets and sponges~\cite{ref9,ref6}. This dimension is then expressed as a weighted sum of entropies of Bernoulli or Mandelbrot measures supported on~$K$ and entropies  of projections of such measures, or the supremum of such sums, and in general it is strictly larger than the Hausdorff dimension. In the deterministic case, it is possible to exploit the approximations used to establish that $\dim_B K$ exists to get Theorem~\ref{dimPKd}. That $\dim_B K_\omega$ does exist in the general random case will be shown in a separate paper and requires to substantially modify the approach used for random Sierpi\'nski carpets and sponges~\cite{ref9,ref6}. For the time being, we find it interesting to give a direct proof combining the formula for the packing dimension of IMMs  with covering numbers estimates obtained in the study of $\dim_HK_\omega$.

The paper is organized as follows. Section~\ref{DM} provides an extension of Theorem~\ref{sufcond} (Theorem~\ref{sufcond2}) and presents our results on the Hausdorff and packing dimensions of non degenerate IMMs, with some attention to the case of IMMs of exponentially continuous and periodic type in order to extend to the random case the dimensional gap property  detected by Das and Simmons in the deterministic case. Section~\ref{ND} provides the proof of Theorem~\ref{sufcond2}, as well as some controls of moments  for inhomogeneous Mandelbrot martingales. Section~\ref{proofsdimmu} is dedicated to the proof of the results of Section~\ref{DM}, Section~\ref{UB} to the proof of Theorem~\ref{dimKd}, while Section~\ref{secdim2} contains the proofs of Theorem~\ref{dimK2} and~\ref{dimKd''}, and Section~\ref{secdimPK} that of Theorem~\ref{dimPKd}. Section~\ref{appendix} provides the definitions of Hausdorff and packing dimensions of a measure, as well as some general lemmas.

\section{Hausdorff and packing dimensions of inhomogeneous Mandelbrot measures supported on $K_\omega$}\label{DM}
Let $\nu$ be an IMM  constructed as in Section~\ref{introduction}. Before presenting our result on the dimensions of $\mu=\pi_*\nu$ in~Section~\ref{thmdim}, some preliminaries are required. 
\subsection{Some preliminaries}\label{preliminaries}First, we state a more general version of  Theorem~\ref{sufcond} on non degeneracy of $\nu$. Next we define some coding useful to describe the orthogonal projections of Mandelbrot measures to subspaces generated by the principal directions and associated to the behavior of the typical Lyapunov exponents of the measure along the scales seen from $\nu$. This coding is also used to define some projections of probability vectors. Finally, we present an assumption that can be made without loss of generality, in order to simplify the exposition of the material to follow.

\noindent
\textbf{\textit{Non degeneracy.}} The following result, which exploits the approach developed in~\cite{WW,Lyons} to get  Theorem~\ref{thmKP}, will be proven in Section~\ref{ND}. Note that since $H(W^{(n)})\le \log\#(\mathcal I)$ for all $n\ge 1$, for the assumptions of item (1) of the statement below to hold,  $b_n$ must be~$O(n)$. 

\begin{theorem}\label{sufcond2}Let $h\in (1,2]$ and $(b_n)_{n\ge 1}$ be an unbounded increasing positive sequence such that  
$$
\sum_{n\ge 1}\displaystyle\frac{\mathbb{E}(\sum_{i\in\mathcal I}W^{(n)}_i |\log W^{(n)}_i|^h)}{b_n^h }<+\infty.
$$
\begin{enumerate}
\item If $\liminf_{N\to +\infty} b_N^{-1}\sum_{n=1}^N H(W^{(n)})>0$, then $\mathbb{E}(\|\nu\|)=1$.   Moreover, if for all $v\in \mathcal{I}^*$ one has $\mathbb{P}(W_i(v)>0)=\mathbb{P}(c_i(v)=1)$, then conditional on $K_\omega\neq\emptyset$, $K_\omega=\supp(\mu_\omega)$; equivalently, $\supp(\mu_\omega)=K_\omega$ almost surely.
\item If $\liminf_{N\to +\infty} b_N^{-1}\sum_{n=1}^N H(W^{(n)})<0$, then $\nu=0$ almost surely. 
\end{enumerate}
\end{theorem}

\noindent
\textbf{\textit{Coding the orthogonal projections on principal subspaces.}} Recall the definition of $D(p,x)$ in Definition~\ref{Goods}. Denote by $\mathscr D=\{D(p,x):\, (p,x)\in P_{\mathcal{I}}\times \mathbb{R}_+\}$. Fix $s\in\mathbb{N}^+$ and $D=(D_r)_{1\le r\le s}\in \mathscr{D}^s$, such that  $D_1\supsetneq \cdots\supsetneq D_s\neq\emptyset$. The following definitions and notations are inspired from those used in \cite{Kol}.

For $1\le r\le s-1$, denote by $\pi^D_{r,r+1}$ the orthogonal projection from $\mathbb R^{D_{r}}$ to $\mathbb{R}^{D_{r+1}}$. Also, for $1\le r\le s$, denote $\pi^{D_r}$ by $\pi^D_r$, and set $\mathcal E^D_r=\{\pi^D_r\circ f_i((0,1)^d):\, i\in\mathcal I\}$. We endow $\mathcal I$ with any total order relation.  Set $\mathcal I^D_1=\mathcal I$. If $s\ge 2$, define recursively a non decreasing collection $ \mathcal I=\mathcal {I}^D_1\supset \cdots\supset \mathcal {I}^D_s$, as well as mappings $\Pi^D_{r,r+1}:\mathcal {I}^D_r\to \mathcal {I}^D_{r+1}$ for $1\le r\le s-1$ as follows: for each $E\in \mathcal E^D_2$, pick the smallest $j=j_E\in \mathcal I^D_1$  such that $E=\pi^D_2\circ f_j((0,1)^d)$, set $\mathcal I^D_2=\{j_E: \, E\in \mathcal E^D_2\}$, and for all $j\in \mathcal I^D_2$ and $i\in \mathcal I^D_1$ such that $\pi^D_2\circ f_j((0,1)^d)=\pi^D_2\circ f_i((0,1)^d)$, set $\Pi^D_{1,2}(i)=j$.  Suppose that $2\le r\le s$ and $\mathcal {I}^D_1\supset \cdots\supset \mathcal {I}^D_{r-1}$ have been constructed as well as $\Pi^D_{\ell,\ell+1}: \mathcal I^D_{\ell}\to \mathcal I^D_{\ell+1}$ for all $1\le \ell\le r-2$. For each $E\in \mathcal E^D_r$, pick the smallest $j=j_E\in \mathcal I^D_{r-1}$  such that $E=\pi^D_r\circ f_j((0,1)^d)$ (noting that $\pi^D_r=\pi^{D}_{r-1,r}\circ \pi^{D}_{r-1}$), set $\mathcal I^D_r=\{j_E: \, E\in \mathcal E^D_r\}$, and for all $j\in \mathcal I^D_r$ and $i\in \mathcal I^D_{r-1}$ such that $\pi^D_r\circ f_j((0,1)^d)=\pi^D_r\circ f_i((0,1)^d)$, set $\Pi^D_{r-1,r}(i)=j$. 

By construction, $\mathcal I^D_r\subset \mathcal I^D_{r-1}$ for all $2\le r\le s$, and setting  $\Pi^D_r(i)=j$ for all $j\in \mathcal I^D_r$ and $i\in \mathcal I^D_1$ such that $\pi^D_r\circ f_j((0,1)^d)=\pi^D_r\circ f_i((0,1)^d)$, one has  $\Pi^D_r=\Pi^D_{r-1,r}\circ \cdots\circ \Pi^D_{1,2}$. 

Each mapping $\Pi^D_r$ extends uniquely as a 1-block factor map from $\mathcal I^n$ to $(\mathcal I^D_r)^n$ for all $n\in \mathbb N$ and from $\Sigma=\mathcal I^{\mathbb{N}^+}$ to  $(\mathcal I^D_r)^{\mathbb{N}^+}$. 

\noindent
\textbf{\textit{Projections of probability vectors.}} If $p=(p_i)_{i\in\mathcal I}$ is a probability vector and $2\le r\le s$, denoting $\Pi_r=\Pi_r^D$, we set $\Pi_rp=((\Pi_rp)_j)_{j\in  \Pi_r(\mathcal I)}$, where $(\Pi_rp)_j=\sum_{i\in\Pi_r^{-1}(\{j\})}p_i$.

\noindent
\textbf{\textit{A reduction.}} In the rest of the paper we assume, without loss of generality, the following geometric property: for all $1\le k\le d$ and $s\in\{0,1\}$, 
$$
\mathcal I(k,s)=\big \{i\in \mathcal{I}:\, d(f_{i}([0,1]^d), \{(z_1,\ldots,z_d)\in [0,1]^d: \, z_{k}=s\})=0\big \}\varsubsetneq \mathcal I.
$$
Otherwise the set $K_\omega$ is almost surely contained in one of the faces of $[0,1]^d$ and we are back to the dimension $d-1$, and if $d=1$ $K_\omega$ is a singleton. 

The previous assumption has the following useful consequence (we postpone the proof to the end of Section~\ref{pfthmdim}).
\begin{proposition}\label{mursnonchargés} Let $\nu$ be an IMM. Suppose that for all $1\le k\le d$ and $s\in\{0,1\}$, for all $N\ge 1$, $\prod_{n=N}^\infty \left (\sum_{i\in\mathcal I(k,s)} \mathbb{E}(W^{(n)}_{i})\right )=0$  (this is  in particular the case when the $\mathbb{E}(W^{(n)}_i)$ are bounded away from $0$ independently of $n$).  

Then, with probability~1, $\mu(\pi([w]))=\nu([w])=\mu(f_w((0,1)^d))$, and $\mu(\partial f_w([0,1]^d))=0$ for all $w\in\mathcal I^*$, where $f_{w_1\cdots w_n}=f_{w_1}\circ\cdots\circ f_{w_n}$ if $n\ge 1$ and $f_{\boldsymbol{\epsilon}}=\mathrm{Id}_{\mathbb{R}^d}$. 
\end{proposition}

\subsection{Main statement}\label{thmdim}$\ $

\noindent
\textbf{\textit{Decomposition of the random weights.}} We will use, for each random vector $W^{(n)}(v)$ involved in the  construction of $\nu$, the same decomposition as that of $W$ in the introduction part, namely $W=(p_i\widetilde W_i)_{i\in\mathcal I}$, where $p$ is the probability vector $(\mathbb{E}(W_i))_{i\in\mathcal I}$.  To do so we consider the sequence $\boldsymbol{p}=(p^{(n)})_{n\in\mathbb{N}^+}$ of probability vectors in $\mathbb R_+^{\mathcal I}$ obtained as follows: 
\begin{equation}\label{pni0}
p^{(n)}_i=\mathbb{E}(W^{(n)}_i)\quad \forall\, n\in\mathbb{N}^+,\, \forall \, i\in\mathcal {I}.
\end{equation}
Then for all $n\in\mathbb{N}^+$, $v\in\mathcal I^{n-1}$ and $i\in\mathcal I$ set  
\begin{equation}\label{tildeW}
\widetilde W_i(v)=
\begin{cases}
1&\text{ if } p^{(n)}_i=0\\
\frac{W_i(v)}{p^{(n)}_i}&\text{ otherwise}
\end{cases}\quad \text{and} \quad\widetilde W^{(n)}_i=
\begin{cases}
1&\text{ if } p^{(n)}_i=0\\
\frac{W^{(n)}_i}{p^{(n)}_i}&\text{ otherwise}
\end{cases}
\end{equation}

\smallskip

\noindent
\textbf{\textit{Nested family of principal subspaces adapted to the  Lyapounov exponents associated to $\boldsymbol{p}=(p^{(n)})_{n\in\mathbb{N}^+}$ at scale $e^{-N}$.}} Recall the definition of the Lyapunov exponents associate with a probability vector~\eqref{chip}. For each $N\in\mathbb{N}^+$, define $\widehat{\boldsymbol{p}}_N=\frac{1}{N}\sum_{n=1}^N p^{(n)}$. Then, for all $k\in\{1,\cdots d\}$, let 
\begin{equation}\label{lambdakN}
\gamma_k(N)=\inf\{n\in\mathbb{N}^+:\,  n\chi_k(\widehat{\boldsymbol{p}}_n)>N\}.
\end{equation}
There exists a unique integer $s(N)\ge 1$ and a unique partition $({A_r(N)})_{1\le r\le s(N)}$ of $\{1\ldots,d\}$ such that $(i)$ for all $1\le r\le s(N)$, for all $k,k'\in A_r(N)$, one has $\gamma_k(N)=\gamma_{k'}(N):=g_r(N)$ and $(ii)$ for all $r<r'$ one has $g_r(N)<g_{r'}(N)$. We set $D_r(N)=\bigcup_{r'=r}^{s(N)} A_{r'}(N)$ for all $1\le r\le s(N)$, and $D(N)= (D_r(N))_{1\le r\le s(N)}$. 

By construction, for all $2\le r\le s(N)$, $k\in D_r(N)$ and $k'\in \bigcup_{r'=1}^{r-1}A_{r'}(N)=\{1,\ldots,d\}\setminus D_r(N)$, one has $\chi_{k}(\widehat{\boldsymbol{p}}_{g_{r-1}(N)})\le \frac{N}{g_{r-1}(N)}< \chi_{k'}(\widehat{\boldsymbol{p}}_{g_{r-1}(N)})$. Thus $D_r(N)=D(p, x)\in\mathscr{D}$, where $p=\widehat{\boldsymbol{p}}_{g_{r-1}(N)}$ and $x=\frac{N}{g_{r-1}(N)}$. Also, $D_1(N)=\{1,\ldots,d\}$. 

When $\nu$ is a Mandelbrot measure associated with random vectors  identically distributed with a vector $W$, setting $p=\mathbb{E}(W)$, one has $\widehat{\boldsymbol{p}}_N=p$  for all $N\in\mathbb{N}^+$ and $D(N)$ is independent of $N$ for $N$ large enough, so that we simply denote it by $D$. Also, we denote the common value of $\chi_{k}(p)$ for $k\in A_r$ by~$\widetilde \chi_{r}(p)$.

Finally we define a sequence $(d_N)_{N\in\mathbb{N}^+}$ which, under mild assumptions, will describe (see the proof of Theorem~\ref{thm-2.4} in Section~\ref{proofsdimmu}) for $\mu$-almost every point $z$, the fluctuations of the local H\"older exponent of $\mu$ in the sense that one essentially has the scaling relation
\begin{equation}\label{scaling}
\mu(B(z,e^{-N}))\approx (e^{-N})^{ d_N}.
\end{equation} 

\begin{definition}\label{HN}
For $N\ge 1$, writing $\Pi$ for $\Pi^{D(N)}$ and $s$ for $s(N)$, set 
\begin{align*}
H_{N,k}=\sum_{n=1}^k H({W^{(n)}})+\sum_{n=k+1}^{g_s(N)} h(\Pi_{r_n}p^{(n)})\quad (0\le k\le g_s(N)),
\end{align*}
where $r_n$ is the index $r$ such that $g_{r-1}(N)+1\le n\le g_{r}(N)$ and we recall that the entropy of a finite dimentional probability vector $q=(q_j)_{j\in\mathcal J}$ is defined as $h(q)=-\sum_{j\in\mathcal J}q_j\log(q_j)$.

Also, set 
\begin{align}\label{uN0}
d_N&=\frac{1}{N}\min \Big (\min_{g_1(N)\le k\le  g_s(N)-1} H_{N,k},\min_{N'\ge g_s(N)} \sum_{n=1}^{N'} H(W^{(n)})\Big )\\
\label{uN1} \text{and}\quad
\widetilde d_N&=\frac{1}{N}\min_{g_1(N)\le k\le  g_s(N)} H_{N,k}.
\end{align}

Note that in the deterministic case, $d_N=N^{-1} H_{N,g_1(N)}=\widetilde d_N$, and that if $H(W^{(n)})\ge 0$ for all $n\ge 1$, which is the case for non-degenerate Mandelbrot measures, then $d_N= \widetilde d_N$. 
\end{definition}

We can now state our result on the dimensions of $\mu$.  The definitions of Hausdorff and packing dimensions, and of exact dimensionality of a measure are recalled in Section~\ref{appendix}. 

\begin{theorem}\label{thm-2.4} Suppose $\mu$ is non degenerate, that the assumptions of Theorem~\ref{sufcond2}(1) hold (so that $\mu$ is non-degenerate),  and that the assumption of Proposition~\ref{mursnonchargés} holds as well. Let 
$$
\underline{d}(\mu)=\liminf_{N\to +\infty} d_N \text{ and } \overline{d}(\mu)=\limsup_{N\to +\infty} d_N. 
$$
\begin{enumerate}
\item With probability~1, conditional on $\{\mu\neq 0\}$, $\dim_H(\mu)\le \underline{d}(\mu)$ and $\dim_P(\mu)\le \overline{d}(\mu)$. 

In particular, if  $\displaystyle \liminf_{N\to+\infty} N^{-1}\sum_{n=1}^NH(W^{(n)})=0$, then $\dim_H(\mu)=0$.

\item Suppose that $\displaystyle\liminf_{N\to+\infty} N^{-1}\sum_{n=1}^NH(W^{(n)})>0$ and $\sup_{n\ge 1}\phi_{\widetilde W^{(n)}}(q)<+\infty$ for some $q \in \left(1,2\right]$.
With probability~1, conditional on $\{\mu\neq 0\}$,  $\dim_H(\mu)= \underline{d}(\mu)$ and $\dim_P(\mu)= \overline{d}(\mu)$.

\item Suppose that $\mu$ is a Mandelbrot measure and $\phi_{W}(q)<+\infty$ for some $q \in \left(1,2\right]$. With probability~1, conditional on $\{\nu\neq 0\}$, $\mu$ is exact dimensional, with dimension
$$
\underline{d}(\mu)=\overline{d}(\mu)=\frac{H(W)}{\widetilde \chi_{1}(p)}+\sum_{r=2}^s\Big ( \frac{1}{\widetilde \chi_{r}(p)}- \frac{1}{\widetilde \chi_{r-1}(p)}\Big )  \min \big (H(W),h(\Pi^D_{r}p))
$$
(recall that the exponents $\widetilde \chi_r(p)$ were introduced just before Definition~\ref{HN}).
\item It turns out that $\underline{d}(\mu)=\liminf_{N\to +\infty} \widetilde d_N$, while $\overline d(\mu)\neq \limsup_{N\to +\infty}\widetilde d_N $ in ge\-neral.
\end{enumerate}
\end{theorem}

\begin{remark}
(1) The case of Mandelbrot measures (item (3)) is an extension of the result obtained in \cite{ref6} for random Sierpi\'nski sponges. We recover the competition  between the entropy dimension of the measure $\nu$ and those of the expectations  of its successive  projections on symbolic spaces of the form $({\mathcal I}_r^{D})^{\mathbb{N}^+}$, $2\le r\le s$. In item~(2), we see how the phenomenon generalises  in the determination of the Hausdorff and packing dimensions of IMMs with, in particular, the additional contribution of the term $\min_{N'\ge g_s(N)} \sum_{n=1}^{N'} H(W^{(n)})$. This term is absent in the deterministic case (inhomogeneous Bernoulli measures) as well as in the case of Mandelbrot measures. It accounts for the influence of long finite subsequences of $(H(W^{(n)}))_{n\ge 1}$ with a high proportion of negative terms, both in the fluctuations of the local H\"older exponent along the scales (see \eqref{scaling}) and in the value of $\dim_P(\mu)$,  but it does not affect the value of $\dim_H(\mu)$. Also,  even when all the $H(W^{(n)})$ are non negative, the expression taken by the sequence $(d_N)_{N\ge 1}$ is not clear to anticipate from the forms it takes when it is specialized to the deterministic case or to Mandelbrot measures.

\noindent
(2) In the simplest of the conformal cases, i.e when the $a_{i,k}$ are all equal to the same positive contraction ratio, $d_N$ reduces to $N^{-1}\min_{N'\ge g_1(N)}\sum_{k=1}^{N'} H(W^{(n)})$, and Theorem~\ref{dimHmu} is a substantial improvement of \cite[Theorem 8]{ref4} (in which one works on the boundary of a general Galton-Watson tree), where it is assumed the very strong property that $\sup_{v\in\mathcal I^*}\|Y(v)\|_q<+\infty$ for some $q\in(1,2]$, a property which holds for instance when  $\sup_{n\in\mathbb{N}^+}\phi_{W^{(n)}}(q)<1$, which implies that $H(W^{(n)})>0$ for all $n\ge1$.  

\noindent(3)(i) In the deterministic case, when $\mu$ is self-affine, the exact dimensionality of $\mu$ follows from the general fact that any self-affine measure associated to an IFS made of inverti\-ble affine maps is exact dimensional \cite{BK,FengDuke}. For Bernoulli measures on good sponges, one can get this exact dimensionality by an appropriate exploitation of the multiplicative structure of Bernoulli measures and there orthogonal projections (which are still Bernoulli measures) and the SLLN \cite{ref12,ref10,ref18,Olsen,ref1}. As mentioned in the introduction,  a similar approach yields Theorem~\ref{thm-2.4}(2) for deterministic inhomogeneous Bernoulli measures~\cite{ref5}).  

\noindent(3)(ii) Using alternatively the differentiability at 1 of the so-called $L^q$-spectrum of the measure, combined with a general result by Ngai~\cite{Ngai} (which is a to get the exact dimensionality) has been done for self-affine  (and more generally Gibbs) and Mandelbrot measures on Sierpi\'nski carpets and sponges~\cite{BM,ref6} (as explained in the introduction, in the random case such a large deviations approach has the advantage that it can exploit the fact that projections of MMs have multiplicative properties (only) in the mean), and self-affine measures on Gatzouras-Lalley sponges~\cite{Kol}. With general good sponges, even in the deterministic case the $L^q$-spectrum is hard to control because of the more complex asymptotic behavior of the Lyapunov exponents associated to the measure $\nu$ along the scales: ordering the principal directions according to the ordering of these Lyapunov exponents yields a permutation which varies along the scales.  Our approach is still based on large deviations, but to circumvent the difficulty raised by the $L^q$-spectrum, one must consider  partition functions constructed over the values of the Lyapunov exponents. This is a special case of the general approach described in the introduction to get Theorem~\ref{thm-2.4}(2). 

\end{remark}

\subsection{The case of IMMs of exponentially continuous and periodic type}\label{epex} Here we consider the extension to inhomogeneous Mandelbrot measures of Das and Simmons result~\cite{ref5} about the Hausdorff dimension of inhomogeneous Bernoulli measures of exponentially continuous and periodic type. This case is interesting in its own right because it yields nice formulas beyond the case of Mandelbrot measures. Also, it makes it easy to construct, from Das and Simmons deterministic example, an example of non-deterministic random sponge whose Hausdorff dimension is not the supremum of the Hausdorff dimensions of the  Mandelbrot measures it supports. 

\noindent
\textbf{\textit{Construction and dimensions.}} Consider an IMM on $K$, and suppose that there exists $\boldsymbol{p}=(p^{(n)})_{n\ge 1}\in P_{\mathcal I}^{\mathbb{N}^+}$ and a non negative random vector $\widetilde W=(\widetilde W_i)_{i\in\mathcal I}$ whose components have expectation 1 and such that each vector $W^{(n)}=(W^{(n)}_i)_{i\in\mathcal I}$ is distributed as $(p^{(n)}_i\widetilde W_i)_{i\in\mathcal I}$. Suppose also that  $\boldsymbol{p}$ is the restriction to the positive integers of a non constant exponentially continuous  and periodic $(p^{(t)})_{t>0}$ and that $\phi_{\widetilde W}(q)<+\infty$ for some $q\in (1,2]$. For each $t>0$, denote by $W^{(t)}$ a random vector distributed like $(p^{(t)}_i\widetilde W_i)_{i\in\mathcal I}$. 

Denote by $\lambda$ the exponential period of this function; one has $\lambda>1$. A calculation shows that 
$$
\liminf_{T\to +\infty} \frac{1}{T}\int_0^TH(W^{(t)})\,\mathrm{d}t=\min_{T\in[1,\lambda]}\frac{1}{T}\Big (\frac{1}{\lambda-1}\int_1^\lambda  H(W^{(t)})\,\mathrm{d}t+\int_1^T H(W^{(t)})\,\mathrm{d}t\Big ).
$$ 
We suppose that $\liminf_{T\to +\infty} T^{-1}\int_0^TH(W^{(t)})\,\mathrm{d}t>0$. The assumptions of Theorem~\ref{sufcond} are then satisfied since the fact that $\phi_{\widetilde W}(q)<+\infty$ implies that $\mathbb{E}(\sum_{i\in\mathcal I}\widetilde W_i\log^2(\widetilde W_i))<+\infty$, which yields $\sup_{n\ge 1} \phi_{W^{(n)}}''(1)<+\infty$. Thus,  the associated measures  $\nu$  and $\mu=\pi_*\nu$ are non degenerate.

In the spirit of the definitions of Proposition~\ref{Moments1}(2), for $T>0$ denote by $\widetilde T$ the minimum of those $T'\ge T$ at which $\int_T^{T'}H(W^{(t)})\,\mathrm{d}t$ reaches its minimum. The mapping $T\to\widetilde T$ is also exponentially continuous  and equivariant, with same exponential period $\lambda$. Like Das and Simmons in~\cite{ref5}, rather than just considering Lyapunov exponents associated to the discrete scales $e^{-N}$ via the formulas \eqref{lambdakN}, we can associate them with continuous scales $e^{-T}$, with the alternative formulas 
$$
\gamma_k(T)=\inf\left \{t>0:\,  t\cdot \chi_k\left(\frac{1}{t}\int_0^t\boldsymbol{p}^{(u)}\,\mathrm{d}u\right)>T\right \}\quad (1\le k\le d)
$$
and get objects $D(T)$, $s(T)$, $g_1(T),\ldots, g_{s(T)}(T)$ defined in the same way as $D(N)$, $s(N)$, $g_1(N),\ldots, g_{s(N)}(N)$; these functions of $T$ are $\lambda$-exponentially periodic. One has the following extension of \cite[Theorem 3.2 (good sponges case)]{ref5}, which provides  $\dim_H(\mu)$ in the deterministic case. 

\begin{theorem}\label{dimep}Suppose that the vectors $p^{(t)}$, $t>0$, have positive entries. With probability~1, conditional on $\nu\neq 0$, one has 
\begin{align}
\label{dimmuep}\dim_H(\mu)&=\inf_{T\in[1,\lambda]} \min (\delta_1(T),\delta_2(T))\\
\label{Dimmuep}\dim_P(\mu)&= \sup_{T\in[1,\lambda]} \min (\delta_1(T),\delta_2(T)),
\end{align}
where for $T>0$, 
\begin{align*}
\delta_1(T)&=T^{-1} \int_0^{\widetilde {g_s(T)}}H(W^{(t)})\,\mathrm{d}t,\\
\delta_2(T)&=T^{-1}\min_{g_1(T)\le T'\le g_s(T)}\int_0^{T'}H(W^{(t)})\,\mathrm{d}t+\int_{T'}^{g_s(T)}h(\Pi^{D(T)}_{r_t}H(W^{(t)}))\,\mathrm{d}t,
\end{align*}
$r_t$ being the index $r$ such that $g_{r-1}(T)<t\le g_r(T)$. 
\end{theorem}
A sketched proof will be given at the end of Section~\ref{proofsdimmu}.  

\noindent
\textbf{\textit{Random perturbation of Das and Simmons example of Gatzouras-Lalley sponge.}} We borrow from \cite{ref5} the information that when $d\ge 3$, one can consider an example of Gatzouras-Lalley sponge $K$, a real number $\lambda >1$, as well as $\boldsymbol{p}=(p^{(n)})_{n\ge 1}\in P_{\mathcal I}^{\mathbb{N}^+}$ which  is the restriction to the positive integers of a non constant $\lambda$-exponentially continuous  and periodic function $(p^{(t)})_{t>0}$, such that the associated inhomogeneous Bernoulli measure $\mu_0$ is fully supported on $K$ and satisfies $\dim_H(\mu_0)>\sup\{\dim(\eta):\, \eta \text{ is self-affine with }\supp(\eta)\subset K\}$ (this supremum equals that taken over the pushforward by the coding map $\pi$ of $T$-invariant ergodic measures on $\Sigma$).

Fix $i_0\in\mathcal I$, and choose the vector $C=(c_i)_{i\in \mathcal I}$ such that $\mathbb{P}_{c_i}=\delta_1$ if $i\in\mathcal I\setminus\{i_0\}$ and $\mathbb{P}_{c_{i_0}}=(1-\alpha)\delta_0+\alpha\delta_1$, with $\alpha\in(0,1)$. The set $K_\omega$ is non empty with probability~1. Consider the IMM $\widetilde \mu_\omega$ obtained by considering the random vectors $W^{(n)}(v)=\big (p^{(n)}_i \frac{\mathbf{1}_{\{c_i(v)=1\}}}{\mathbb{P}(c_i=1)}\big )_{i\in\mathcal I}$ for $n\ge 1$ and $v\in\mathcal I^{n-1}$. Define $W^{(t)}=(p^{(t)}_i \frac{\mathbf{1}_{\{c_i=1\}}}{\mathbb{P}(c_i=1)})_{i\in\mathcal I}$. By construction, if $\alpha$ is close enough to $1$, one has $\min_{t>0}H(W^{(t)})>0$. In particular, $\widetilde T=T$ for all $T>0$ so that it is direct to see that $\dim_H(\mu_\omega)$ tends to $\dim_H(\mu_0)$ as $\alpha$ tends to $1$. On the other hand, it is also clear from Theorem~\ref{thm-2.4}(3) that as $\alpha$ tends to 1, the supremum of the Hausdorff dimensions of Mandelbrot measures supported on $K_\omega$ converges to that of the  Hausdorff dimensions of self-affine measures supported on $K$. This yields a family of non-deterministic examples of statistically self-affine sponges for which Theorem~\ref{dimKd} cannot be improved by restricting the variational principle to Mandelbrot measures. 

\section{Non degeneracy and moments of inhomogeneous Mandelbrot martingales}\label{ND}

We consider an inhomogeneous Mandelbrot measure $\nu_\omega$ as defined in Section~\ref{introduction}. We first establish Theorem~\ref{sufcond2}. Next we consider some estimates of moments of order in $(1,2]$ for the random variables $Y(v)$, $v\in\mathcal I^*$. 

\subsection{Proof of Theorem~\ref{sufcond2}} We adapt the size-biasing approach to Theorem~\ref{thmKP} developed in~\cite{WW} and~\cite{Lyons}, and combine it with the classical approach to the strong law of large numbers for non necessarily identically distributed  random variables (see, e.g., \cite[Section 3.6]{Br}). For $n\ge 1$, denote by $\mathcal G_n$ the $\sigma$-algebra generated by $\{W_{i}(w) : i \in \mathcal{I}, \ w \in \bigcup_{k=0}^{n-1}\mathcal{I}^{k}))$ and set $\mathcal G_\infty=\sigma(\bigcup_{n\ge 1}\mathcal G_n))_{n\ge 1}$.  Denote by $\mathbb{Q}_n$ the probability measure on $(\Omega, \mathcal G_n)$ defined by $\mathbb{Q}_n(\mathrm{d}\omega)=Y_n(\omega)\, \mathbb{P}_{|\mathcal G_n}(\mathrm{d}\omega)$. Due to the martingale structure of $(Y_n,\mathcal G_n)_{n\ge 1}$, Kolmogorov's extension theorem yields a unique probability measure $\mathbb{Q}$ on $(\Omega, \mathcal G_\infty)$ such that $\mathbb Q_{|\mathcal G_n}=\mathbb Q_n$ for all $n\ge 1$. 

According to \cite[Theorem 4.3.5]{Dur}, if $\limsup_{n\to +\infty} Y_n=+\infty$ $\mathbb Q$-a.s., then $Y=0$ $\mathbb P$-a.s., while if $\limsup_{n\to +\infty} Y_n<+\infty$ $\mathbb Q$-a.s., then $(Y_n)_{n\ge 1}$ is $\mathbb{P}$-uniformly integrable, and if $0<\mathbb{Q}(\limsup_{n\to +\infty} Y_n=+\infty)<1$, then $0<\mathbb{E}(Y)<1$ (this is a consequence of the Radon-Nikodym decomposition of $\mathbb{Q}$ with respect to $\mathbb{P}$). 

Denote by $\mathcal Q_n$ the probability measure on $(\Omega\times\Sigma,\mathcal G_n\otimes\mathcal B(\Sigma))$ defined by 
$$ \mathcal Q_n(\mathrm{d}\omega,\mathrm{d}\boldsymbol{i})= (\# \mathcal I)^{n}Q(\omega,\boldsymbol{i}_{|n})\, \mathbb{P}_{|\mathcal G_n}(\mathrm{d}\omega)\, \lambda(\mathrm{d}\boldsymbol{i}),
$$
 where $\lambda$ is the measure of maximal entropy on $\Sigma$. 
 
 There is a unique probability measure $\mathcal Q$ on $(\Omega\times\Sigma,\mathcal G_\infty\otimes\mathcal B(\Sigma))$ such that for all $n\ge 1$, $\mathcal Q_{|\mathcal G_n\otimes\mathcal B(\Sigma)}=\mathcal Q_n$. Moreover, denoting by $\pi_\Omega:\Omega\times \Sigma\to\Omega$ the projection on the first coordinate, one has $\mathbb{Q}=\mathcal Q\circ \pi_\Omega^{-1}$.  
 
We note that $\log (Q(\omega,\boldsymbol{i}_{|n}))=\sum_{k=1}^n \log(W_{\boldsymbol{i}_k}(\omega, \boldsymbol{i}_{|k-1}))$ and  that the random variables $X_k: (\omega,\boldsymbol{i})\mapsto\log(W_{\boldsymbol{i}_k}(\omega, \boldsymbol{i}_{|k-1}))$ are $\mathcal Q$-independent. Also, for any non negative measurable function $g$, $\mathbb{E}_\mathcal Q(g(X_k))=\mathbb{E}(\sum_{i\in\mathcal I} W^{(k)}_ig(W^{(k)}_i))$.

Moreover, our assumption on the sequence $(W^{(n)})_{n\ge 1}$ translates into $\sum_{k\ge 1}\frac{\mathbb{E_{\mathcal Q}}(|X_k|^h)}{b_n ^h}<+\infty$. It implies that the series $\sum_{k\ge 1}\frac{\mathbb{E_{\mathcal Q}}(|X_k-\mathbb{E}_{\mathcal Q}(X_k)|^h)}{b_n^h}$ converges. In particular, the sequence  $\big(\sum_{n=1}^N\frac{X_n-\mathbb{E}_{\mathcal Q}(X_n)}{b_n}\Big )_{N\ge 1}$, which is by construction a martingale with respect to its natural filtration, is bounded in $L_{\mathcal Q}^h$, by Lemma~\ref{lemvBE}. By Kronecker's Lemma, the almost sure convergence of the martingale then implies that $\mathcal Q$-a.s., $b_N^{-1}\sum_{n=1}^N X_n +b_N^{-1}\sum_{n=1}^N H(W^{(n)})$ tends to $0$ as $N\to+\infty$, where we used that  $\mathbb{E}_{\mathcal Q}(X_n)=-H(W^{(n)})$.

Now, observe with~\cite{WW,Lyons} that for all $\boldsymbol{i}\in\Sigma$ and $n\ge 1$, 
\begin{equation}\label{controlYn}
Q(\omega,\boldsymbol{i}_{|n})\le Y_n(\omega)=Q(\omega,\boldsymbol{i}_{|n}) +\sum_{k=0}^{n-1} Q(\omega,\boldsymbol{i}_{|k})\cdot M_{n,k}(\omega,\boldsymbol{i}),
\end{equation}
where 
$$
 M_{n,k}(\omega,\boldsymbol{i})=\sum_{i\neq \boldsymbol{i}_{k+1}} W_{i}(\boldsymbol{i}_{|k})Y_{n-k-1}(\boldsymbol{i}_{|k}i).
$$ 

\noindent
\textit{Proof of (1).} Suppose that $c=\liminf_{N\to +\infty} b_N^{-1}\sum_{n=1}^N H(W^{(n)})>0$. By the observation made above, $\limsup_{k\to +\infty}\frac{\log Q(\omega,\boldsymbol{i}_{|k})}{b_k}=-c$ $\mathcal Q$-a.s. Fix $0<c'<c$. 

For each $\boldsymbol{i}\in \Sigma$, define $\mathcal G^{\boldsymbol{i}}=\sigma (W(\boldsymbol{i}_{|n}), \, n\ge 0)$. Using the independence between $Q(\omega,\boldsymbol{i}_{|k})$ and each $Y_{n-k-1}(\boldsymbol{i}_{|k}i)$ in the right hand side of \eqref{controlYn}, one obtains 
$$
\mathbb{E}_{\mathbb Q}(Y_n|\mathcal G^{\boldsymbol{i}})=Q(\omega,\boldsymbol{i}_{|n})+\sum_{k=0}^{n-1} Q(\omega,\boldsymbol{i}_{|k})\cdot \sum_{i\neq \boldsymbol{i}_{k+1}}W^{(k+1)}_i(\boldsymbol{i}_{|k}).
$$
Now, note that  for all $\epsilon\in (0,c'/2)$, 
\begin{align*}
\mathcal{Q}\Big (\sum_{i\neq \boldsymbol{i}_{k+1}}W^{(k+1)}_i(\boldsymbol{i}_{|k})\ge e^{\epsilon b_k}\Big )&\le \mathcal{Q}\Big( \Big [\log^+\sum_{i\in\mathcal I}W^{(k+1)}_i(\boldsymbol{i}_{|k})\Big ]^h\ge \epsilon^h b_k^{h}\Big )\\
&\le( \epsilon^h b_k^h)^{-1}\mathbb{E}_{\mathcal Q}\Big( \Big [\log^+\sum_{i\in\mathcal I}W^{(k+1)}_i(\boldsymbol{i}_{|k})\Big ]^h\Big )\\
&=( \epsilon^h b_k^{h})^{-1}\mathbb{E}\Big (\Big [\sum_{i\in\mathcal I}W^{(k+1)}_i\Big ]\Big [\log^+\sum_{i\in\mathcal I}W^{(k+1)}_i\Big ]^h\Big).
\end{align*}
Moreover, by convexity of $x\ge 0\mapsto x(\log^+(x))^h$, our assumption also implies that $\sum_{k\ge 1}b_k^{-h} \mathbb{E}\Big (\Big [\sum_{i\in\mathcal I}W^{(k+1)}_i\Big ]\Big [\log^+\sum_{i\in\mathcal I}W^{(k+1)}_i\Big ]^h\Big)$ is finite.  By the Borel-Cantelli Lemma, we conclude that $\mathcal Q$-a.s., $ \sum_{i\neq \boldsymbol{i}_{k+1}}W^{(k+1)}_i(\boldsymbol{i}_{|k})\le e^{\epsilon b_k}$ for $k$ large enough. Moreover, $\mathcal Q$-a.s., $Q(\omega,\boldsymbol{i}_{|k})\le e^{-c'b_k/2}$  for $k$ large enough. 

It follows that if we denote by $\pi_\Sigma: \Omega\times\Sigma\to\Sigma$  the projection on the second coordinate, for $\mathcal Q\circ\pi_\Sigma^{-1}$ almost every $\boldsymbol{i}$, $\liminf_{n\to +\infty} \mathbb{E}_{\mathbb Q}(Y_n|\mathcal G_{\boldsymbol{i}})<+\infty$, hence by the Fatou Lemma $\mathbb{E}_{\mathbb Q}(\liminf_{n\to +\infty}Y_n|\mathcal G_{\boldsymbol{i}})<+\infty$. This implies that $\liminf_{n\to +\infty}Y_n<+\infty$ $\mathbb{Q}$-a.s. However, by construction, $(Y_n^{-1},\mathcal G_n)_{n\ge 1}$ is a non negative martingale under $\mathbb Q$,  so $\liminf_{n\to +\infty}Y^{}_n<+\infty$ $\mathbb{Q}$-a.s implies that $\liminf_{n\to +\infty}Y^{-1}_n$ is $\mathbb Q$-almost surely  positive, and finally $\limsup_{n\to +\infty}Y_n<+\infty$ $\mathbb Q$-almost surely. 

The fact that $\{Y>0\}=\{\Sigma_{\omega}\neq\emptyset\}$ up to a set of null $\mathbb{P}$-probability follows from Kolmogorov's $0$-$1$ law and our assumption which implies that for all $n\ge 1$ one has  $\{Y_n>0\}= \{\Sigma_{n,\omega}\neq\emptyset\}$. Then, that $\supp(\mu_\omega)=K_\omega$ almost surely, conditional on $K_\omega\neq 0$, follows from the same argument applied recursively to the surviving subtrees. 

\noindent 
\textit{Proof of (2).} In this case, since $\mathcal Q$-a.s.
$$
\limsup_{N\to +\infty} b_N^{-1}\log(Q(\omega,\boldsymbol{i}_{|N}))=-\liminf_{N\to +\infty} b_N^{-1}\sum_{n=1}^N H(W^{(n)})>0,
$$
we can directly use the left hand side of \eqref{controlYn} to get   $\limsup_{n\to +\infty}Y_n=+\infty$ $\mathbb Q$-a.s.
\subsection{Estimates for the $L^q$ moments of $Y_n(v)$}

If $\lim_{N\to +\infty} \sum_{n=1}^\infty H(W^{(n)})=+\infty$ (which is the case under the assumptions of Theorem~\ref{sufcond2}(1)), for $N\ge 1$ set  
\begin{align}
\label{AN}A_N&=\min\Big \{\sum_{n=N+1}^{N'}H(W^{(n)}):\, N'\ge N\Big\}\\
\label{Ntilde} \text{and }\widetilde N&=\min\Big\{N'\ge N:\, \sum_{n=N+1}^{N'}H(W^{(n)})=A_N\Big\}.
\end{align}
Note that $N\mapsto \widetilde N$ is non decreasing.

The first part of the following result is, except for the first claim,  a corollary of  results obtained in \cite{ref4}. The second one, which invokes $\widetilde N$, is new and will play an important role in the estimation of the Hausdorff and packing dimensions of IMMs. 
 
\begin{proposition}\label{Moments1} 
\begin{enumerate}  Let $q\in(1,2]$. 
\item Suppose that $\sum_{n\ge 1}\left(\prod_{k=1}^n \phi_{W^{(k)}}(q)\right)^{\frac{1}{q}}<+\infty$. Then, $\lim_{N\to +\infty} \sum_{n=1}^N H(W^{(n)})=+\infty$. Also, for all $v\in \mathcal I^*$, $(Y_n(v))_{n\ge 0}$ converges to $Y(v)$, almost surely and in $L^q$ norm. In particular, $\mathbb{E}\left(Y(v)\right) = 1$. Moreover, there exists  a constant $C=C(\# \mathcal I)> 0$  such that for all  $v\in \mathcal I^*$, $\|Y(v)\|_q\le C \sum_{n=0}^{\infty} \left(\prod_{k=1}^n \phi_{W^{(|v|+k)}}(q) \right)^{\frac{1}{q}} < +\infty$, with the convention that  the empty product at $n=0$ is equal to 1. 

In particular, if $\sup_{n\ge 1} \phi_{W^{(n)}}(q)<1$, then $\sup_{v\in\mathcal{I}^*}(\|Y(v)\|_q)_{v\in\mathcal{I}^*}<+\infty$. 

\item Suppose that  $\sup_{n\ge 1} \phi_{W^{(n)}}(q)<+\infty$ and $\liminf_{N\to +\infty}N^{-1}\sum_{n=1}^N H(W^{(n)})>0$. Then, there exists $\hat q\in (1, q)$ such that 
 $\sum_{n\ge 1}\left(\prod_{k=1}^n \phi_{W^{(k)}}(q')^{\frac{1}{q'}}\right)<+\infty$ for all $q'\in (1,\hat q]$.  
Let  $\epsilon>0$ such that 
\begin{equation}\label{epsilon}
\sum_{n=1}^NH(W^{(n)})\ge n\epsilon \text{ for all $N$ sufficiently large. }
\end{equation}
Let $N_\epsilon=\min\big \{N\ge 1:\, \forall \, N'\ge N, \sum_{n=1}^{N'}H(W^{(n)})\ge N'\epsilon\big\}$. There exists  $\tilde q\in (1,\hat q]$ such that for $q'\in (1,\tilde q]$, there are constants $C=C(q',\#\mathcal I,\epsilon)\ge 1$ and $c=c(q',\# \mathcal I,\epsilon)>0$ such that for all $N\ge N_\epsilon$,  one has $\widetilde N\le \frac{\log (\#\mathcal{I})}{\epsilon} N$ and for $v\in \mathcal{I}^N$, 
\begin{equation}\label{Bnq}
B(N,q')\le \mathbb{E}\left((Y(v))^{q'}\right) \le \frac{CN^{q'} e^{cN(q'-1)^2}}{(1-e^{-\frac{(q'-1)}{4q'}\epsilon})^{q'}}B(N,q'),
\end{equation}
where $B(N,q')=\max \Big (1, \exp \big [-(q'-1)\sum_{n=N+1}^{\widetilde N}H(W^{(n)})\big ]\Big )$. 
\end{enumerate}
\end{proposition}
\begin{remark}Recall the sequence $(d_N)_{N\ge 1}$ considered in Definition~\ref{HN}. The integer $\widetilde N$ defined in \eqref{Ntilde} is also equal to $\min\Big\{N'\ge N:\, \sum_{n=1}^{N'}H(W^{(n)})=A'_N\Big\}$, where $A'_N=\min\Big\{\sum_{n=1}^{N'}H(W^{(n)}): \, N'\ge N\Big\}$. 

Note also that by \eqref{HW}, any $\epsilon$ such that \eqref{epsilon} holds is necessarily smaller than or equal to $\log(\# \mathcal I)$.
\end{remark}
\begin{proof}
(1) To see that $\lim_{N\to +\infty} \sum_{n=1}^N H(W^{(n)})=+\infty$, suppose on the contrary that there is $M>0$ and an increasing sequence of integers $(n_j)_{j\in\mathbb{N}^+}$ such that $\sum_{n=1}^{n_j} H(W^{(n)})\le M$ for all $j\in \mathbb{N}^+$. One checks that the derivative at 1 of $\prod\limits_{k=1}^{n_j} \phi_{W^{(k)}}$ equals  $-\sum_{n=1}^{n_j} H(W^{(n)})$ so by convexity of this product, $\prod\limits_{k=1}^{n_j} \phi_{W^{(k)}}(q')\ge 1-M(q'-1)$ for all $q'>1$. Taking $q'>1$ close enough to $1$ yields $\sum_{n\ge 1} \left (\prod\limits_{k=1}^{n} \phi_{W^{(k)}}(q')\right)^{\frac{1}{q'}}=+\infty$, hence by convexity $\sum_{n\ge 1} \left(\prod\limits_{k=1}^{n} \phi_{W^{(k)}}(q)\right)^{\frac{1}{q}}=+\infty$, which contradicts $\sum_{n\ge 1}\left(\prod_{k=1}^n \phi_{W^{(k)}}(q)\right)^{\frac{1}{q}}<+\infty$. 

The other claims can be deduced from \cite[Theorem 6]{ref4}. For the sake of completeness, just observe that for all $v\in\mathcal I^*$ and $n\ge 0$, 
$$
Y_{n+1}(v)-Y_n(v)=\sum_{w\in \mathcal I^n} Q^v(w)\Big (\sum_{i\in\mathcal I}W_i(vw)-1\Big).
$$ 
Set $\mathcal G_{v,n}=\sigma(Q^v(w):\, w\in \mathcal I^n)$, and note that in the right-hand side the random variables  $\sum_{i\in\mathcal I}W_i(vw)-1$, $w\in\mathcal I^n$, are centered and i.i.d, and independent of $\mathcal G_{v,n}$. Lemma~\ref{lemvBE} applies, conditional on $\mathcal G_{v,n}$, and yields 
$$
\mathbb{E}(|Y_{n+1}(v)-Y_n(v)|^q|\mathcal G_{v,n})\le 2^q\sum_{w\in \mathcal I^n}Q^v(w)^q\|\sum_{i\in\mathcal I} W_i^{(|v|+n+1)}-1 \|_q^q.
$$
Moreover, using the branching property, one gets  $\mathbb{E}(\sum_{w\in \mathcal I^n}Q^v(w)^q)=\prod_{k=1}^n \phi_{W^{(|v|+k)}}(q)$. Then,
\begin{align*}
\|Y_{n+1}(v)-Y_n(v)\|_q &\leq 2 \Big (\prod\limits_{k=1}^n \phi_{W^{(|v|+k)}}(q)\Big )^{\frac{1}{q}} \Big (1+\Big \|\sum_{i\in\mathcal I} W_i^{(|v|+n+1)}\Big \|_q\Big )\\
&\le 2 \Big (\prod_{k=1}^n \phi_{W^{(|v|+k)}}(q)\Big )^{\frac{1}{q}} \Big (1+ (\#\mathcal I)\phi_{W^{(|v|+n+1)}}(q)^{1/q}\Big ).
\end{align*}
Moreover, $Y_0(v)=1$. This is enough to get the remaining part of  item~(1). 

\noindent
(2) Note that $\sup_{n\ge 1}\phi_{W^{(n)}}(q)<+\infty$ implies $\sup_{n\ge 1}\sup_{q''\in [1,q']}\phi''_{W^{(n)}}(q')<+\infty$, which, together with the inequality $\phi_{W^{(n)}}(q')\ge  1+\phi_{W^{(n)}}'(1)(q'-1)=  1-(q'-1)H(W^{(n)})\ge 1-(q'-1)\log (\# \mathcal I)$ valid for all $q'>1$, implies that for some $\bar q\in (1,q)$ one has $c_{\bar q}=\sup_{n\ge 1}\sup_{q'\in [1,\bar q]}\big (\log(\phi_{W^{(n)}})\big)''(q')<+\infty$.  Then,  for all $q'\in (1,\bar q]$, using Taylor-Lagrange's expansion of each $\phi_{W^{(n)}}$, $n\ge 1$, yields 
for all $q'\in (1,\bar q)$, using Taylor-Lagrange's expansion of each $\phi_{W^{(n)}}$ yields 
 \begin{align}
\nonumber \phi_{W^{(n)}}(q')&=\exp \left ((\log(\phi_{W^{(n)}})'(1) \cdot (q'-1)+ c(q') (q'-1)^2\right )\\
\label{expansion}&=\exp \left (-H(W^{(n)}) (q'-1)+ \frac{c(q')}{2} (q'-1)^2\right ),
\end{align}
 where $0\le c(q')\le c_{q'}:=\sup_{n\ge 1}\sup_{t\in [1,q']}\big (\log(\phi_{W^{(n)}})\big)''(t)\le c_{\bar q}$.  

In particular,  under the assumption $\liminf_{N\to+\infty} N^{-1}\sum_{n=1}^NH(W^{(n)})>0$ (which already implies the non degeneracy of $\nu$ by Theorem~\ref{sufcond}), taking $\epsilon>0$ and $N_\epsilon$ as in the statement, for all $1<q'<\hat q=\min \big (\bar q, 1+ \frac{\epsilon}{c_{\bar q}}\big )$,  $\sum_{n=1}^N\log\phi_{W^{(n)}}(q')\le  -(q'-1)\epsilon N/2$ for all  $N\ge N_\epsilon$ large enough, which implies  $\sum_{n\ge 1}\left(\prod\limits_{k=1}^n \phi_{W^{(k)}}(q')\right)^{\frac{1}{q'}}<+\infty$. 

Now, let $N\ge N_\epsilon$ and $v\in\mathcal I^N$. Note that  for all $1<q'\le q$, one has $\mathbb{E}(Y(v)^{q'})\ge  \mathbb{E}(Y(v))^{q'}=1$. This together with the super-additivity of $t\ge 0\mapsto t^{q'}$ and the expectation taken successively on the equality $Y(v)=\sum_{w\in \mathcal{I}^{\widetilde N-|v|}}Q^v(w)Y(vw)$ yields $\mathbb{E}(Y(v)^{q'})\ge \mathbb{E}(Y(\tilde v)^{q'}) \prod_{k=|v|+1}^{\widetilde N}\phi_{W^{(k)}}(q')\ge  \prod_{k=|v|+1}^{N'}\phi_{W^{(k)}}(q')$, where $\tilde v$ is some element of $\mathcal{I}^{\widetilde N}$. This implies the first inequality in \eqref{Bnq} since each function $\log(\phi_{W^{(n)}})$ is convex, and satisfies  $\log (\phi_{W^{(n)}})(1)=0$ and $(\log(\phi_{W^{(n)}}))'(1)=\phi_{W^{(n)}}'(1)$. 

For the second inequality, we first need to control $\widetilde N$ from above, as well as the integer 
$$
\widehat N=\min\Big\{N'\ge \widetilde N+1:\, \forall\, n'\ge N',\,   \sum_{n=\widetilde N+1}^{n'}H(W^{(n)})\ge (n'-\widetilde N)\epsilon/2\Big\}.
$$

Recall the definitions \eqref{AN} and  \eqref{Ntilde} of $A_N$ and $\widetilde N$, and that due to \eqref{HW}, for all $n\ge 1$ one has $H(W^{(n)})\le\log (\#\mathcal{I})$.

\noindent
\textbf{Claim:} for all $N\ge N_\epsilon$, one has $\widetilde N\le \frac{\log (\#\mathcal{I})}{\epsilon}\cdot N$ and $\widehat N\le 3 \left(\frac{ \log (\#\mathcal{I})}{\epsilon}\right )^2\cdot N$. 

Indeed, if $N\ge N_\epsilon$, then $\widetilde N \epsilon\le \sum_{n=1}^{\widetilde N}H(W^{(n)})\le  \sum_{n=1}^{N}H(W^{(n)})\le N\log (\# \mathcal I)$. This yields the first claim. Next, since $\sum_{n=1}^{\widetilde N}H(W^{(n)})\le \widetilde N \log(\#\mathcal I)$ and  $N\ge N_\epsilon$ implies that $\sum_{n=1}^{N'}H(W^{(n)})\ge N'\epsilon$ for all $N'\ge N$, if  $N'\ge \widetilde N+1(\ge N)$, one has $ \sum_{n=\widetilde N+1}^{N'}H(W^{(n)})\ge (N'-\widetilde N)\epsilon/2$ as soon as $N'\epsilon -\widetilde N\log(\#\mathcal I)\ge (N'-\widetilde N)\epsilon/2$. The previous inequality holds for all $N'\ge \widetilde N \cdot \frac{2}{\epsilon} \log (\#\mathcal{I})$ ($\ge \widetilde N+1$). Consequently, $\widehat N\le \lceil\widetilde N\cdot \frac{2}{\epsilon} \log (\#\mathcal{I})\rceil \le \widetilde N\cdot \frac{3}{\epsilon} \log (\#\mathcal{I})$. This yields the second part of the claim.

 Now, for $1<q'<\hat q$ and $n\ge 1$, by (1), we have
 \begin{align*} 
 C(\# \mathcal I)^{-1}\mathbb{E}(Y(v)^{q'})^{1/q'}&\le \sum_{n=N}^{\widehat N-1}\prod_{n'=N+1}^n \phi^{1/q'}_{W^{(n')}}(q')\\
 &\quad \quad +\sum_{n=\widehat N}^\infty \Big (\prod_{n'=N+1}^{\widetilde N} \phi^{1/q'}_{W^{(n')}}(q')\Big) \cdot \prod_{n'=\widetilde  N+1}^{n} \phi^{1/q'}_{W^{(n')}}(q'),
 \end{align*}
and using \eqref{expansion} as well as the definitions of $\widetilde N$ and $\widehat N$ we can get  
\begin{align*}
\sum_{n=N+1}^{\widehat N-1}\prod_{n'=N+1}^n \phi^{1/q'}_{W^{(n')}}(q')&\le (\widehat N-N)\sup_{N+1\le n\le \widehat N-1}\Big (\prod_{n'=N+1}^{n} \phi^{1/q'}_{W^{(n')}}(q')\Big)\\
&\le (\widehat N-N)e^{\frac{c_{q'}}{2}\frac{(q'-1)^2}{q'}(\widehat N-N)} \exp \Big (-\frac{(q'-1)}{q'}\sum_{n'=N+1}^{\widetilde N} H(W^{(n')})\Big ).
\end{align*}
Also, setting $\gamma_{q'}= \frac{c_{q'}}{2}\frac{(q'-1)^2}{q'}$,
\begin{align*} 
&\sum_{n=\widehat N}^\infty \Big (\prod_{n'=N+1}^{\widetilde N} \phi^{1/q'}_{W^{(n')}}(q')\Big) \cdot \prod_{n'=\widetilde N+1}^{n} \phi^{1/q'}_{W^{(n')}}(q')\\
&\le  e^{\gamma_{q'}(\widetilde N-N)} \exp \Big (-\frac{(q'-1)}{q'}\sum_{n'=N+1}^{\widetilde N} H(W^{(n')})\Big )\sum_{n=\widehat N}^\infty e^{\gamma_{q'}(n-\widetilde  N)} e^{-\frac{(q'-1)}{q'}\sum_{n'=\widetilde N+1}^{n} H(W^{(n')})}
\\
&\le  
e^{\gamma_{q'}(\widehat N-N)} \exp \Big (-\frac{(q'-1)}{q'}\sum_{n'=N+1}^{\widetilde N} H(W^{(n')})\Big )\sum_{n=\widehat N}^\infty e^{\gamma_{q'}(n-\widehat N)} e^{- \frac{(q'-1)}{q'}(n-\widehat N)\epsilon/2}.
\end{align*}

We can find $\tilde q\in (1,\hat q]$ such that for all $q'\in (1,\tilde q)$ one has $\frac{c_{q'}}{2}(q'-1)\le \epsilon/4$, and
$$
\sum_{n=\widehat N}^\infty e^{\gamma_{q'}(n-\widehat N)} e^{- \frac{(q'-1)}{q'}(n-\widehat N)\epsilon/2}\le (1-e^{-\frac{(q'-1)}{4q'}\epsilon})^{-1}.
$$
By using the inequality $\widehat N\le 3\left(\frac{\log (\#\mathcal{I})}{\epsilon} \right )^2$ we finally get that for some positive constant $C=C(q',\# \mathcal I,\epsilon) $ and $c=C(q',\# \mathcal I,\epsilon)$, one has 
$$
\mathbb{E}(Y(v)^{q'})\le  \frac{CN^{q'} e^{cN(q'-1)^2}}{(1-e^{-\frac{(q'-1)}{4q'}\epsilon})^{q'}}\max \Big (1, \exp \big [-(q'-1)\sum_{n=N+1}^{\widetilde N}H(W^{(n)})\big ]\Big ).
$$
\end{proof}

\section{Proofs of Theorems~\ref{thm-2.4} and~\ref{dimep}}\label{proofsdimmu}
Our estimates of the dimensions of the measure $\mu$ are based on a large deviations argument using appropriate partition functions. Rather than directly use the partition functions adapted to get the results of Section~\ref{DM}, we prefer to derive general estimates in the next preliminary section.
\subsection{Estimates of some partition functions} \label{partfunc} We use the notation introduced in Sections~\ref{sechigherdim} and~\ref{preliminaries}. Fix $s\in\mathbb{N}^+$ and $D=(D_r)_{1\le r\le s}\in \mathscr{D}^s$, such that  $D_1\supsetneq \cdots\supsetneq D_s\neq\emptyset$. 

Recall \eqref{pni0} and \eqref{tildeW}. For $n\in\mathbb{N}^+$, $r\in\{1,\ldots, {s}\}$,  $j\in \mathcal I^D_r$ and $q\ge 0$ recall that 
$$
(\Pi^D_rp^{(n)})_j=\sum_{i\in (\Pi^{D}_r)^{-1}(\{j\})}p^{(n)}_{i}
$$ 
and set
\begin{equation}\label{taunq}
\tau^{D,n}_r(q) = - \log \sum_{j\in \mathcal I^D_r} (\Pi^D_rp^{(n)})_j^q .
\end{equation}
Also, recall (see \eqref{phiW}) that 
$$
T_{W^{(n)}}(q) =-\log(\phi_{W^{(n)}}(q))= - \log  \mathbb{E}\left(\sum_{i \in \mathcal{I}} \left({W}^{(n)}_{i}\right)^q \right)=- \log  \mathbb{E}\left(\sum_{i \in \mathcal{I}} \left(p^{(n)}_{i}  \widetilde{W}^{(n)}_{i}\right)^q \right). 
$$

Now we consider a finite sequence of positive integers $g=(g_r)_{1\le r\le s}$, such that  $g_{1}<\cdots <g_s$.  
Also set $g_{0}=0$. Like for $D$, specific choices for $g$, adapted to $(p^{(n)})_{n\in\mathbb{N}^+}$, will be considered in the next section. 

For each $(U_1,\ldots, U_s)\in \prod_{r=1}^s (\mathcal{I}^D_{r})^{g_r-g_{r-1}}$, set 
$$
B(U_1,\ldots, U_s)= \big \{\boldsymbol{i}\in \mathcal{I}^{\mathbb{N}^+}:\, \forall \, 1\le r\le s, \, \Pi^D_{r} (T^{g_{r-1}}\boldsymbol{i})\in [U_r]\big\}.
$$
Then,  
\begin{equation}\label{partition}
\mathcal F^D(g)=\{B(U_1,\ldots, U_s): \, (U_1,\ldots, U_s)\in \prod_{r=1}^s (\mathcal{I}^D_{r})^{g_r-g_{r-1}}\}
\end{equation}
is a partition of $\mathcal{I}^{\mathbb{N}^+}$. 

\begin{definition}\label{SN}
For $q \ge 0$ set 
\begin{align*}
S_{k}(q)&=\sum_{n=1}^k T_{W^{(n)}}(q)+\sum_{n=k+1}^{g_s} \tau^{D,n}_{r_n}(q)\quad (0\le k\le g_s),
\end{align*}
where $r_n$ is the index $r$ such that $g_{r-1}+1\le n\le g_{r}$.

\end{definition}
We have the following controls from above for $\mathbb{E} \big (\sum_{B\in\mathcal F^D(g)}\nu (B)^q\big )$, where we distinguish the cases $q>1$ and $q\le 1$. 
\begin{proposition}\label{tau} 
Fix $q \in \left(1,2\right]$ such that $\phi_{W^{(n)}}(q)<+\infty$ for all $n\ge 1$. One has 
\begin{align}\label{partf1}
\mathbb{E}\Big (\sum_{B\in \mathcal F^{D}(g)}\nu (B)^q\Big )\leq e^{-S_{g_s}(q)}\mathbb{E}((Y^{(g_s)})^q)+ M_{q,g_s}\sum_{k=g_1}^{g_s-1}  e^{-S_{k}(q)},
\end{align}
where $Y^{(g_s)}$ is any of the  $Y(U'_1,\ldots,U'_s)$, $(U'_1,\ldots,U'_s)\in \prod_{r=1}^s \mathcal{I}^{g_r-g_{r-1}}$, which are identically distributed, and $M_{q,g_s}=(\#\mathcal I)^2\sup_{g_1\le n\le g_s}\phi_{\widetilde W^{(n)}}(q)$, where $\widetilde W^{(n)}$ is defined in~\eqref{tildeW}. 
\end{proposition}
Note that under additional assumptions, the term $\mathbb{E}((Y^{(g_s)})^q)$ in \eqref{partf1} can itself be controlled thanks to Proposition~\ref{Moments1}(2). This will be used in the proof of Theorem~\ref{thm-2.4}(2).

\begin{proposition}\label{SN'}
For all  $q\in (0,1]$, one has 
\begin{align*}
\mathbb{E}\Big (\sum_{B\in \mathcal F^{D}(g)}\nu (B)^q\Big )\leq \min\Big (e^{-\sum_{n=1}^{\widetilde{g_s}}T_{W^{(n)}}(q)},\min_{g_1\le k\le g_s-1}e^{-S_{k}(q)}\Big ),
\end{align*}
where $\widetilde {g_s}$ is defined as in \eqref{Ntilde}. 
\end{proposition}

Before starting the proof, we note that for every $p\in\mathbb{N}$ and $U\in\mathcal I^p$, the probability distribution of the family of random vectors $(\widetilde W(Uv))_{v\in \mathcal{I}^*}$  does not depend on $U$. Thus, each such family generates a copy $\nu^{(U)}$ of a random inhomogeneous Mandelbrot measure $\nu^{p}$ (so that $\nu^{(\epsilon)}=\nu^{0}=\nu$), as well as the associated sequence of measures $(\nu^{(U)}_n)_{n\in\mathbb{N}^+}$, defined in the same way as  $(\nu_n)_{n\in\mathbb{N}^+}$ was defined, that is by uniformly distributing (with respect to the uniform measure on $(\Sigma,\mathcal{B}(\Sigma))$), the mass $Q^U(w)$ over each cylinder $\left[w\right]$ of generation $n$.

\begin{proof}[Proof of Proposition~\ref{tau}] For each $r\in \{1,\ldots, s\}$, we simply denote $\mathcal I^D_r$ by $\mathcal I_r$ (recall that $\mathcal I_1=\mathcal I$) and $\Pi^D_r$ by~$\Pi_r$.
Also, denote by $m$ the inhomogeneous Bernoulli product measure on $\mathcal{I}^{\mathbb{N}^+}$ associated to the probability vectors $(p^{(n)})_{n\ge1}$, that is the measure $\otimes_{n=1}^\infty (\sum_{i\in\mathcal I}p^{(n)}_i\delta_{i})$. Note that $m=\mathbb{E}(\nu)$ and for each $1\le r\le s$, ${\Pi_r}_*m$, the pushforward of $m$ on $\mathcal I_r^{\mathbb{N}^+}$ by $\Pi_r$, is the inhomogeneous  Bernoulli product measure on $\mathcal{I}_r^{\mathbb{N}^+}$associated to the probability vectors $(\Pi_rp^{(n)})_{n\ge1}$. The shift operation on $\mathcal I_r^{\mathbb{N}^+}$ is denoted by $T_r$. 

Set $j_r=g_r-g_{r-1}$. For each $(U_1,\ldots, U_s)\in  \prod_{r=1}^s \mathcal{I}_{r}^{g_r-g_{r-1}}$, write $U_{r}=U_{r,1}\cdots U_{r,j_r}$. By construction one has 
\begin{align}
\nonumber\mathbb{E}(\nu(B(U_1,\ldots, U_s)))&=m(B(U_1,\ldots, U_s))\\
\label{mU1Us}&=\prod_{r=1}^s {\Pi_r}_*m([U_r])=\prod_{r=1}^s \prod_{n=1}^{j_r} (\Pi_{r}p^{(g_{r-1}+n)})_{U_{r,n}},
\end{align}
and if this number is different from $0$, 
\begin{align*}
&\frac{\nu(B(U_1,\ldots, U_s))}{m(B(U_1,\ldots, U_s))} \\
&=\sum_{(U'_r)_{r=1}^s\in \prod_{r=1}^s \Pi_r^{-1} (\{U_r\})}\frac{\nu_{g_s}([U'_1U'_2\cdots U'_s])}{\prod_{r=1}^s {\Pi_r}_*m([U_r])}\cdot Y(U'_1\cdots U'_s)\\
&=\sum_{(U'_r)_{r=1}^s\in \prod_{r=1}^s \Pi_r^{-1} (\{U_r\})}\frac{\nu_{g_1}([U'_1])}{m([U_1])}\frac{\nu^{(U'_1)}_{g_2-g_1}([U'_2])}{{\Pi_2}_*m([U_2])}\cdots \frac{\nu^{(U'_1\cdots U'_{s-1})}_{g_s-g_{s-1}}([U'_s])}{{\Pi_s}_*m([U_s])}\cdot Y(U'_1\cdots U'_s)\\
&= \frac{\nu_{g_1}([U_1])}{m([U_1])}\sum_{(U'_r)_{r=2}^s\in \prod_{r=2}^s \Pi_r^{-1} (\{U_r\})}\frac{\nu^{(U'_1)}_{g_2-g_1}([U'_2])}{{\Pi_2}_*m([U_2])}\cdots \frac{\nu^{(U'_1\cdots U'_{s-1})}_{g_s-g_{s-1}}([U'_s])}{{\Pi_s}_*m([U_s])}\cdot Y(U'_1\cdots U'_s).
\end{align*}
Define, with $U'_1=U_1$, $Z(U'_1,\ldots,U'_s)=Y(U'_1\cdots U'_s)$, and for $2\le r\le s$
\begin{multline*}
Z(U'_1,\ldots,U'_{r-1})\\=\sum_{(U'_{t})_{t=r}^s\in \prod_{t=r}^s (\Pi_t)^{-1} (\{U_t\})}\frac{\nu^{(U'_1U'_2\cdots U'_{r-1})}_{g_r-g_{r-1}}([U'_r])}{{\Pi_r}_*m([U_r])}\cdots \frac{\nu^{(U'_1U'_2\cdots U'_{s-1})}_{g_s-g_{s-1}}([U'_s])}{{\Pi_s}_*m([U_s])}\cdot Z(U'_1\cdots U'_s).
\end{multline*}
One thus has 
\begin{equation}\label{muBrec}
\frac{\mu(B(U_1,\ldots, U_s))}{\prod_{r=1}^s {\Pi_r}_*m([U_r])} =\frac{\nu_{g_1}([U_1])}{m([U_1])} Z(U_1)=\frac{\nu_{g_1}([U_1])}{m([U_1])} Z(U'_1),
\end{equation}
and for $2\le r\le s$, 
\begin{equation}\label{Z}
Z(U'_1,\ldots,U'_{r-1})=\sum_{U'_{r}\in \Pi_r^{-1} (\{U_{r}\})}\frac{\nu^{(U'_1U'_2\cdots U'_{r-1})}_{g_r-g_{r-1}}([U'_r])}{{\Pi_r}_*m([U_r])} Z(U'_1,\ldots,U'_{r}).
\end{equation}
Note that the $Z(U'_1,\ldots,U'_{r-1})$ are identically distributed. So $\mathbb{E}(Z(U'_1,\ldots,U'_{r-1})^q)$ depends only on $(U_1,\ldots,U_{r-1})$. We denote this value by $\mathcal Z_q(U_1,\ldots,U_{r-1})$.
We are going to estimate $\mathcal Z_q(U_1,\ldots,U_{r-1})$ recursively. To do so, we fix $(U'_1,\ldots,U'_{r-1})$ and start by writing the term $\frac{\nu^{(U_1'U'_2\cdots U'_{r-1})}_{g_r-g_{r-1}}([U'_r])}{{\Pi_r}_*m([U_r])}$ in \eqref{Z} in its natural form of product of independent random variables. This requires some notation.

For $n\ge 1$, $j\in \Pi_r(\mathcal I)$, $i\in \Pi_r^{-1}(\{j\})$ and $v\in\mathcal I^{n-1}$,  we define
$$
(V_{r})^{(n)}_{i,j}(v) =
\begin{cases}
\displaystyle \frac{W^{(n)}_i(v)}{(\Pi_rp^{(n)})_{j}}=  \frac{p^{(n)}_i \widetilde W^{(n)}_{i}(v)}{(\Pi_rp^{(n)})_{j}}&\text{if }(\Pi_rp^{(n)})_{j}\neq 0\\
0&\text{otherwise},
\end{cases}
$$
and simply write $(V_{r})^{(1)}_{i,j}$ for $(V_{r})^{(1)}_{i,j}(\boldsymbol{\epsilon})$.  For $j\in \mathcal I_r$, the random vectors $\big ((V_{r})^{(n)}_{i,j}(v)\big )_{i\in \Pi_r^{-1} (\{j\})}$, $v\in\mathcal I^{n-1}$, are identically distributed and we denote by $(V_r)^{(n)}_{j}$ one of these vectors. For all $j\in \Pi_r(\mathcal I)$, by construction one has 
\begin{align*}
\mathbb{E}\Big (\sum_{i\in \Pi_r^{-1} (\{j\})}(V_r)^{(n)}_{i,j}\Big )=1.
\end{align*}

Write $U_{r,1}\cdots U_{r,j_r}=u_1\cdots u_{j_r}$ to lighten the notation, as well as  $U'_r=u'_1\cdots u'_{j_r}$. Also, set $(\widetilde V_r)_{i,j}^{(n)}(v)= (V_r)_{i,j}^{(g_{r-1}+n)}(U'_1,\ldots,U'_{r-1}v)$ for all $1\le n\le j_r$ and $v\in\mathcal{I}^*$. It is easily seen that 
$$
\frac{\nu^{(U'_1U'_2\cdots U'_{r-1})}_{g_r-g_{r-1}}([U'_r])}{{\Pi_r}_*m([U_r])}=\prod_{n=1}^{j_r}(\widetilde V_r^{(n)})_{u'_n,u_n}(u'_1\cdots u'_{n-1}).
$$
Hence, remembering \eqref{Z} and denoting  $Z(U'_1,\ldots,U'_{r-1})$ by $X_{1,\ldots,j_r}(u_1,\ldots,u_{j_r})$, we get 
$$
X_{1,\ldots,j_r}(u_1,\ldots,u_{j_r})=\sum_{u'_{1}\in \Pi_r^{-1} (\{u_{1}\})}(\widetilde V_r^{(1)})_{u'_1,u_1}\cdot X^{u'_1}_{2,\ldots,j_r}(u_2,\ldots,u_{j_r}),
$$
where
$$
X^{u'_1}_{2,\ldots,j_r}(u_2,\ldots,u_{j_r})=\sum_{(u'_{n})_{n=2}^{j_r}\in \Pi_r^{-1} (\{(u_{n})_{n=2}^{j_r}\})} \Big (\prod_{n=2}^{j_r}(\widetilde V_r^{(n)})_{u'_n,u_n}(u'_1\cdots u'_{n-1})\Big ) \cdot  Z(U'_1,\ldots,U'_{r}).
$$
Now we start like  Kahane in \cite{ref26} to estimating  the $L^q$ moment of Mandelbrot martingales: we use the subadditivity of $x\ge 0 \mapsto x^{\frac{q}{2}}$ ($q\in (1,2]$). This yields, dropping the dependence on $(u_1,\ldots,u_{j_r})$ and $(u_1,\ldots,u_{j_r})$ in $X_{1,\ldots,j_r}(u_1,\ldots,u_{j_r})$ and $X^{u'_1}_{2,\ldots,j_r}(u_2\cdots u_{j_r})$ respectively, 
\begin{align}
\nonumber \mathbb{E} \left(X_{1,\ldots,j_r}^q\right) &\leq  \mathbb{E}\left[\left(\sum_{u'_1\in (\Pi_r)^{-1}(\{u_1\})} \left((\widetilde V_{r})^{(1)}_{u'_1,u_1})^2\right)^{\frac{q}{2}} \left(X^{u'_1}_{2,\ldots,j_r}\right)^{\frac{q}{2}}\right)^2\right] \\\nonumber&= \mathbb{E}\left[\sum_{u'_1\in (\Pi_r)^{-1}(\{u_1\})}\left((\widetilde V_{r})^{(1)}_{u'_1,u_1}\right )^{q} \left(X^{u'_1}_{2,\ldots,j_r}\right)^{q}\right] \\& 
\label{secondterm}\quad +\mathbb{E}\left[\sum_{u'_1 \neq v'_1\in (\Pi_r)^{-1}(\{u_1\})} \left((\widetilde V_{r})^{(1)}_{u'_1,u_1}\right)^{\frac{q}{2}} \left(X^{u'_1}_{2,\ldots,j_r}\right)^{\frac{q}{2}} \left((\widetilde V_{r})^{(1)}_{v'_1,u_1}\right)^{\frac{q}{2}} \left(X^{v'_1}_{2,\ldots,j_r})\right)^{\frac{q}{2}} \right].
\end{align}
By construction, one has that $(i)$ the random variables  $\left ((\widetilde V_{r})^{(1)}_{u'_1,u_1}, (\widetilde V_{r})^{(1)}_{v'_1,u_1}\right )$, $X^{u'_1}_{2,\ldots,j_r}$, and $X^{v'_1}_{2,\ldots,j_r}$ are mutually independent;  $(ii)$ $X^{u'_1}_{2,\ldots,j_r}$ and $X^{v'_1}_{2,\ldots,j_r}$ are identically distributed and of expectation 1; $(iii)$ $(\widetilde V_{r})^{(1)}_{u'_1,u_1}\le \widetilde W^{(g_{r-1}+1)}_{u'_1}$.  Since $q/2\le 1$, Jensen's inequality yields $\mathbb E( (X^{u'_1}_{2,\ldots,j_r})^{q/2})\le 1$. Then, the Cauchy-Schwarz inequality applied to  the right-hand side of the inequality $\mathbb{E}\left(\big ((\widetilde V_{r})^{(1)}_{u'_1,u_1}\big )^{\frac{q}{2}}\big ((\widetilde V_{r})^{(1)}_{v'_1,u_1}\big )^{\frac{q}{2}}\right )\le \mathbb{E}\left(\big (\widetilde W^{(g_{r-1}+1)}_{u'_1}\big )^{\frac{q}{2}} \big (\widetilde W^{(g_{r-1}+1)}_{v'_1}\big )^{\frac{q}{2}}\right)$ implies 
$$
\sum_{u'_1 \neq v'_1\in (\Pi_r)^{-1}(\{u_1\})} \mathbb{E}\left(\big ((\widetilde V_{r})^{(1)}_{u'_1,u_1}\big )^{\frac{q}{2}}\big ((\widetilde V_{r})^{(1)}_{v'_1,u_1}\big )^{\frac{q}{2}}\right )\le \sum_{u'_1 \neq v'_1\in (\Pi_r)^{-1}(\{u_1\})} \sum_{i\in\mathcal I}  \mathbb{E}\left(\big (\widetilde W^{(g_{r-1}+1)}_{i}\big )^{q}\right )
$$
so that the term in \eqref{secondterm} is bounded from above by 
$$
 \left(\sup_{u_1\in \mathcal I_r}\big (\# (\Pi_r)^{-1}(\{u_1\})\big )\right )^2 \sum_{i\in\mathcal I}  \mathbb{E}\left(\big (\widetilde W^{(g_{r-1}+1)}_{i}\big )^{q}\right )\le (\#\mathcal I)^2\phi_{\widetilde W^{(g_{r-1}+1)}}(q)\le M_{q,g_s}. 
$$
Thus, setting 
$$
T_{(V_r)^{(n)}_{j}}(q) = - \log \mathbb{E}\Big (\sum_{i\in \Pi_r^{-1}(\{j\})} ((V_r)^{(n)}_{i,j})^q\Big ),
$$
we get 
\begin{align*}\label{Iter}
\mathbb{E} \left(X_{1,\ldots,j_r}(u_1\cdots u_{j_r})^q\right) &\le M_{q,g_s}+ \sum_{u'_1\in (\Pi_r)^{-1}(\{u_1\})} \mathbb{E}\left[\left((\widetilde V_{r})^{(1)}_{u'_1,u_1}\right )^{q}\right] \mathbb{E}\left[X^{u'_1}_{2,\ldots,j_r}(u_2\cdots u_{j_r})^{q}\right]\\
\nonumber& =M_{q,g_s}+ \exp\left (-T_{(\widetilde V_{r})^{(1)}_{u_1}}(q)\right )\cdot   \mathbb{E}\left[X^{u'_1}_{2,\ldots,j_r}(u_2\cdots u_{j_r})^{q}\right].
\end{align*}
We can iterate the previous estimates on the expectations $\mathbb{E}\left[X^{u'_1}_{2,\ldots,j_r}(u_2\cdots u_{j_r})^{q}\right]$, then on those they lead to, and so on... by using recursive relations of the form 
\begin{equation}\label{recursioncruciale}
X^{u'_1\cdots u'_{j-1}}_{j,\ldots,j_r}(u_{j}\cdots u_{j_r})= \sum_{u'_j\in\Pi_r^{-1}(\{u_j\})} (\widetilde V_{r})^{(j)}_{u'_j,u_j}X^{u'_1\cdots u'_{j}}_{j+1,\ldots,j_r}(u_{j+1}\cdots u_{j_r}),
\end{equation}
with  $X^{u'_1\cdots u'_{j_r}}_{j_r+1,\ldots,j_r}(u_{j_r+1}\cdots u_{j_r})=Z(U'_1,\ldots,U'_{r})$.

This yields, setting $S^{(r)}_{U_{r,1}\cdots U_{r,j}}(q)=\sum_{i=1}^jT_{(\widetilde V_r)^{(i)}_{U_{r,i}}}(q) $ (this definition being extended to the case $r=1$ for the estimate starting from \eqref{muU1Us} below) : 
\begin{equation}\label{Iter}
\mathcal Z_q(U_1,\ldots,U_{r-1})\leq M_{q,g_s} \left(1+\sum_{j=1}^{j_r-1}e^{-S^{(r)}_{U_{r,1}\cdots U_{r,j}}(q)} \right) + e^{-S^{(r)}_{U_{r,1}\cdots U_{r,j_r}}(q)}\mathcal Z_q(U_1,\ldots,U_{r}) .
\end{equation}
One deduces from  \eqref{muBrec}, \eqref{Z} and \eqref{Iter} used recursively from $r=2$ to $r=s-1$ that 
\begin{align}
\label{muU1Us}&\mathbb{E}\left[\left(\frac{\mu(B(U_1,\ldots, U_s))}{m(B(U_1,\ldots, U_s))} \right)^q\right]\\
\nonumber &=\frac{\mathbb{E}(\nu_{g_1}([U_1])^q)}{m([U_1])^q} \mathcal Z_q(U_1)\\
\nonumber &\le \frac{\mathbb{E}(\nu_{g_1}([U_1])^q)}{m([U_1])^q}M_{q,g_s} \Big (1+ \sum_{j=1}^{j_2-1} e^{-S^{(2)}_{U_{2,1}\cdots U_{2,j}}(q)}\Big )\\&\quad \quad \quad \quad \quad \nonumber\quad +\frac{\mathbb{E}(\nu_{g_1}([U_1])^q)}{m([U_1])^q} e^{-S^{(2)}_{U_{2,1}\cdots U_{2,j_2}}(q)}\mathcal Z_q(U_1,U_{2}) \text{ (if $s\ge 2$)}\\
\nonumber&= M_{q,g_s} \Big (e^{-S^{(1)}_{U_{1}}(q)} + e^{-S^{(1)}_{U_{1}}(q)}\sum_{j=1}^{j_2-1} e^{-S^{(2)}_{U_{2,1}\cdots U_{2,j}}(q)}\Big )+  e^{-S^{(1)}_{U_{1}}(q)-S^{(2)}_{U_{2}}(q)} \mathcal Z_q(U_1,U_{2})\\
\nonumber&\vdots\\
\nonumber &\le M_{q,g_s}\Big (e^{-S^{(1)}_{U_{1}}(q)}+ \sum_{r=2}^{s-1} e^{-\sum_{r'=1}^{r-1} S^{(r')}_{U_{r'}}(q)} \sum_{j=1}^{j_r}  e^{-S^{(r)}_{U_{r,1}\cdots U_{r,j}}(q)}\Big) \\
\nonumber&\quad+\mathbf{1}_{\{s\ge 2\}}M_{q,g_s} e^{-\sum_{r=1}^{s-1} S^{(r)}_{U_{r}}(q)} \sum_{j=1}^{j_s-1}  e^{-S^{(s)}_{U_{s,1}\cdots U_{s,j}}(q)} +e^{-\sum_{r=1}^{s} S^{(r)}_{U_{r}}(q)}  \mathbb{E}\big ((Y^{(g_s)})^q\big ).
\end{align}
Denote by $T_q(U_1,\cdots,U_s)$ the right hand side of the last inequality. Also, for $1\le r\le s$ and $g_{r-1}+1\le n\le g_r$ set $\mathcal J_n=\mathcal I_r$, and for $u\in\mathcal J_n$ set $a_u=\big ((\Pi_rp^{(n)})_u\big )^q$ and $b_u=e^{-T_{V^{(n)}_{r,u}}(q)}$.  Due to \eqref{mU1Us} and the last inequality,  one has 
\begin{align*}
\mathbb{E}\Big (\sum_{B\in \mathcal F^{D}(g)}\nu (B)^q\Big )&=\sum_{(U_1,\ldots, U_s)\in  \prod_{r=1}^s \mathcal{I}_{r}^{g_r-g_{r-1}}}\mathbb{E}\big (\nu(B(U_1,\ldots, U_s))^q\big )\\
&\le \sum_{(U_1,\ldots, U_s)\in  \prod_{r=1}^s \mathcal{I}_{r}^{g_r-g_{r-1}}} m(B(U_1,\ldots, U_s))^q\cdot T_q(U_1,\cdots,U_s)\\
&=M_{q,g_s}\sum_{k=g_1}^{g_s-1}\sum_{(u_n)_{n=1}^{g_s}\in\prod_{n=1}^{g_s}\mathcal J_n}\Big (\prod_{n=1}^{g_s}a_{u_n}\Big )\Big (\prod_{n=1}^kb_{u_n}\Big)\\
&\quad\quad \quad\quad \quad\quad + \mathbb{E}\big ((Y^{(g_s)})^q\big )\sum_{(u_n)_{n=1}^{g_s}\in\prod_{n=1}^{g_s}\mathcal J_n}\Big (\prod_{n=1}^{g_s}a_{u_n}\Big )\Big (\prod_{n=1}^{g_s}b_{u_n}\Big) \\
&=M_{q,g_s}\sum_{k=g_1}^{g_s-1}\Big (\prod_{n=1}^k\Big(\sum_{u\in\mathcal J_n}a_ub_u\Big) \Big)\Big (\prod_{n=k+1}^{g_s}\Big(\sum_{u\in\mathcal J_n}b_u\Big) \Big)\\
&\quad\quad \quad\quad \quad\quad +\mathbb{E}\big ((Y^{(g_s)})^q\big )\prod_{n=1}^{g_s}\Big(\sum_{u\in\mathcal J_n}a_ub_u\Big).
\end{align*}
Recalling \eqref{taunq} and noticing that by construction one has
\begin{equation}\label{Lqfibre}
 \sum\limits_{j\in \mathcal I_r}\big ((\Pi_rp^{(n)})_j\big )^q e^{-T_{V^{(n)}_{r,j}}(q) }=e^{-T_{W^{(n)}}(q)},
\end{equation}
we get the desired conclusion.\end{proof}

\begin{proof}[Proof of Proposition~\ref{SN'}] With the notation of the previous proof, fix $2\le r\le s$ as well $ 0\le j\le j_r$. The situation is much simpler than when $q\ge 1$ because one can simply use the subbaditivity of $x\ge 0\mapsto x^q$ to get, instead of \eqref{Iter}, using the definition \eqref{recursioncruciale} and the convention that $S^{(r)}_{\emptyset}=0$ in the case that $j=0$,
\begin{equation*}
\mathcal Z_q(U_1,\ldots,U_{r-1})\leq e^{-S^{(r)}_{U_{r,1}\cdots U_{r,j}}(q)}\mathbb{E}((X^{u'_1,\ldots,u'_j}_{j+1,\ldots, j_r})^q)\le e^{-S^{(r)}_{U_{r,1}\cdots U_{r,j}}(q)},
\end{equation*}
since $\mathbb{E}(X^{u'_1,\ldots,u'_j}_{j+1,\ldots, j_r})=1$. This implies that 
$$
\mathbb{E}\left[\left(\frac{\nu(B(U_1,\ldots, U_s))}{m(B(U_1,\ldots, U_s))} \right)^q\right]\le e^{-\sum_{r'=1}^{r-1} S^{(r')}_{U_{r'}}(q)} e^{-S^{(r)}_{U_{r,1}\cdots U_{r,j}}(q)},
$$
and summing over $(U_1,\ldots, U_s)$ yields, for $k=g_{r-1}+j$,
$$
\mathbb{E}\Big (\sum_{B\in \mathcal F^{D}(g)}\nu (B)^q\Big )\leq e^{-S_{k}(q)}. 
$$
The inequality
$$
\mathbb{E}\Big (\sum_{B\in \mathcal F^{D}(g)}\nu (B)^q\Big )\leq e^{-\sum_{n=1}^{\widetilde{g_s}}T_{W^{(n)}}(q)}
$$
follows from writing that 
\begin{multline*}
\nu (B(U_1,\ldots,U_s))\\=\sum_{(U'_r)_{r=1}^s\in \prod_{r=1}^s \Pi_r^{-1} (U_r)}\sum_{U'\in \mathcal{I}^{\widetilde {g_s}-g_s}} \nu_{\widetilde{g_s}}([U'_1U'_2\cdots U'_sU'])\cdot Y(U'_1\cdots U'_sU'),
\end{multline*}
then using again that $x\ge 0\mapsto x^q$ is subbaditive, taking the expectation using the independences and the branching property, and the fact that $\mathbb{E}(Y(U'_1\cdots U'_sU')^q)\le 1$. 
\end{proof}

\subsection{Proof of Theorem~\ref{thm-2.4}}\label{pfthmdim} Recall that  in Section~\ref{thmdim} we introduced the sequences $(D(N))_{N\ge 1}$, $(s(N))_{N\ge 1}$ and $(g(N)=(g_1(N),\ldots,g_{s(N)}(N)))_{N\ge 1}$ associated with $\boldsymbol{p}$.  This makes it possible to associate, to each $N\ge 1$, the partition $\mathcal F^{D(N)}(g(N))$ of $\Sigma$ defined in~\eqref{partition}, and that we simply denote by $\mathcal F^{D}_N(g)$.  For each $\boldsymbol{i}\in\Sigma$, the element of $\mathcal F_N^{D}(g)$ which contains $\boldsymbol{i}$ is denoted by $B_N(\boldsymbol{i})$. 

We are going to apply Propositions~\ref{tau} and~\ref{SN'} with these partition functions, and for each $N\ge 1$ and $1\le k\le g_{s(N)}(N)$,  the associated function $S_k(\cdot)$ considered in Definition~\ref{SN} is now denoted $S_{N,k}(\cdot)$. Note that the quantity $H_{N,k}$ introduced in Definition~\ref{HN} equals $S_{N,k}'(1)$.

The proof of Theorem~\ref{thm-2.4} will be deduced from the following result for $\nu$ on $\Sigma$. 
\begin{theorem}\label{dimHmu} Suppose that the assumptions of Theorem~\ref{sufcond2}(1) hold, so that $\nu$ is non-degenerate. Let 
$$
\underline{d}(\nu)=\liminf_{N\to +\infty} d_N \text{ and } \overline{d}(\nu)=\limsup_{N\to +\infty} d_N. 
$$
where $d_N$ is defined as in \eqref{uN0}.  
\begin{enumerate}
\item Suppose that $\liminf_{N\to+\infty} N^{-1}\sum_{n=1}^NH(W^{(n)})=0$. With probability~1, conditional on $\{\nu\neq 0\}$, or $\nu$-almost every~$\boldsymbol{i}$ one has 
$$
\liminf_{N\to +\infty} \frac{\log(\nu(B_N(\boldsymbol{i})))}{-N}= \underline{d}(\nu)=0 \text{ and }\limsup_{N\to +\infty} \frac{\log(\nu(B_N(\boldsymbol{i})))}{-N}\le \overline{d}(\nu).
$$

\item Suppose that $\liminf_{N\to+\infty} N^{-1}\sum_{n=1}^NH(W^{(n)})>0$ and $\sup_{n\ge 1}\phi_{\widetilde W^{(n)}}(q)<+\infty$ for some $q \in \left(1,2\right]$.
With probability~1, conditional on $\{\nu\neq 0\}$, for $\nu$-almost every~$\boldsymbol{i}$ one has 
$$
\lim_{N\to +\infty}\left|\frac{\log(\nu(B_N(\boldsymbol{i})))}{-N}-d_N\right|=0,
$$
hence
$$
\begin{cases}
\liminf_{N\to +\infty}\displaystyle \frac{\log(\nu(B_N(\boldsymbol{i})))}{-N}= \underline{d}(\nu)\\
\limsup_{N\to +\infty}\displaystyle \frac{\log(\nu(B_N(\boldsymbol{i})))}{-N}= \overline{d}(\nu)
\end{cases}.
$$
\end{enumerate}
\end{theorem}

The proof of Theorem~\ref{dimHmu} will use the following lemma. 
\begin{lemma} \label{BC}
Let $\rho$ be a positive and finite Borel measure on $\mathcal{I}^{\mathbb{N}^+}$. Let $(\mathcal F_N)_{N\in\mathbb{\mathbb{N}^+}}$ be a sequence of partitions of $\Sigma$. For $\boldsymbol{i}\in \Sigma$ and $N\in\mathbb{\mathbb{N}^+}$, denote by $B_N(\boldsymbol{i})$ the  element of $\mathcal F_N$ which contains $\boldsymbol{i}$. Also let $(v_N)_{N\in\mathbb{\mathbb{N}^+}}\in \mathbb{R}^{\mathbb{N}^+}$. Suppose that for all $\eta> 0$ there exists $q > 1$ such that  $\sum_{N \geq 1} e^{N(q-1)(v_N-\eta)} \sum_{B\in\mathcal F_N} \rho(B)^q < +\infty$. Then $\liminf_{N \rightarrow \infty} \left(\frac{\log(\rho(B_N(\boldsymbol{i})))}{-N} - v_N \right) \geq 0$ for $\rho$-almost every $\boldsymbol{i}$. Similarly, if for all $\eta> 0$ there exists $0 < q < 1$ such that $\sum_{N \geq 1} e^{N(q-1)(v_N+\eta)} \sum_{B\in\mathcal F_N} \rho(B)^q < +\infty$. Then $\limsup_{N \rightarrow \infty} \left(\frac{\log(\rho(B_N(\boldsymbol{i})))}{-N} - v_N \right) \leq 0$ for $\rho$-almost every $\boldsymbol{i}$. 
\end{lemma}
We give a proof for the sake of completeness.
\begin{proof}
Let us prove the first assertion. One has 
\begin{align*}
\begin{split}
&\rho \left(\left\{\boldsymbol{i}\in \mathcal{I}^{\mathbb{N}^+}, \ \frac{\log(\rho(B_N(\boldsymbol{i})))}{-N} - v_N \leq -\eta\right\}\right) \\&= \rho \left(\left\{\boldsymbol{i}\in \mathcal{I}^{\mathbb{N}^+},  \ \rho(B_N(\boldsymbol{i}))^{q-1} \geq e^{-N(q-1)(v_N-\eta)}\right\}\right) \\&\leq e^{N(q-1)(v_N-\eta)} \sum_{B\in\mathcal F_N} \rho(B)^q . 
\end{split}
\end{align*}
Under our assumption, by the Borel-Cantelli lemma one deduces that for  $\rho$-a.e.  $\boldsymbol{i}$, one has $\frac{\log(\rho(B_N(\boldsymbol{i})))}{-N} - v_N  \geq -\eta$ for $N$ large enough. Consequently, $$\liminf_{n \rightarrow \infty} \left(\frac{\log(\rho(B_N(\boldsymbol{i})))}{-N} - v_N \right) \geq -\eta,$$
which yields the desired conclusion. 
The other inequality is proven similarly. 
\end{proof}

It will be useful for many proofs to come,  to note that since the $|a_{i,k}|$ are uniformly bounded away from $0$, there are positive constants $\Lambda_a$ and $\Lambda'_a$ such that, independently of $(p^{(n)})_{n\in\mathbb{N}^+} $ and $N\ge 1$, one has 
\begin{equation}\label{Lambdaa}
\Lambda'_a N\le g_1(N)\le g_s(N)\le \Lambda_a N.
\end{equation}
We can take 
\begin{equation}\label{LambdaLambda'}
\begin{cases}
\Lambda_a=1+\max\{|\log(|a_{i,k}|)|^{-1}:\, i\in\mathcal{I},\, 1\le k \le d\}\\
\Lambda'_a=\min \{|\log(|a_{i,k}|)|^{-1}:\, i\in\mathcal{I},\, 1\le k \le d\}.
\end{cases} 
\end{equation}
We will denote $\widetilde{g_{s(N)}(N)}$ by  $\widetilde{g}_{s(N)}(N)$.

\begin{proof}[Proof of Theorem~\ref{dimHmu}] It is convenient to first prove (2). Recall \eqref{tildeW}. It is direct to see that $\sup_{n\ge 1}\phi_{\widetilde W^{(n)}}(q)<+\infty$ implies that $\sup_{n\ge 1}\phi_{W^{(n)}}(q)<+\infty$. We leave it to the reader to check that there is an open  neighborhood $\mathcal U$ of $1$ in $[0,q]$ over which the second derivatives of the mappings $T_{W^{(n)}}$ and  $\tau^{D,n}_{r}$ are bounded independently of $(p^{(n)})_{n\in\mathbb{N}^+}$ and  $(D(N))_{N\in\mathbb{N}^+}$. Thus, noting that $(\tau^{D(N),n}_{r})'(1)=h\big (\Pi^{D(N)}_{r}p^{(n)}\big) $ and $T_{W^{(n)}}'(1)=H(W^{(n)})$, we deduce that there exists a constant $M>0$ depending on $(\widetilde W^{(n)})_{n\ge 1}$ and $q$ only such that for all $q'\in \mathcal U$ such that $q'>1$ one has both $\tau^{D(N),n}_{r}(q')\ge h(\Pi^D_{r}p^{(n)}) (q'-1)-M(q'-1)^2$ and $T_{W^{(n)}}(q')\ge H(W^{(n)})(q'-1)-M(q'-1)^2$. Moreover, $\mathcal U$ can be taken so that for all $q'\in \mathcal U$, the conclusions of Proposition~\ref{Moments1}(2) hold. 

Fix $\eta>0$. Fix $\epsilon\in (0,\log(\#\mathcal I))$ and $N_\epsilon\in\mathbb{N}^+$ such that $\sum_{n=1}^NH(W^{(n)})\ge N\epsilon$ for all $N\ge N_\epsilon$. Then consider $q'>1$ close enough to $1$ in $\mathcal U$ so that $(q'-1)M \Lambda_a\frac{\log(\#\mathcal{I})}{\epsilon}<\eta/4$. Recalling Definitions~\ref{SN} and~\ref{HN}, we get 
\begin{align}
\label{controlSN} &-S_{N,k}(q')\le (q'-1)^2 M g_s(N) -(q'-1)H_{N,k} \text{ for $g_1(N)\le k\le g_s(N)$}\\
\nonumber
\text{and }&\\
\label{controlSN2}& -\sum_{n=g_s(N)+1}^{ \widetilde{g}_s(N)}T_{W^{(n)}}(q')\le  (q'-1)^2 M ( \widetilde{g}_s(N)-g_s(N))  -(q'-1)\sum_{n=g_s(N)+1}^{ \widetilde{g}_s(N)}H(W^{(n)}).
\end{align}

Now, recall Proposition~\ref{tau} and Proposition~\ref{Moments1}(2). Our assumption $\sup_{n\ge 1}\phi_{\widetilde W^{(n)}}(q)<+\infty$ implies that $M_{q',g_s(N)}$ is uniformly bounded by a constant $M_{q'}$. Moreover, \eqref{Lambdaa} and the facts that $\widetilde g_s(N)\le \frac{\log(\#\mathcal{I})}{\epsilon} g_s(N)$ and $g_s(N)\le \Lambda_aN$ imply that $\widetilde g_s(N)\le \Lambda_a\frac{\log(\#\mathcal{I})}{\epsilon}N$. Thus,  for $N$ large enough we get 
\begin{align*}
&\mathbb{E}\Big (\sum_{B\in \mathcal F^{D}_N(g)}\nu (B)^{q'}\Big )\\
&\le M_{q'} e^{N(q'-1)\eta/4} \sum_{k=g_1(N)}^{g_s(N)-1}e^{-(q'-1)H_{N,k}}\\
&\quad + e^{N(q'-1)\eta/4} e^{-(q'-1)H_{N,g_s(N)}}  \frac{Cg_s(N)^{q'} e^{cg_s(N)(q'-1)^2}}{(1-e^{-\frac{(q'-1)}{4q'}\epsilon})^{q'}}B(g_{s}(N),q')\\
&= M_{q'} e^{N(q'-1)\eta/4} \sum_{k=g_1(N)}^{g_s(N)-1}e^{-(q'-1)H_{N,k}}\\
&\quad + e^{N(q'-1)\eta/4} \frac{Cg_s(N)^{q'} e^{cg_s(N)(q'-1)^2}}{(1-e^{-\frac{(q'-1)}{4q'}\epsilon})^{q'}}\max \left (e^{-(q'-1)H_{N,g_s(N)}}, e^{-(q'-1)\sum_{n=1}^{ \widetilde{g}_s(N)}H(W^{(n)})}\right ).
\end{align*}
We can also suppose that $q'-1$ is small enough so that $2\frac{Cg_s(N)^{q'} e^{cg_s(N)(q'-1)^2}}{(1-e^{-\frac{(q'-1)}{4q'}\epsilon})^{q'}}\le e^{N(q'-1)\eta/4}$ for $N$ large enough. Then, by definition of $d_N$, since $q'-1>0$, each term contributing to  the sum in the right-hand side of the last equality is dominated by $e^{-N(q'-1)(d_N -3\eta/4)}$, which yields 
$$
\mathbb{E}\Big (\sum_{B\in \mathcal F^{D}_N(g)}\nu (B)^{q'}\Big )\le (g_s(N)-g_1(N)+1) e^{-N(q'-1)(d_N-3\eta/4)}. 
$$
Consequently,
$$
\mathbb{E}\Big (\sum_{N\ge 1}e^{N(q'-1)(d_N-\eta)}\sum_{B\in \mathcal F^{D}_N(g)}\nu (B)^{q'}\Big )<+\infty.
$$
It follows that with probability~1, conditional on $\nu\neq 0$,  for all $\eta>0$, the series inside the above expectation converges. Using Lemma~\ref{BC}, we deduce that  $\liminf_{N\to +\infty} \frac{\log(\nu(B_N(\boldsymbol{i})))}{-N}-d_N\ge 0$ for $\nu$-almost every $ \boldsymbol{i}$. 

Next fix $\eta>0$ and consider $q''\in(0,1)$ close enough to $1$ in $\mathcal U$ so that $(1-q'')M \Lambda_a<\eta/4$.  Note that \eqref{controlSN} and \eqref{controlSN2} still hold. We then deduce from Proposition~\ref{SN'}, the definition of $d_N$ and the fact that $q''-1<0$ that
\begin{align*}
\mathbb{E}\Big (\sum_{B\in \mathcal F^{D}_N(g)}\nu (B)^{q''}\Big )\le e^{N(q''-1)\eta/4} e^{-(q''-1)d_N},
\end{align*}
which implies 
\begin{equation}\label{controlq''}
\mathbb{E}\Big (\sum_{N\ge 1}e^{N(q''-1)(d_N-\eta)}\sum_{B\in \mathcal F^{D}_N(g)}\nu (B)^{q''}\Big )<+\infty.
\end{equation}
Lemma~\ref{BC} then yields that with probability~1, conditional on $\nu\neq 0$,  for $\nu$-almost every~$ \boldsymbol{i}$,  $\limsup_{N\to +\infty} \frac{\log(\nu(B_N(\boldsymbol{i})))}{-N}-d_N\le 0$.
\begin{remark}\label{remmoments}
Note that to get \eqref{controlq''}, we did not use the assumption of Proposition~\ref{Moments1}(2). 
\end{remark}
Finally, $\liminf_{N\to +\infty}\Big |\frac{\log(\nu(B_N(\boldsymbol{i})))}{-N}-d_N\Big|=0$, for $\nu$-almost every $ \boldsymbol{i}$, hence the desired conclusion holds. 

\noindent (1) Due to Remark~\ref{remmoments} this follows from the same argument as in item (2).\end{proof}

Next we prove Theorem~\ref{thm-2.4}. We will use the following proposition, which is a consequence of the strong law of large numbers applied, for each $1\le k\le d$, to the sequence of uniformly bounded and independent random variables $X_n(\omega,\boldsymbol{i})=\log(|a_{i_n,k}|)$ with respect to the Peyri\`ere measure  defined on $(\Omega\times\Sigma,\mathcal G_n\otimes\mathcal B(\Sigma))$ as
\begin{equation}\label{Peyriere}
\mathcal Q(\mathrm{d} \omega,\mathrm{d}\boldsymbol{i})=\mathbb{P}(\mathrm{d\omega})\nu_\omega(\mathrm{d}\boldsymbol{i}).
\end{equation}

\begin{proposition}\label{Lyap}With probability~1, conditional on $\nu\neq 0$, for  $\nu$-almost every $\boldsymbol i\in\mathcal{I}^{\mathbb{N}^+}$, for all $1\le k\le d$, one has 
$$
\lim_{N\to +\infty} \Big |\chi_k(\widehat{\boldsymbol{p}}_N)+\frac{1}{N}\sum_{n=1}^N\log(|a_{i_n,k}|)\Big |
=0.
$$
\end{proposition}

\begin{proof}[Proof of Theorem~\ref{thm-2.4}]We start with item (2).  Recall the sets $(A_r)_{1\le r\le s}=(A_r(N))_{1\le r\le s(N)}$ defined in Section~\ref{thmdim}. For each $B=B(U_1,\ldots,U_s)\in \mathcal F_N^{D(N)}(g)$, let $Q_B$ be the parallelepiped $\prod_{r=1}^s \big (\pi^{A_r}\circ f_{U'_1U'_2\cdots U'_r}([0,1]^d)\big)$, where we recall that $\pi^{A_r}$ is the orthogonal projection on $\mathbb{R}^{A_r}$, and $(U'_1,\ldots ,U'_s)$ is any element of $\mathcal U(B)=\prod_{r=1}^s\Pi_r^{-1}(U_r)$,  $\Pi_r$ standing for $\Pi^{D(N)}_r$ (the independence with respect to $(U'_1,\ldots,U'_s)$ comes from the fact that $\pi^{A_r}=\pi^{A_r}\circ \pi^{D_r}$ and the definition of the sets $D_r(N)$). 

Note that by construction, the sets $Q_B$ have pairwise disjoint interiors. Also, for each $B\in  \mathcal F_N^{D}(g)$, one has  $B=\bigcup_{w\in\mathcal U(B)}[w]$, $\pi(B)\subset Q_B$,  and $K_\omega\cap Q_B=\bigcup_{w\in \mathcal U(B)}K_\omega\cap f_w([0,1]^d)$.

By Proposition~\ref{mursnonchargés}, the boundaries of the sets $f_w((0,1)^d)$ have $0$ $\mu$-mass, and $\nu([w])=\mu(f_w((0,1)^d))$ for all $w\in\Sigma^*$. This implies that $\nu(B)=\sum_{w\in\mathcal U(B)}\mu(f_w((0,1)^d))=\mu(\bigcup_{w\in \mathcal U(B)}f_w([0,1]^d))$, since $\nu(B)=\sum_{w\in\mathcal U_B} \nu([w])$ and the sets  $f_w([0,1]^d)$ have pairwise disjoint interiors. Consequently, since $\mu$ is supported on $K_\omega$ and $K_\omega\cap Q_B=\bigcup_{w\in \mathcal U(B)}K_\omega\cap f_w([0,1]^d)$, we conclude that  $\nu(B)=\mu(Q_B)$. Moreover, for $\mu$-almost every $z\in K$, for every $N\in\mathbb{N}^+$, there is a unique element $B\in \mathcal F^{D}_N(g)$ such that $z\in \mathrm{int}(Q_B)$. This is due to Proposition~\ref{mursnonchargés} again. Denote this  $Q_B$ by $Q_N(z)$. Theorem~\ref{dimHmu} implies that  for $\mu$-almost every $z\in K$ one has 
$$
\liminf_{N\to +\infty} \frac{\log(\mu(Q_N(z)))}{-N}= \underline{d}(\mu) \text{ and }\limsup_{N\to +\infty} \frac{\log(\mu(Q_N(z)))}{-N}= \overline{d}(\mu).
$$
 Also, due to Proposition~\ref{Lyap}, and the definition of the $g_r(N)$, for $\mu$-almost every $z\in K$, if $r(Q_N(z))$ and $R(Q_N(z))$ stand for the smallest and the biggest side of $Q_N(z)$, one has 
$$
\log(r(Q_N(z))))\sim \log(R(Q_N(z)))\sim -N \text{ at $+\infty$}.
$$
Lemma~\ref{calcdim} can thus be applied with $\mathcal G_N=\{Q_B:\, B\in\mathcal F^{D(N)}_N(g)\}$, $\delta_1=\delta_2=\underline{d}(\mu)$, $\Delta_1=\Delta_2=\overline{d}(\mu)$,  as well as $\epsilon_1$ and $\epsilon_2$ arbitrarily close to $0$. It follows that $\dim_H(\mu)= \underline d(\mu)$ and  $\dim_P(\mu)= \overline d(\mu)$.

\noindent 
(3) The fact that the law of $W^{(n)}$ is independent of $n$ implies that $\widehat{\boldsymbol{p}}_N$ is independent~$N$, so $g_r(N)/N$ converges to $\widetilde \chi^{-1}_r$ as $N$ tends to $+\infty$; moreover, the  $H(W^{(n)})$ are positive so  $\widetilde N=N$ for all $N\ge 1$. The previous properties combined with the definitions  of $\underline{d}(\mu)$ and $\overline{d}(\mu)$  and point (1) of the theorem yield the desired 
conclusion. 

\noindent 
(1) Similar arguments as in the proof of (2) yield $\dim_H(\mu)\le \underline d(\mu)$ and $\dim_P(\mu)\le \overline d(\mu)$. However, the assumption $\liminf_{N\to+\infty} N^{-1} \sum_{n=1}^H H(W^{(n)})=0$ and the definition of $d_N$  directly imply $\dim_H(\mu)=0$, since $g_1(N)/N$ is bounded. 

\noindent 
(4) It is easily seen from the definitions of $d_N$ and $\widetilde d_N$ that 
\begin{align*}
\underline{d}(\mu)=\min \Big (\liminf_{N\to+\infty} \widetilde d_N, \liminf_{N\to+\infty} \frac{1}{N}\sum_{n=1}^{ \widetilde{g}_s(N)}H(W^{(n)})\Big ).
\end{align*}
Denote  $\widetilde{g}_s(N)$ by $M$. By definition of the Lyapounov exponents, and since $\widetilde M=M\ge g_s(N)$, there exists a constant $C>0$ independent of $N$ such that if $N'$ is the largest integer  for which  $g_{s(N')}(N')\le M$, one has $N'\ge N-C$ and  $M-C\le g_{s(N')}(N')\le \widetilde{g}_{s(N')}(N')\le M=\widetilde M$.  Hence, since by the assumptions of Theorem~\ref{sufcond2}(1) one has $|H(W^{(n)})|=o(n)$, one obtains  that  for all $\epsilon>0$, for $N$ large enough,  $\frac{1}{N}\sum_{n=1}^{ \widetilde{g}_s(N)}H(W^{(n)})\ge \frac{1}{N'}\sum_{n=1}^{ {g_{s(N')}(N')}}H(W^{(n)})-\epsilon$. This is enough to conclude that $\underline{d}(\mu)\ge \liminf_{N\to+\infty} \widetilde d_{N}$, hence $\underline{d}(\mu)= \liminf_{N\to+\infty} \widetilde d_{N}$. 

Now we give an example for which one has 
$$
\overline d(\mu)<\min \Big (\limsup_{N\to+\infty} \widetilde d_N, \limsup_{N\to+\infty} \frac{1}{N}\sum_{n=1}^{\widetilde g_s(N)}H(W^{(n)})\Big ).
$$ 
We work on a Sierpi\'nski carpet, so that the Lyapunov exponents $\chi_1$ and $\chi_2$ are constant, and we assume that they are distinct. We fix a probability vector $p=(p_i)_{i\in\mathcal I}$ with positive components, as well as three non negative random vectors $\widetilde W_1=(\widetilde W_{1,i})_{i\in\mathcal I}$, $\widetilde W_2=(\widetilde W_{2,i})_{i\in\mathcal I}$ and $\widetilde W_3=(\widetilde W_{3,i})_{i\in\mathcal I}$ whose components are positive and bounded, with  expectation 1, such that  setting $W_j=(p_i\widetilde W_{j,i})_{i\in\mathcal I}$, and $H_j=H(W_j)$, one has  $H_2>H_1>h(\Pi_2(p))$ and $H_3<0$.  We define a sequence $(W^{(n)})_{n\ge 1}$ of random vectors as follows. 

Fix $M_0=1$ and an integer $N_1>1$. Then set $N_2=\lceil \frac{\chi_1}{\chi_2}N_1\rceil$, $M_1= g_2(N_1)$, $M_2=g_2(N_2)$, $M_3=M_2-H_3^{-1}\sum_{n=g_2(N_1)+1}^{M_2}H_2$. Then, define 
$$
W^{(n)}= \begin{cases}
W_1&\text{if }M_0\le n\le M_1,\\
W_2& \text{if }M_1+1\le n\le M_2,\\
W_3&\text{if }M_2+1\le n\le M_3.
\end{cases}
$$
By construction, $\sum_{n=M_1+1}^N H(W^{(n)})>0$ for all $N\in[M_1+1,M_3)$ (with a maximum at $M_2$) and $\sum_{n=M_1+1}^{M_3}H(W^{(n)})=0$, so that with the definition of $W^{(n)}$ for $n\ge M_3+1$ to follow, this implies that $\widetilde N=M_3$ for all $N\in [M_1+1,M_3]$. The construction continues recursively by updating the values $M_0$ and $N_1$  so that $M_0=M_3+1$ and $N_1$ satisfies $g_2(N_1)>M_3^2\ge g_2(N_1-1)$, and defining  $N_2$, $M_2$, $M_3$ and the $W^{(n)}$ as above, and so on... in particular, asymptotically $g_1(N_2)\sim g_2(N_1)$, $M_2\sim \frac{\chi_1}{\chi_2}M_1$  and $M_3\sim\big ( \frac{\chi_1}{\chi_2}+\frac{H_2}{|H_3|}(\frac{\chi_1}{\chi_2}-1)\big )M_1$. 

By construction $\liminf_{N\to+\infty}\sum_{n=1}^NH(W^{(n)}) \ge \frac{H_1}{ \frac{\chi_1}{\chi_2}+\frac{H_2}{|H_3|}(\frac{\chi_1}{\chi_2}-1)}>0$, so the sequence  $(W^{(n)})_{n\ge 1}$ yields a non degenerate IMM. Moreover, it is easily checked that $\overline d(\mu)\le \frac{H_1}{\chi_1}+(\frac{1}{\chi_2}-\frac{1}{\chi_1})h(\Pi_2(p))$ while $ \limsup_{N\to+\infty} \widetilde d_N\ge\frac{\chi_2}{\chi_1^2}(H_2-H_1)+ \frac{H_2}{\chi_1}+(\frac{1}{\chi_2}-\frac{1}{\chi_1})h(\Pi_2(p))$ and $ \limsup_{N\to+\infty} \frac{1}{N}\sum_{n=1}^{\widetilde g_2(N)}H(W^{(n)})\ge \frac{H_1}{\chi_2}$. 
\end{proof}

\begin{proof} [Proof of Proposition~\ref{mursnonchargés}] At first, note that since the $f_i((0,1)^d)$, $i\in\mathcal I$, are pairwise disjoint, for $w\in\mathcal I^*$, the inclusion $\pi([w])\subset f_w([0,1]^d)$ implies that for $\mu(\pi([w]))>\nu([w])$ to hold,  there must be some cylinder $w'$ such that $[w']\cap [w]=\emptyset $ but $\partial f_w([0,1]^d)\cap \partial f_{w'}([0,1]^d)\neq \emptyset$ and $\mu(\partial f_w([0,1]^d)\cap \partial f_{w'}([0,1]^d))>0$. Take such a cylinder.  Without loss of generality we can suppose that $|w|=|w'|$; then, for each point $z\in \partial f_w([0,1]^d)\cap \partial f_{w'}([0,1]^d)$, upon exchanging $w$ and $w'$, if necessary, their must exist $1\le k\le d$ such that for all $n>n_0=|w|$, there exists  $\widetilde w_n\in   \mathcal I(k,0)$ and  $\widetilde w'_n\in \mathcal I(k,1)$ such that for all $n>n_0$, $z\in f_{w\cdot \widetilde w_{n_0+1}\cdots  \widetilde w_n}([0,1]^d)\cap  f_{w'\cdot \widetilde w'_{n_0+1}\cdots  \widetilde w'_n}([0,1]^d)$. Also, if $z$ belongs to another parallepiped $f_{w''}([0,1]^d)$ whose interior is disjoint from that of $f_{w}([0,1]^d)$, the same property as with $w'$ must hold with $w''$. Thus, 
$$
\pi^{-1}\big ( \partial f_w([0,1]^d)\cap \partial f_{w'}([0,1]^d)\big )\subset \bigcup_{p\ge 1}\bigcup_{s\in\{0,1\},1\le k\le d}\bigcap_{n>p} \bigcup_{\substack{w''\in\mathcal{I}^n\\ w''_{p+1},\ldots,w''_n \in  \mathcal I(k,s)}}[w''].
$$ 
However, it is easily seen that by construction of $\nu$, for all $s\in\{0,1\}$  and $1\le k\le d$, and for all $m>n> p$, 
$$
\mathbb{E}\Big (\nu_m\Big (\bigcup_{\substack{w''\in\mathcal{I}^n\\ w''_{p+1},\ldots,w''_n \in \mathcal I(k,s)}}[w'']\Big)\Big)\le \prod_{j=p+1}^n\Big(\sum_{i\in\mathcal I(k,s)}\mathbb{E}(W^{(j)}_{i})\Big).
$$
By using the Fatou lemma one deduces that 
$$
\mathbb{E}\Big (\nu\Big( \bigcup_{\substack{w''\in\mathcal{I}^n\\ w''_{p+1},\ldots,w''_n \in \mathcal I(k,s)}}[w']\Big)\Big)\le \prod_{j=p+1}^n\Big (\sum_{i\in\mathcal I(k,s)}\mathbb{E}(W^{(j)}_{i})\Big).
$$
Consequently, 
$$
\mathbb{E}\Big (\nu\Big(\bigcap_{n>p} \bigcup_{\substack{w''\in\mathcal{I}^n\\ w''_{p+1},\ldots,w''_n \in \mathcal I(k,s)}}[w'']\Big)\Big)=0.
$$
Finally, it is almost sure that for all $w,w'$ in $\Sigma^*$ such that $[w']\cap [w]=\emptyset $, one has $\nu\big (\pi^{-1}\big ( \partial f_w([0,1]^d)\cap \partial f_{w'}([0,1]^d)\big )\big )=0$, hence $\mu(\pi([w]))=\nu([w])$. 

It follows from what precedes that the $\mu$-mass of $\partial f_w([0,1]^d)$ is only due to that the $\nu$-mass of subcylinders of $[w]$. Reasoning  as above  we see that for each $n\in\mathbb{N}^+$, the subcylinders of $w$ of generation $|w|+n$ which do contribute to this mass must be of the form $ww'$, with $w'\in  \bigcup_{k=1}^d\bigcup_{s\in\{0,1\}} \mathcal I(k,s)^n$. Then, a calculation similar to the previous one  shows that $\nu (\pi^{-1}(\partial f_w([0,1]^d)))=0$.  
\end{proof}

\subsection{Sketch of proof of Theorem~\ref{dimep}} We already know by Theorem~\ref{thm-2.4}(2) that $\dim_H(\mu)$ and $\dim_P(\mu)$ do exist.  For each rational number $T\in[1,\lambda)$, consider the sequence of scales $e^{-T_N}$ with $T_N=T\lambda^N$, $N\ge 0$. 

These scales define a sequence $(D(T_N), s(T_N), g_1(T_N),\ldots, g_{s(T_N)}(T_N))_{N\ge 1}$, a sequence $(d_{T_N})_{N\ge 1}$ by replacing $N$ by $T_N$ in Definition~\ref{HN}, and associated partitions as in Propositions~\ref{tau} and~\ref{SN'}. Applying these propositions to this new sequence of partitions, approximating the sums involved in $d_{T_N}$ by integrals thanks to the exponential continuity and periodicity property of $(W^{(t)})_{t>0}$, and noting that the assumption of Proposition~\ref{mursnonchargés} holds since the mapping $t\mapsto p^{(t)}$ is positive and continuous as well as exponentially periodic, yields that with probability~1, conditional on $\mu\neq 0$, 
\begin{align*}
\liminf_{N\to +\infty} \frac{\log(\mu(Q_{T_N}(z)))}{-T_N}&=\liminf_{N\to+\infty} \min (\delta_1(T_N),\delta_2(T_N))\\
 \limsup_{N\to +\infty} \frac{\log(\mu(Q_{T_N}(z)))}{-T_N}&=\limsup_{N\to+\infty} \min (\delta_1(T_N),\delta_2(T_N))
 \end{align*}
for $\mu$-almost every $z$, where  $Q_{T_N}(z)$ is a parallelepided whose sides lengths have logarithms equivalent to $-T_N$. Since the set of rational numbers of $[1,\lambda)$ is countable, the previous equality holds simultaneously for all  $T\in \mathbb Q\cap[1,\lambda)$. However, by $\lambda$-exponential periodicity, one has $\min (\delta_1(T_N),\delta_2(T_N))=\min (\delta_1(T),\delta_2(T))=:\delta(T)$. It follows that given $\epsilon>0$, for each integer $q\in\mathbb{N}^+$, there exists $N_q\in \mathbb{N}^+$ such that for $\mu$-almost every $z$, for all $T\in \mathcal D_{q,\lambda}=(q^{-1} \mathbb{N})\cap [1,\lambda)$ and $N\ge N_q$, one has $ \frac{\log(\mu(Q_{T_N}(z)))}{-T_{N}}\in [\delta(T)-\epsilon, \delta(T)+\epsilon]$. Moreover, for all $j\in \mathbb{N}^+$ there exists $T^{(j)}\in \mathcal D_{q,\lambda}$ and $N\in\mathbb{N}$ such that $T^{(j)}\lambda^N\le j<(T^{(j)}+q^{-1})\lambda^N$. This makes it possible to construct a parallelepiped $\widetilde Q_j(z)$ containing $z$ as interior point, and whose  sides lengths have logarithms equivalent to $-j$ as $j\to+\infty$, and such that for $j$ large enough $\frac{\log(\mu(\widetilde Q_j(z)))}{-j}\in \bigcup_{T\in  \mathcal D_{q,\lambda}}   [\delta(T)-2\epsilon, \delta(T)+2\epsilon]$. Pick $T_{\min}(q)$ and $T_{\max}(q)$ at which $\delta_{| \mathcal D_{q,\lambda}}$ takes its minimum and its maximum respectively. The previous lines together with Lemma~\ref{calcdim} imply that $\dim_H(\mu)\in  [\delta(T_{\min}(q))-2\epsilon, \delta(T_{\min}(q))+2\epsilon]$ and $\dim_P(\mu)\in [\delta(T_{\max}(q))-2\epsilon, \delta(T_{\max}(q))+2\epsilon]$. Letting $q$ tend to $+\infty$ and then  $\epsilon$ to 0 yields \eqref{dimmuep} and  \eqref{Dimmuep} when $T$ is restricted to rational numbers of $[1,\lambda]$. In particular, $\dim_H(\mu)\ge \inf_{T\in[1,\lambda]}\min(\delta_1(T),\delta_2(T))$ and $\dim_P(\mu)\le \sup_{T\in[1,\lambda]}\min(\delta_1(T),\delta_2(T))$. On the other hand, for any $\epsilon$, we can take $T_{\epsilon}$ and $T_\epsilon'$ in $[1,\lambda]$ such that   $\inf_{T\in[1,\lambda]}\min(\delta_1(T),\delta_2(T))\ge \min(\delta_1(T_\epsilon),\delta_2(T_\epsilon))-\epsilon$ and $\sup_{T\in[1,\lambda]}\min(\delta_1(T),\delta_2(T))\le \min(\delta_1(T_\epsilon'),\delta_2(T_\epsilon'))-\epsilon$. Considering $\widetilde  {\mathcal D}_{q,\lambda}=\mathcal D_{q,\lambda}\cup\{T_\epsilon, T_\epsilon'\}$ instead of $\mathcal D_{q,\lambda}$ and letting $\epsilon\to 0$  yields $ \dim_H(\mu)\le \inf_{T\in[1,\lambda]}\min(\delta_1(T),\delta_2(T))$ and $\dim_P(\mu)\ge \sup_{T\in[1,\lambda]}\min(\delta_1(T),\delta_2(T))$.

\section{Proof of Theorem~\ref{dimKd}}\label{UB}

We establish that the value provided by Theorem~\ref{dimKd} for $\dim_H K_\omega$ is sharp. To do so, we use suitable coverings.  In the spirit of what Bedford did for Sierpi\'nski carpets~\cite{ref3}, Gatzouras and Lalley for statistically self-affine Sierpi\'nski carpets~\cite{ref9}, and recently the first author and Feng in \cite{ref6} for statistically self-affine Sierpi\'nski sponges, our argument appeals to digit frequencies.  However, as a counterpart of the fact that considering Mandelbrot measures is in general too limited to get a variational principle for $\dim_H K$, we associate to each  element of $\mathscr L$ (recall the Definition~\ref{L} of $\mathscr L$) a sequence of \textit{localized} frequencies for the digits of any point $\boldsymbol{i}\in\mathcal{I}^{\mathbb{N}^+}$.

Fix $\ell= (\ell_m)_{m\in\mathbb{N}^+}\in \mathscr L$, and denote $L_m=\ell_1+\cdots+\ell_m$ for $m\in\mathbb{N}^+$, and $L_0=0$. 

For $\boldsymbol{i}\in \mathcal{I}^{\mathbb{N}^+}$, $i\in\mathcal I$, $m\in\mathbb{N}^+$,  and $L_{m-1}+1\le n\le L_m$, set 
\begin{equation}\label{pni}
p^{(n)}_i(\boldsymbol{i})=\frac{1}{\ell_m}\sum_{n=L_{m-1}+1}^{L_m} \mathbf{1}_{\{i\}}(\boldsymbol{i}_n).
\end{equation}
If $L_{m-1}+1\le n\le L_m$, then  the probability vector $p^{(n)}(\boldsymbol{i})$  provides the frequency of the digits $i\in\mathcal I$ in the subword $ \boldsymbol{i}_{L_{m-1}+1}\cdots \boldsymbol{i}_{L_m}$. 

Before dealing with $\dim_H K_\omega$, we provide another approach to get the upper bound $\dim_H(\mu)\le \underline{d}(\nu)$ in Theorem~\ref{thm-2.4}(1) using coverings, and under suitable assumptions. It will exhibit estimates for the expectation of some covering numbers, which turn out to be crucial to get the sharp upper bound for $\dim_H K_\omega$.  

\subsection{Alternative proof of the upper bound $\dim_H(\mu)\le \underline{d}(\mu)$ in Theorem~\ref{thm-2.4}(1) when $\mu$ is of type $\ell$}\label{UPdimH} Recall \eqref{tildeW}. We assume that for all integers $m\ge 1$, $p^{(n)}$ is independent of $n$ for $n\in [L_{m-1}+1,L_m]$, and that there exists $\eta\in (0,1)$ such that $\inf_{n\ge 1, i\in\mathcal I}p^{(n)}_i\ge \eta$. We will use the following lemma, whose proof we  postpone to the end of this subsection.

\begin{lemma}\label{local control} Under the assumptions of Theorem~\ref{thm-2.4}(1), if $\log(m)=o(\ell_m)$ and with probability~1, conditional on $\nu\neq 0$, for $\nu$-almost every $\boldsymbol{i}$, one has 
\begin{equation}\label{prophp}
\lim_{m\to +\infty} \ell_m^{-1}\left |\sum_{n=L_{m-1}+1}^{L_m} \log (W_{i_n}(\boldsymbol{i}_{|n-1}))+H(W^{(L_m)})\right|=\lim_{n\to +\infty} \left \|p^{(n)}(\boldsymbol{i})-p^{(n)}\right \|_\infty=0.
\end{equation}
\end{lemma}

Recall the definition of the sets $\Sigma_{\omega,n}$ in \eqref{Sigman}. Define the set 
$$
E=\bigcap_{0<\delta< 1}\bigcup_{M\ge 1} \Big (E(M,\delta):=\bigcap_{m\ge M} E(M,m,\delta)\Big ),
$$
where
\begin{equation}\label{epMmd}
E(M, m,\delta)=\bigcap_{m'= M}^m\left \{\boldsymbol{i}\in\Sigma:\,  
\begin{cases}
\boldsymbol{i}\in \Sigma_{\omega,L_{m'}},\\
\ell_{m'}^{-1}\left |\sum_{n=L_{m'-1}+1}^{L_{m'}} \log (W_{i_n}(\boldsymbol{i}_{|n-1}))+H(W^{(L_m)})\right|\le \delta\\
\left \|p^{(n)}(\boldsymbol{i})-p^{(n)}\right \|_\infty\le \delta, \forall \, L_{m'-1}+1\le n\le L_{m'}
\end{cases}
\right\}.
\end{equation}
Since $\ell_m=o(L_m)$ as $m\to\infty$, for each $\delta\in (0,1)$,  we can fix an integer $M_\delta $ such that  $\ell_{m}\le \delta L_{m-1}$ for all $m\ge M_\delta$.  

Recall \eqref{chip} and let us establish that for all $k\in\{1,\cdots,d\}$,  for $N\ge L_M/\delta$ and $\boldsymbol{i}\in E(M,\delta)$, one has 
\begin{equation}\label{lyap}
\Big |\chi_k(\widehat{\boldsymbol{p}}_N)+\frac{1}{N}\sum_{n=1}^N\log(|a_{i_n,k}|)\Big |\le   \big ((2+\#\mathcal I) \max_{i,k} |\log(|a_{i,k}|)|\big )\delta:=\lambda_a\delta.
\end{equation}

For $N\in\mathbb{N}^+$, denote by $m(N)$ the greatest integer such that $L_{m(N)}\le N-1$. Then, for $\delta\in (0,1)$, $M\ge M_\delta$ and $N>L_M$, so that $m(N)\ge M$, write $N\widehat {\boldsymbol{p}}_N=\sum_{n=1}^{L_M}p^{(n)}+\sum_{m=M+1}^{m(N)}\ell_m p^{(L_m)}+\sum_{n=L_{m(N)}+1}^Np^{(n)}$.  Also, note that for all $k\in\{1,\ldots,d\}$ and $n\in\mathbb{N}^+$, one has $0<\chi_k(p^{(n)})\le \max_{i} |\log(|a_{i,k}|)|$. Thus, setting  $\mathcal E_N=[1,L_{M}]\cup [L_{m(N)}+1,N]$, one has 
$$
\left|\sum_{n\in \mathcal E_N}\chi_k(p^{(n)})+\log(|a_{i_n,k}|)\right |\le (L_{M}+N-L_{m(N)})\max_{i} |\log(|a_{i,k}|)|.
$$
Moreover, using the definition of $p^{(n)}(\boldsymbol{i})$, which is independent of $n$ over each interval $[L_{m-1}+1,L_m]$, the third condition in the definition of $E(M,m,\delta)$  implies that for  $M\ge M_{\delta}$, $N> L_M$ and $\boldsymbol{i}\in E(M,\delta)$,  one has
\begin{align*}
\left|\sum_{m=M+1}^{m(N)}\ell_m \chi_k(p^{(L_m)})+\sum_{n=L_M+1}^{L_{m(N)}}\log(|a_{i,k}|)\right |&= \left|\sum_{m=M+1}^{m(N)}\ell_m \sum_{i\in\mathcal I} \big (p_i^{(L_m)}-p_i^{(L_m)}(\boldsymbol{i})\big )\log(|a_{i,k}|)\right|\\
&\le \sum_{m=M+1}^{m(N)}\ell_m \sum_{i\in\mathcal I}\delta \max_{i} |\log(|a_{i,k}|)|\\
&=(L_{m(N)}-L_M)(\#\mathcal I)\delta \max_{i} |\log(|a_{i,k}|)|.
\end{align*}
Thus 
$$
\Big |\chi_k(\widehat{\boldsymbol{p}}_N)+\frac{1}{N}\sum_{n=1}^N\log(|a_{i_n,k}|)\Big |\le \Big (\frac{L_M}{N}+\frac{N-L_{m(N)}}{N}+(\#\mathcal I)\delta\frac{L_{m(N)}-L_M}{N}\Big )\max_{i} |\log(|a_{i,k}|)|.
$$
Moreover, $\frac{L_{m(N)}-L_M}{N}\le 1$ and $\frac{N-L_{m(N)}}{N}\le \frac{\ell_{m(N)+1}}{L_{m(N)}}\le  \delta$. This is enough to get \eqref{lyap}. 

\begin{remark}\label{remindepassump}
Lemma~\ref{local control}  implies that there is a Borel set $F$ of full $\mu$-measure such that $F\subset \pi (E)$, but \textit{it is $\dim_H \pi(E)$ that we are going to estimate, independently of the assumption $\log(m)=o(\ell_m)$}. 
\end{remark}

Fix $\delta\in (0,1)$. By definition of the sets $A_r(N)$ and the integers $g_r(N)$, $1\le r\le s=s(N)$ (see section~\ref{thmdim}), \eqref{lyap} implies that  for  $M\ge M_{\delta}$, $N$ such that $g_1(N)\ge L_M/\delta$ and $\boldsymbol{i}\in E(M,\delta)$, for all $1\le r\le s$, one has
$$
\sup_{k\in A_r(N)}\prod_{n=1}^{g_r(N)} |a_{i_n,k}|\le e^{\lambda_a \delta  g_r(N)} e^{-g_r(N)\chi_k(\widehat{\boldsymbol{p}}_{g_r(N)})}\le e^{\lambda_a \delta  g_r(N)} e^{-N}\le  e^{\lambda_a \delta  \Lambda_aN } e^{-N},
$$
where $\Lambda_a$ is defined as in \eqref{Lambdaa}. In particular, if we use the notation introduced in the proof of Theorem~\ref{thm-2.4}(2),  all the sides of the parallelepiped $Q_{B_N(\boldsymbol{i})}$ are smaller than or equal to $e^{\lambda_a \Lambda_a \delta N } e^{-N}$. 

Below we find, for $N$ large enough,  an upper bound for the expectation of the number $\mathcal N_N$ (which depends on $M$) of sets $B_{N}(\boldsymbol{i})$ of $\mathcal F^D_N(g)$ such that $\boldsymbol{i}\in E(M,m( \widetilde{g}_s(N)),\delta)$. This will provide an asymptotic almost sure upper bound for this number and suitable coverings of $\pi (E(M,\delta))$. 

We will use the following observation.
\begin{remark}\label{Hinfty}
It follows from~\eqref{HW} that for all $n\ge 1$, 
$$
\max \big(|H(W^{(n)})|, h(p^{(n)}\big )\le H_\infty:= \log(\#\mathcal I) + \sup_{n\in\mathbb{N}^+}\sup_{i\in\mathcal I} \mathbb{E}(\widetilde W^{(n)}_i\log(\widetilde W^{(n)}_i))<+\infty.
$$ 
\end{remark}

Let $\epsilon\in(0,\log (\#\mathcal I))$ such that $\sum_{n=1}^NH(W^{(n)})\ge N\epsilon$ for $N$ large enough and let  $N_\epsilon=\min\big \{N\ge 1:\, \forall\, N'\ge N, \sum_{n=1}^{N'}H(W^{(n)})\ge N'\epsilon\big\}$. 

\noindent
\textbf{Claim:} Recall~\eqref{LambdaLambda'} where $\Lambda'_a$ is defined. The exists a constant $C=C(\#\mathcal I,\epsilon, \delta, H_\infty)$ such that for all $M\ge M_\delta$ and $N\ge \max (L_M/(\delta\Lambda'_a), N_\epsilon/\Lambda'_a)$, one has 
\begin{equation}\label{claim}
\mathbb{E}(\mathcal N_N)\le e^{(C\delta +d_{N}) N}. 
\end{equation}
Let us assume the claim and prove that $\dim_H(\mu)\le \underline{d}(\nu)$. Fix $M\ge M_\delta$. Let $(N_j)_{j\in\mathbb{N}^+}$ be a strictly  increasing sequence of integers such that $\underline{d}(\nu)=\lim_{j\to +\infty}d_{N_j}$. By the Borel-Cantelli Lemma, the claim  implies that with probability~1, for $j$ large enough, one has $\mathcal N_{N_j}\le e^{((C+1)\delta +d_{N_j}) N_j}$ (indeed, $\sum_{j\ge 1}e^{-((C+1)\delta +d_{N_j}) N_j}\mathbb{E}(\mathcal N_{N_j})<+\infty$). Consider the associated $\mathcal N_{N_j}$ sets of the form $B_{N_j}(\boldsymbol{i})$. To each such set is associated the  parallelepiped $Q_{B_{N_j}(\boldsymbol{i})}$, and by definition, the union of these parallelepipeds covers $\pi (E(M,m(  \widetilde{g}_s(N)),\delta))$, hence $\pi (E(M,\delta))$; moreover, by the discussion following Remark~\ref{remindepassump}, all these parallelepipeds have a diameter smaller than or equal to $\sqrt{d} e^{\lambda_a\Lambda_a\delta N_j } e^{-N_j}$. By definition of the $t$-dimensional Hausdorff measure $\mathcal H^t$, this implies that if $t> \underline{d}(\nu)+ (\lambda_a\Lambda_a+(C+1))\delta$, then $\mathcal H^{t} (\pi (E(M,\delta)))=0$, hence $\dim_H (\pi (E(M,\delta)))\le t$, so $\dim_H (\pi (E(M,\delta)))\le \underline{d}(\nu)+ (\lambda_a\Lambda_a+(C+1))\delta$. This is independent of $M$, hence $\dim_H F \le \sup_{M\ge 1} \dim_H (\pi (E(M,\delta)))\le \underline{d}(\nu)+ (\lambda_a\Lambda_a+(C+1))\delta$. Taking the infimum over $\delta$ yields $\dim_H F\le \underline{d}(\nu)$, so that $\dim_H(\mu)\le \underline{d}(\nu)$.

\medskip

Now we prove the claim. This is equivalent to proving that there exists a constant $C=C(\#\mathcal I,\epsilon, \delta, H_\infty)$ such that for all $M\ge M_\delta$ and $N\ge \max (L_M/(\delta\Lambda'_a), N_\epsilon/\Lambda'_a)$, one has 
\begin{equation}\label{espN}
\mathbb{E}(\mathcal N_N)\le \begin{cases}e^{C\delta N+\sum_{n=1}^{  \widetilde{g}_s(N)} H(W^{(n)})}\\
  e^{C\delta N+ H_{N,k}},\ g_1(s)\le k\le g_s(N)-1.
  \end{cases}
  \end{equation}  

We start with $\mathbb{E}(\mathcal N_N)\le e^{C\delta N+\sum_{n=1}^{  \widetilde{g}_s(N)} H(W^{(n)})}$. 

Let $M\ge M_\delta$. Fix $N\in\mathbb{N}^+$ such that $N\ge \max (L_M/(\delta\Lambda'_a), N_\epsilon/\Lambda'_a)$. In particular $g_1(N)\ge \Lambda'_a N\ge L_M/\delta$ (since $\delta\in(0,1)$) and $g_1(N)\ge N_\epsilon$.

Recall that each set $B_{N}(\boldsymbol{i})$ takes the form $B(U_1,\ldots,U_s)$, with $U_r\in (\mathcal{I}^{D(N)}_r)^{g_r(N)-g_{r-1}(N)}$ for $1\le r\le s=s(N)$. Observe that for $B(U_1,\ldots,U_{s})\cap E(M,m(  \widetilde{g}_s(N)),\delta)\neq\emptyset $ to hold,  $B(U_1,\ldots,U_{s})$ must contain a cylinder  $[U]$ of generation $  \widetilde{g}_s(N)$ which intersects $E(M, m(  \widetilde{g}_s(N)),\delta)$. So $\mathcal N_N$ is smaller than or equal to the cardinality of the set of those cylinders. Set $L=L_{M-1}$, $L'=L_{m( \widetilde{g}_s(N))}$ and $L''= \widetilde{g}_s(N)-L_{m( \widetilde{g}_s(N))}$. Such a $U$ writes $uvw$ with $(u,v,w)\in {\mathcal I}^{L}\times  {\mathcal I}^{L'-L}\times  {\mathcal I}^{L''}$, and by definition of $E(M, m(  \widetilde{g}_s(N)),\delta)$, $[uvw]\cap E(M, m(  \widetilde{g}_s(N)),\delta)\neq \emptyset$ implies that 
$$
\sum_{n=1}^{L'-L} \log (W_{v_n}((uv)_{|L+n-1}))+H(W^{(L+n)})\ge - (L'-L)\delta, 
$$
i.e. 
\begin{equation}\label{ineqforMarkov}
e^{(L'-L)\delta+\sum_{n=1}^{L'-L}H(W^{(L+n)})}\prod_{n=1}^{L'-L}W_{v_n}((uv)_{|L+n-1})\ge 1.
\end{equation}
Applying the Markov inequality with respect to the counting measure over $\mathcal{I}^{L}\times  \mathcal{I}^{L'-L}\times   \mathcal{I}^{L''}$ to  the function of the left hand side of \eqref{ineqforMarkov} viewed as a function of $(u,v,w)$, and taking the expectation, we can get ($u$ being any element of  $\mathcal{I}^{L}$),
\begin{align*}
\mathbb{E}(\mathcal N_N)&\le (\#\mathcal{I})^{L+L''}  e^{(L'-L)\delta+\sum_{n=1}^{L'-L}H(W^{(L+n)})} \mathbb{E}\Big (\sum_{v \in \mathcal I^{L'-L}}\prod_{n=1}^{L'-L}W_{v_n}((uv)_{|L+n-1})\Big )\\
&= (\#\mathcal{I})^{L+L''}  e^{(L'-L)\delta+\sum_{n=1}^{L'-L}H(W^{(L+n)})} \mathbb{E}(Y_{L'-L}(u))
\\
&=(\#\mathcal{I})^{L+L''}  e^{(L'-L)\delta+\sum_{n=L+1}^{L'}H(W^{(n)})}  .
\end{align*}

We have 
\begin{align*}
L+L''&\le L_{M-1}+\ell_{m(  \widetilde{g}_s(N))+1}\\
&\le \delta g_s(N)+ \delta L_{m(  \widetilde{g}_s(N))}\le 2 \delta   \widetilde{g}_s(N)\le 2\delta \widetilde C g_s(N)\le 2\delta  \widetilde C \Lambda_a N,
\end{align*}
where $\widetilde C = \frac{\log(\#\mathcal{I})}{\epsilon}$, by using Proposition~\ref{Moments1}(2) to get the fourth inequality (note that $g_s(N)\ge g_1(N)\ge N_\epsilon$). Also,   $L'-L\le  \widetilde{g}_s(N)\le  \widetilde C \Lambda_a N$.  Since the  $|H(W^{(n)})|$ are uniformly bounded by $H_\infty$ (Remark~\ref{Hinfty}), we conclude that  
$$
(\#\mathcal{I})^{L+L''}  e^{(L'-L)\delta+\sum_{n=L+1}^{L'}H(W^{(n)})}\le e^{C_1\delta N} e^{\sum_{n=1}^{ \widetilde{g}_s(N)}H(W^{(n)})},
$$
where $C_1=(1+2\log(\#\mathcal I)+2H_\infty)  \widetilde C \Lambda_a$. This is the first part of~\eqref{espN}. %

\smallskip

Before proving the second part of \eqref{espN}, we need additional notation and an observation: recall that for all $N\ge 1$,  $m(N)$ is the greatest integer such that $L_{m(N)}\le N-1$. 
For all $1\le r\le s$ denote $m(g_r(N))$ by $m_r$.  Also, simply denote $D(N)$ by $D$. For each $2\le r\le s$, the word~$U_r$,  if written $U_r=u_{g_{r-1}(N)+1}\cdots u_{g_r(N)}$, has the following decomposition into words whose indexes belong to intervals of $\mathbb{N}^+$ over which $p^{(n)}$ is independent of $n$:
$$
U_r=U_{r,m_{r-1}+1}\cdot U_{r,m_{r-1}+2}\cdots U_{r,m_r}\cdot U_{r,m_r+1},
$$
where
$$
\begin{cases}
U_{r,m_{r-1}+1}=u_{g_{r-1}(N)+1}\cdots u_{L_{m_{r-1}+1}}\\
U_{r,m}= u_{L_{m-1}+1}\cdots u_{L_m}&\text{for } m_{r-1}+2\le m\le m_r\\
U_{r,m_r+1}=u_{L(m_r)+1}\cdots u_{g_r(N)}.
\end{cases}
$$
For $m_{r-1}+1\le m\le m_r$ and $(U_1,\ldots,U_{r-1}, \odot_{m'=m_{r-1}+1}^m U_{r,m'})\in \Big (\prod_{r'=1}^{r-1} (\mathcal{I}^D_{r'})^{g_{r'}(N)-g_{r'-1}(N)}\Big )\times (\mathcal{I}^D_{r})^{L_m-g_{r-1}(N)}$, set 
\begin{multline*}
B(U_1,\ldots,U_{r-1}, \odot_{m'=m_{r-1+1}}^m U_{r,m'})\\
=\Big(\bigcap_{r'=1}^{r-1}( \Pi^D_{r'} \circ T^{g_{r'-1}(N)})^{-1} ([U_{r'}]) \Big )\cap \Big (\bigcap_{m'=m_{r-1}+1}^m( \Pi^D_{r} \circ T^{L(m')})^{-1}\Big ) ([U_{r,m'}])\Big).
\end{multline*}

We observe that if $B(U_1,\ldots, U_s)\cap E(M,m(g_s(N)),\delta)\neq\emptyset$ then, for all $2\le r\le s$ and $m_{r-1}+1\le m\le m_r+1$, one has 
$$
U_{r,m}\in \mathcal U_{r,m}=\left \{\Pi^D_r(U): \, U\in \mathcal{I}^{\ell_m},\, \sup_{i\in\mathcal I}\left | p_i^{(L_{m})}-\ell_{m}^{-1}\sum_{n=1}^{\ell_{m}}\mathbf{1}_{\{ i\}}(U_n)\right |\le \delta\right \};
$$ 
also, 
$$
\mathcal U_{r,m}\subset \mathcal U'_{r,m}= \left \{ U'\in (\mathcal{I}^D_r)^{\ell_{m}}:\, \sup_{i\in\mathcal I^D_r}\left | (\Pi^D_rp^{(L_{m})})_i-\ell_{m}^{-1}\sum_{n=1}^{\ell_{m}}\mathbf{1}_{\{ i\}}(U'_n)\right |\le (\#\mathcal I)\delta\right \},
$$
and it is standard that 
$$
\#  \mathcal U'_{r,m}\le \exp\big ((\# I)\delta  \sup_{i\in\mathcal I^D_r} |\log (\Pi^D_rp^{(L_{m})})_i |\ell_{m}\big)\cdot \exp(h( \Pi^D_rp^{(L_{m})})\ell_{m}), 
$$
so 
\begin{equation}\label{cardU}
\# \mathcal U_{r,m}\le \exp\big ((\# I)\delta  |\log(\eta)|\ell_m\big ) \exp (h( \Pi^D_rp^{(L_{m})})\ell_{m}),
\end{equation}
since we assumed that  $|\log (\Pi^D_rp^{(L_{m})})_i|\le |\log(\eta)|$.

\smallskip Now fix $M\ge M_\delta$, $N\ge  \max(L_M/(\delta\Lambda'_a), N_\epsilon/\Lambda'_a)$, and $g_1(N)\le k\le g_s(N)-1$. Recall that we want to prove~\eqref{espN}. Denote by $r_k$ the unique $2\le r\le s$ such that $g_{r-1}(N)+1\le k\le g_r(N)$.  

By the definition of $m(k)$ one has $L_{m(k)}+1\le k\le   L_{m(k)+1}$. Define  
$$
\mathcal M_{k}=\{m_{r_k-1}+2\le m\le  m_{r_k}: \, m> m(k)\}\cup \bigcup_{r>r_k}\{m_{r-1}+2\le m\le  m_r\}.
$$ 

The previous observation shows that setting $r(m)=r$ if $g_{r-1}(N)+1\le L_m\le g_r(N)$, one has 
\begin{align}
\nonumber&\{(U_1,\ldots, U_s): \, B(U_1,\ldots, U_s)\cap E(M,m(g_s(N),\delta)\neq\emptyset\}\\
\label{biginclusion}&\subset \left \{(U_r)_{r=1}^s: \, 
\begin{cases} B(U_1,\ldots,U_{r_k-1}, \odot_{m'=m_{r_k-1+1}}^{m(k)} U_{r_k,m'})\cap E(M,m(k),\delta)\neq\emptyset\}\\
\forall\, m\in\mathcal M_k,\, U_{r(m),m}\in \mathcal U_{r(m),m}\\
\forall \, r_k\le r\le s: \, U_{r,{m_{r}+1}}\in (\mathcal I^D_{r})^{g_{r}(N)-L_{m_{r}}},\\
\forall \, r_k\le r\le s: \, U_{r,m_{r-1}+1}\in (\mathcal I^D_{r})^{L_{m_{r-1}+1}-g_{r-1}(N)}\\
\text{if } m(k)\not\in\{m_{r_{k-1}}+1, m_{r_{k}}+1\},\, \text{then }U_{r_k,m(k)}\in  (\mathcal I^D_{r_k})^{\ell_{m(k)}}
\end{cases}\right\}.
\end{align}

Note that the proof of the first part of \eqref{espN} also yields, with 
\begin{equation}\label{C'1}
C'_1=(1+2\log(\#\mathcal I)+2H_\infty) \Lambda_a,
\end{equation}
that 
\begin{align*}
&\mathbb{E}(\#\{U=(U_1,\ldots,U_{r_k-1}, \odot_{m'=m_{r_k-1+1}}^{m(k)} U_{r,m'}): B(U)\cap E(M,m(k),\delta)\neq\emptyset\})\\
&\quad \le e^{C'_1\delta N} e^{\sum_{n=1}^{L_{m(k)}}H(W^{(n)})}.
\end{align*}
Using this fact,  as well as  \eqref{cardU} and the trivial inequality $\# \mathcal I^D_{r}\le \#\mathcal I$, we get from \eqref{biginclusion} that 
$$
\mathbb{E}(\mathcal N_N)\le e^{C'_1\delta N} e^{\sum_{n=1}^{L_{m(k)}}H(W^{(n)})} \Big (\prod_{m\in \mathcal M_k} e^{(\# \mathcal I)\delta  |\log(\eta)|\ell_m} e^{h( \Pi^D_{r(m)}p^{(L_{m})})\ell_{m}}\Big )(\#\mathcal I)^{\ell_{m(k)}+\sum_{r=1}^s\ell_{m_{r}+1}}.
$$
By the definition of $\mathcal M_k$ and $H_{N,k}$, one has 
\begin{align*}
&\Big |H_{N,k}-\sum_{n=1}^{L_{m(k)}}H(W^{(n)})-\sum_{m\in\mathcal M_k}h( \Pi^D_{r(m)}p^{(L_{m})})\ell_{m}\Big |\\
&\le \ell_{m(k)} \left (|H(W^{(L_{m(k)})})|+h( \Pi^D_{r_k}p^{(L_{m(k)})})\right )\\&\quad + \sum_{r=r_k}^s (g_{r}(N)-L_{m_{r}})h(\Pi^D_{r}p^{(L_{m_{r}+1})})+(L_{m_{r-1}+1}-g_{r-1}(N)) h(\Pi^D_{r}p^{(L_{m_{r-1}+1})})\\
&\le \Big (\ell_{m(k)}+\sum_{r=1}^s\ell_{m_{r}+1}\Big )(H_\infty+\log(\#\mathcal I)).
\end{align*}
Moreover,  
\begin{align*}
\ell_{m(k)}+\sum_{r=1}^s\ell_{m_{r}+1}&\le (s+1)\delta \Lambda_a N\\
\text{and }\prod_{m\in \mathcal M_k} e^{(\# \mathcal I)\delta  |\log(\eta)|\ell_m}&\le e^{(\# \mathcal I)|\log(\eta)|\delta g_s(N)}\le e^{(\# \mathcal I)|\log(\eta)|\delta\Lambda_a N}.
\end{align*} 
Thus, taking into account that  $C'_1=(1+2\log(\#\mathcal I)+2H_\infty) \Lambda_a$ and $s+1=s(N)+1\le d+1$, we obtain
$$
\mathbb{E}(\mathcal N_N)\le e^{C''_1\delta N+H_{N,k}}
$$
with 
\begin{equation}\label{C''1}
C''_1=[(2(d+1)+3)(\log(\#\mathcal I)+H_\infty)+ (\#\mathcal I)|\log(\eta)| ]\Lambda_a. 
\end{equation}
Taking the infimum of the various upper bounds we found for $\mathbb{E}(\mathcal N_N)$, we conclude that 
$$
\mathbb{E}(\mathcal N_N)\le e^{(C\delta+d_N)N},
$$
with 
\begin{equation}\label{constant C}
C= [(2(d+1)+3)(\log(\#\mathcal I)+H_\infty)  \widetilde C + (\#\mathcal I)|\log(\eta)| \big )]\Lambda_a,
\end{equation}
where 
\begin{equation}\label{constant tildeC}
\widetilde C=\frac{\log (\#\mathcal{I})}{\epsilon}.
\end{equation}

\begin{remark}\label{alternativeupperbound}
Among the upper bounds we obtained, only the first one depends on the property that $\liminf_{N\to\infty} \sum_{n=1}^NH(W^{(n)})>0$, since it requires the consideration of $ \widetilde{g}_s(N)$. 

In fact, we have also obtained that even if this property does not hold, for $M\ge M_\delta$ and  $N\ge  L_M/\delta\Lambda'_a$, if $\widetilde{\mathcal N}_N$ stands for the number of those sets $B_{N}(\boldsymbol{i})$, with $\boldsymbol{i}\in E(M,m({g_s(N)}),\delta)$, then 
\begin{equation}\label{constant C'}
\mathbb{E}(\widetilde{\mathcal N}_N)\le e^{C''_1\delta N +\min_{g_1(N)\le k\le g_s(N)}H_{N,k}}=e^{(C''_1\delta +\widetilde d_N) N} . 
\end{equation}
Also, independently of the above property, our estimates show that if $M\ge M_\delta$, $N\ge L_{M}/\delta$, and $\widehat{\mathcal N}_N$ stands for the cardinality of those $U_1\in\mathcal{I}^{N}$ such that $[U_1]\cap E(M, m(N),\delta)\neq\emptyset$, then 
\begin{equation}\label{espN'}
\mathbb{E}(\widehat{\mathcal N}_N)\le e^{C'_1\delta N+\sum_{n=1}^{N}H(W^{(n)})}.
\end{equation} 
These observations will be useful to get the sharp upper bounds for $\dim_H K_\omega$ and  $\dim_P K_\omega$. 
\end{remark}

\begin{proof}[Proof of Lemma~\ref{local control}]
Recall the definition \eqref{Peyriere} of the Peyri\`ere measure $\mathcal Q$. Fix $\delta\in (0,q-1)$. For $m\in\mathbb{N}^+$, set 
$$
G(m,\delta)=\left\{(\omega,\boldsymbol{i}): \,  \ell_m^{-1}\left |\sum_{n=L_{m-1}+1}^{L_m} \log (W_{i_n}(\boldsymbol{i}_{|n-1}))+H(W^{(L_m)})\right|>\delta\right\}.
$$
An application of Markov's inequality easily shows, using the fact that the distribution of $W^{(n)}$ is constant over $[L_{m-1}+1,L_m]$, that  
\begin{align*}
\mathcal{Q}\left(G(m,\delta)\right)&\le e^{-\ell_m\delta} e^{\ell_m\delta H(W^{(L_m)})}(\phi_{W^{(L_m)}}(1+\delta))^{\ell_m}\\
&\quad \quad +e^{-\ell_m\delta} e^{-\ell_m\delta H(W^{(L_m)})}(\phi_{W^{(L_m)}}(1-\delta))^{\ell_m}.
\end{align*}
Using the same estimates as in the proof of Theorem~\ref{dimHmu}, we can get  that there exists $C>0$ independent of $\delta$ and $m$ such that 
$$
\max \Big(\delta H(W^{(L_m)})+\log\phi_{W^{(L_m)}}(1+\delta), -\delta H(W^{(L_m)})+\log\phi_{W^{(L_m)}}(1-\delta)\Big )\le C\delta^2.
$$
So if $0<\delta <\delta_0=\min (q-1,(2C)^{-1})$,  then $\mathcal{Q}(G(m,\delta))\le e^{-\ell_m \delta/2}$. Since we assumed that $\log(m)=o(\ell_m)$, we can find a sequence $(\delta_m)_{m\in\mathbb{N}^+}\in (0,\delta_0)^{\mathbb{N}^+}$ which tends to 0 and such that 
$\sum_{M\ge 1}\sum_{m\ge M}e^{-\ell_m\delta_m}<+\infty$. It follows that $\mathcal Q (\limsup_{m\to +\infty} G(m,\delta_m))=0$, and the conclusion regarding $\lim_{m\to +\infty} \ell_m^{-1}\left |\sum_{n=L_{m-1}+1}^{L_m} \log (W_{i_n}(\boldsymbol{i}_{|n-1}))+H(W^{(L_m)})\right|$ follows from the Borel-Cantelli lemma. 

To deal with the other limit, we set $V^{(n)}(\boldsymbol{i})=(\mathbf {1}_{\{i\}}(i_n))_{i\in\mathcal I}$ and note that 
$$
\left \|p^{(n)}(\boldsymbol{i})-p^{(n)}\right \|_\infty=\max_{i\in\mathcal I}\, \ell_m^{-1}\left |\sum_{n=L_{m-1}+1}^{L_m} V_i^{(n)}(\boldsymbol{i})-p_i^{(n)}\right|.
$$ 
It is thus enough to treat each $\lim_{m\to +\infty} \ell_m^{-1}\left |\sum_{n=L_{m-1}+1}^{L_m} V_i^{(n)}(\boldsymbol{i})-p_i^{(n)}\right|$ individually. Fix $i\in\mathcal I$ and for $m\ge 1$ and $\delta>0$ 
set 
$$
G_i(m,\delta)=\left\{(\omega,\boldsymbol{i}): \,  \ell_m^{-1}\left |\sum_{n=L_{m-1}+1}^{L_m} V_i^{(n)}(\boldsymbol{i})-p_i^{(L_m)}\right|>\delta\right\}.
$$
We have 
\begin{align*}
\mathcal Q(G_i(m,\delta))&\le  e^{-\ell_m\delta} \Big (\sum_{j\in\mathcal I}\mathbb{E}(W_j^{(L_m)} )e^{\delta\mathbf{1}_{\{i\}}(j)}\Big)^{\ell_m} e^{-\ell_m\delta  p_i}+e^{-\ell_m\delta} \Big (\sum_{j\in\mathcal I}\mathbb{E}(W_j^{(L_m)}) e^{-\delta\mathbf{1}_{\{i\}}(j)}\Big)^{\ell_m} e^{\ell_m \delta p_i}\\
&=e^{-\ell_m\delta} \big (((e^{\delta}-1)p_i+1)e^{-\delta  p_i}\big )^{\ell_m}+e^{-\ell_m\delta} \big (((e^{-\delta}-1)p_i+1)e^{\delta  p_i}\big )^{\ell_m}\\
&\le e^{-\ell_m\delta/2}
\end{align*}
for $\delta$ small enough. We  conclude as for the first limit.
\end{proof}

\subsection{Proof of Theorem~\ref{dimKd};  upper bound for $\dim_H K_\omega$}\label{dimHK}

To getting suitable coverings of $K_\omega$, we will work with the  subclass  $\mathcal{M}$ of inhomogeneous Mandelbrot measures constructed with random vectors of the form 
\begin{equation}\label{Wn}
W(v)=W^{(n)}(v)=\left(p^{(n)}_i \frac{\mathbf{1}_{\{c_i(v)=1\}}}{\mathbb{P}(c_i=1)}\right)_{i\in\mathcal I}\text{ for $v\in \mathcal{I}^{n-1}$},
\end{equation}
where $\boldsymbol{p}=(p^{(n)})_{n\in\mathbb{N}^+}\in P_{\mathcal{I}}^{\mathbb{N}^+}$. Note that in this case the probability distribution of $(\widetilde W^{(n)}_i(v))_{i\in\mathcal I}=\Big (\frac{\mathbf{1}_{\{c_i(v)=1\}}}{\mathbb{P}(c_i=1)}\Big )_{i\in\mathcal I}$ does not depend on $n$. We denote $\widetilde W^{(1)}(\epsilon)$ by $\widetilde W$.  The associated inhomogeneous Mandelbrot measures $\nu$ and $\mu$ are also denoted by $\nu_{\boldsymbol{p}}$ and $\mu_{\boldsymbol{p}}$.

One has $\phi_{\widetilde W}(q)<+\infty$ for all $q>0$, so in particular for some $q\in (1,2]$ that we can fix arbitrarily. Moreover, the constant $H_\infty$ of Remark~\ref{Hinfty}, hence  the constant $C$ defined in  \eqref{constant C},  are independent of $\boldsymbol{p}$ such that $\liminf_{N\to +\infty}N^{-1}\sum_{n=1}^NH(W^{(n)})>\epsilon$.

\begin{remark}\label{mnc}If the components $p^{(n)}$, $n\ge 1$, of $\boldsymbol{p}$ are all positive, and if $\mu_{\boldsymbol{p}}$ is non degenerate, the support of $\mu_{\boldsymbol{p}}$ is equal to $K_\omega$ almost surely, and if, moreover, the $p^{(n)}_i$ are uniformly bounded from below by a real number $\eta>0$, the assumptions of Proposition~\ref{mursnonchargés} are fulfilled. 
\end{remark}

Now we can start the construction of coverings of $K_\omega$. Recall that we have fixed $\ell\in\mathscr L$. 

Fix $\epsilon>0$, and  $\eta\in (0, (\#\mathcal{I})^{-2})$ to be specified later as a function of $\epsilon$.    

Let 
$$
\mathscr{P}_{\mathcal{I}}(\epsilon)=\left \{(p^{(n)})_{n\in\mathbb{N}^+}\in P_{\mathcal{I}}^{\mathbb{N}^+}:\, \sum_{n=1}^N H(W^{(n)})\ge N\epsilon \text{ for all $N$ large enough}\right\}
$$
and  for $j\ge 1$ 
$$
\mathscr P_{\mathcal{I},j}(\epsilon)=\left \{(p^{(n)})_{n\in\mathbb{N}^+}\in \mathscr{P}_{\mathcal{I}}(\epsilon):\, \sum_{n=1}^N H(W^{(n)})\ge N\epsilon \text{ for $N\ge j$}\right\}.
$$
Let $P_{\mathcal I}(\eta)=\{(p_i)_{i\in\mathcal I}:\, p_i\ge \eta, \, \forall\, i\in\mathcal I\}$. Fix $\mathcal P_\eta\subset  P_{\mathcal I}(\eta)$ of cardinality at most $ \eta^{-\#\mathcal I}$  such that $\left \{B(q,(\#\mathcal I)\eta)\right \}_{q\in \mathcal P_\eta}$ is an $(\#\mathcal I)\eta$-covering  of $P_{\mathcal I}$; here we use the norm $\|\,\|_\infty$ on $\mathbb{R}^{\mathcal{I}}$. This is indeed possible since if $(p_i)_{i\in\mathcal I} \in P_{\mathcal I}$, picking  $i_0\in\mathcal I$  such that $p_{i_0}\ge (\#\mathcal I)^{-1}(\ge \eta)$ and setting 
$$
\widetilde p_i=
\begin{cases}
\eta&\text{if } p_i<\eta\\
\big \lfloor\frac {p_i}{\eta}\big\rfloor \eta& \text{if } p_i\ge \eta\text{ and } i\neq i_0\\
1-\sum_{i\neq i_0} \widetilde p_i& \text{if } i=i_0,
\end{cases}
$$
we leave the reader check that $\widetilde p\in \mathcal P_{\mathcal I}(\eta)$ and $\|p-\widetilde p\|_\infty\le (\#\mathcal I)\eta$ (in fact the upper bound $(\#\mathcal I -1)\eta$ holds). Moreover, there are at most $\#\mathcal I$ possibilities for $i_0$ and for each such $i_0$ at most $\eta^{-(\#\mathcal I-1)}$ probability vectors $\widetilde p$ as above, so in total at most $(\#\mathcal I)\eta^{-(\#\mathcal I-1)}\le \eta^{-\#\mathcal I}$ such vectors (what will really matter is that this number is finite). 

Let 
\begin{align*}
\mathscr{P}^{\ell,\eta}_{\mathcal{I}}&=\left \{(p^{(n)})_{n\in\mathbb{N}^+}\in P_{\mathcal{I}}^{\mathbb{N}^+}:\, \forall\, m\ge 1,\ \exists \, q \in \mathcal P_\eta,\ p^{(L_{m-1}+1)}=\cdots=p^{(L_m)}=q\right \},\\
\mathscr{P}^{\ell,\epsilon,\eta}_{\mathcal{I}}&=\mathscr{P}^{\ell,\eta}_{\mathcal{I}}\cap \mathscr{P}_{\mathcal{I}}(\epsilon), \text{ and for $j\ge 1$ }\mathscr{P}^{\ell,\epsilon,\eta}_{\mathcal{I},j}=\mathscr{P}^{\ell,\eta}_{\mathcal{I}}\cap \mathscr P_{\mathcal{I},j}(\epsilon).
\end{align*}
Note that Remark~\ref{mnc} applies to $\mu_{\boldsymbol{p}}$ if $\boldsymbol{p}\in \mathscr{P}^{\ell,\epsilon,\eta}_{\mathcal{I}}$.

For $c=\max (1,\widetilde C \Lambda_a)>0$ (where $\widetilde C$ is defined in \eqref{constant tildeC}) and $N\in\mathbb{N}^+$ set 
$$
\mathscr{P}^{\ell,\epsilon,\eta, cN}_{\mathcal{I}}=\left \{(p^{(n)})_{1\le n\le cN}: (p^{(n)})_{n\in\mathbb{N}^+}\in \mathscr{P}^{\ell,\epsilon,\eta}_{\mathcal{I}}\right \}.
$$

Note that if $L_{m-1}+1\le cN<L_m$, and if one sets $\gamma_N=c \frac{m}{L_{m-1}}$, then 
\begin{equation}\label{cardPlepsiloni}
\#\mathscr{P}^{\ell,\epsilon,\eta,  c N}_{\mathcal{I}}\le (\# \mathcal P_\eta)^m\le (\# \mathcal P_\eta)^{\gamma_N N},
\end{equation}
and $\lim_{N\to +\infty} \gamma_N=0$ (we will see that the fact that $\gamma_N$ depends on $\epsilon$ via $\widetilde C$ will not matter since we will let $N$ tend to $+\infty$ before letting $\epsilon$ tend to $0$). 

For each $\boldsymbol{p}=(p^{(n)})_{n\in\mathbb{N}^+}\in \mathscr{P}^{\ell,\epsilon,\eta}_{\mathcal{I},j}$, the associated sequence $(d_N)_{N\ge 1}$ defined in  \eqref{uN0} (with $W^{(n)}$ as in~\eqref{Wn}) is also denoted $d(\boldsymbol{p})=(d_N(\boldsymbol{p}))_{N\ge 1}$. 

Fix $\boldsymbol{i}\in\mathcal I^{\mathbb{N}^+}$, and recall the definition \eqref{pni} of $\big (p^{(n)}(\boldsymbol{i})\big)_{n\ge 1}$. We can pick $\boldsymbol{p}=(p^{(n)})_{n\in\mathbb{N}^+}\in \mathscr{P}^{\ell,\eta}_{\mathcal{I}}$ such that $\|p^{(n)}(\boldsymbol{i})-p^{(n)}\|_\infty\le (\#\mathcal I)\eta$. Since  $H(W^{(n)})=h(p^{(n)})+\sum_{i\in\mathcal I}p^{(n)} \log\mathbb{P}(c_i=1)$ and conditional on $\boldsymbol{i}\in \Sigma_\omega$, one has $\log (W_{i_n}(\boldsymbol{i}_{|n-1}))=\log (p^{(n)}_{i_n}) -\log(\mathbb{P}(c_{i_n}>0))$, for all $m\ge 1$ we get 
\begin{align}
\nonumber &\left|\ell_m^{-1}\sum_{n=L_{m-1}+1}^{L_m} \log (W_{i_n}(\boldsymbol{i}_{|n-1}))+H(W^{(n)})\right|\\
\nonumber 
&\le(\# \mathcal I)^2\eta\big (\max_{n\ge 1, i\in\mathcal I}|\log(p^{(n)}_i)|+\max_{i\in\mathcal I}|\log(\mathbb{P}(c_i=1))|\big )\\
\label{deltaeta}&\le \delta=\delta(\eta):= (\# \mathcal I)^2\eta\big (|\log(\eta)|+ \max_{i\in\mathcal I}|\log(\mathbb{P}(c_i=1))|\big ).
\end{align}

We take $\eta$ small enough so that $\delta\in (0,1)$ and $(\#\mathcal I)\eta\le \delta$. We then distinguish two cases.

\smallskip

\noindent
\textbf{Case 1}: there exists $j\ge 1$ such that  $\boldsymbol{p}\in \mathscr{P}^{\ell,\epsilon,\eta}_{\mathcal{I},j}$. 

Using the same definition of $M_\delta$ as in the alternative proof of the upper bound for $\dim_H(\mu)$ given by Theorem~\ref{thm-2.4}(1) (see Section~\ref{UPdimH}),  that is an integer such that  $\ell_{m}\le \delta L_{m-1}$ for all $m\ge M_\delta$, we can fix an integer $n_j\ge 1$, independent of $\boldsymbol{i}$ and $\boldsymbol{p}\in \mathscr{P}^{\ell,\epsilon,\eta}_{\mathcal{I},j}$, such that for all $N\ge n_j$, one has $g_1(N)\ge \max (L_{M_\delta}/\delta, j)$ (where $g_1(N)$ is associated to $\mu_{\boldsymbol{p}}$). Note that $j$ plays the role of the integer $N_\epsilon$ considered in Section~\ref{UPdimH}.

In particular, still with the notation of Section~\ref{UPdimH}, if $\boldsymbol{i}\in  \Sigma_\omega$, then $\|p^{(n)}(\boldsymbol{i})-p^{(n)}\|_\infty\le(\#\mathcal I) \eta\le \delta$ and \eqref{deltaeta} imply that $\boldsymbol{i}\in E_{\boldsymbol{p}}(M_\delta,\delta)$, where the subscript $\boldsymbol{p}$ notifies the dependence of $E(M_\delta,\delta)$ with respect to $\boldsymbol{p}$. So $\boldsymbol{i}\in E_{\boldsymbol{p}}(M_\delta,m(\widetilde g_N(s)),\delta)$ for all $N\ge n_j$. 

Also, for all $N\ge n_j$, \eqref{claim} provides an upper bound for the expectation of the number $\mathcal N_{\boldsymbol{p},N}$ of elements $B_N$ in $\mathcal F_{\boldsymbol{p},N}^{D(N)}$ which contain some $\boldsymbol{i}\in E_{\boldsymbol{p}}(M_\delta,m( \widetilde{g}_s(N)),\delta)$. Specifically, $\mathbb{E}(\mathcal{N}_{\boldsymbol{p},N})\le e^{(C\delta+d_N(\boldsymbol{p}))N}$. The parallelepiped $Q_{B_N(\boldsymbol{i})}$ associated to $B_N(\boldsymbol{i})$ has a diameter smaller than or equal to $\sqrt{d} e^{\lambda_a\Lambda_a\delta N} e^{-N}$. 
Moreover, the collection $\{B_N\}$ and the number $d_N(\boldsymbol{p})$ are  entirely determined by  $(p^{(n)})_{1\le n\le c N}$, since they are determined by $(p^{(n)})_{1\le n\le  \widetilde{g}_s(N)}$.  We denote this collection $\{B_N\}$ by $\mathcal B_N((p^{(n)})_{1\le n\le cN})$. We know from \eqref{cardPlepsiloni} that there are less than $(\# \mathcal P_\eta)^{\gamma_N N}$ such collections as $\boldsymbol{p}$ varies in $\mathscr{P}^{\ell,\epsilon,\eta}_{\mathcal{I},j}$. We denote the set of these collections by $\mathscr{\mathscr B}^j_N$, and  for $\mathcal B=\mathcal B_N((p^{(n)})_{1\le n\le cN})\in \mathscr{B}^j_N$ denote by $\mathcal N_{\mathcal B}$  and $d_N(\mathcal B)$ respectively the number $\mathcal N_{\boldsymbol{p},N}$ of elements of  $\mathcal B$ and the number $d_N(\boldsymbol{p})$ associated to  $(p^{(n)})_{1\le n\le cN}$ as above. 

\smallskip

\noindent
\textbf{Case 2}: suppose that  $\boldsymbol{p}\in \mathscr{P}^{\ell,\eta}_{\mathcal{I}} \setminus \mathscr{P}^{\ell,\epsilon,\eta}_{\mathcal{I}}$. The same reasoning as above shows that for all $N\ge 1 $ such that $N\ge L_{M_\delta}/\delta$, one has $\boldsymbol{i}\in E_{\boldsymbol{p}}(M_\delta,m(N),\delta)$. By \eqref{espN'}, the expectation of the  number $\widehat{\mathcal N}_N$ of cylinder $[U_1]$ of generation $N$ which intersect $E_{\boldsymbol{p}}(M_\delta,m(N),\delta)$ is smaller than or equal to $e^{(C'_1\delta) N} e^{\sum_{n=1}^{N}H(W^{(n)})}$. So when $\sum_{n=1}^{N}H(W^{(n)})<N\epsilon$, we have $\mathbb{E}(\widehat{\mathcal N}_N)\le e^{(C'_1\delta+\epsilon) N}$. Moreover, these cylinders project via $\pi$ onto sets of diameter less than $\sqrt{d}e^{-N/\Lambda_a}$. This collection of cylinders depends only on $(p^{(n)})_{1\le n\le N}$. Denote it by $\widehat{\mathcal B}_N((p^{(n)})_{1\le n\le N})$. Again, there are at most $(\# \mathcal P_\eta)^{\gamma_N N}$ such collections as $(p^{(n)})_{n\in\mathbb{N}^+}$ varies in $\mathscr{P}^{\ell,\eta}_{\mathcal{I}} \setminus \mathscr{P}^{\ell,\epsilon,\eta}_{\mathcal{I}}$. Denote the set of these collections by $\widehat{\mathscr{B}}_N$, and  for $\widehat{\mathcal B}\in \widehat{\mathscr{B}}_N$ denote by $\widehat{\mathcal N}_{\widehat{\mathcal B}}$ and $(W^{(n)}_{\widehat{\mathcal B}})_{1\le n\le N}$ respectively the number of elements of  $\widehat{\mathcal B}$ and the sequence of associated random vectors.

Now we can estimate $\dim_HK_\omega$ from above. As we noticed above, Remark~\ref{mnc} applies to $\mu_{\boldsymbol{p}}$ if $\boldsymbol{p}\in \mathscr{P}^{\ell,\epsilon,\eta}_{\mathcal{I}}$, hence Theorem~\ref{thm-2.4}(2) applies and $\dim_H(\mu_{\boldsymbol{p}})=\liminf_{N\to+\infty}d_N(\boldsymbol{p})$. Let  $D_{\epsilon,\eta}=\sup\{\liminf_{N\to+\infty}d_N(\boldsymbol{p}): \, \boldsymbol{p}\in \mathscr{P}^{\ell,\epsilon,\eta}_{\mathcal{I}}\}$.

The previous discussion (Case 1 and Case 2) shows that, by considering for each generations $N$ the elements $\boldsymbol{p}\in \mathscr{P}^{\ell,\epsilon,\eta,cN}_{\mathcal{I}}$ for which $d_N(\boldsymbol{p})\le D_{\epsilon,\eta}+\epsilon$, one has $K_\omega\subset \Big (\bigcup_{j\ge 1}E_j\Big )\cup \widehat E$,
where
\begin{align*}
E_j&=\bigcap_{J\ge n_j}\bigcup_{N\ge J} \bigcup_{\substack{\mathcal B\in \mathscr{B}^j_N\\ d_N(\mathcal B)\le D_{\epsilon,\eta}+\epsilon}}\bigcup_{B\in \mathcal B} Q_B\\
\widehat E&= \bigcap_{J\ge L_{M_\delta}/\delta}\bigcup_{N\ge J} \bigcup_{\substack{\widehat{\mathcal B}\in \widehat{\mathscr{B}}_N\\\sum_{n=1}^{N}H(W_{\widehat{\mathcal B}}^{(n)})<N\epsilon}}\bigcup_{[U]\in \widehat{\mathcal B}} \pi([U]).
\end{align*}

Using that $\max(\# \mathscr{B}^j_N, \#\widehat{\mathscr{B}}_N)\le (\# \mathcal P_\eta)^{\gamma_N N}$, for all $j\ge 1$ and $J\ge n_j$ we get 
\begin{align*}
&\sum_{N\ge J} \sum_{\substack{\mathcal B\in \mathscr{B}^j_N\\ d_N(\mathcal B)\le D_{\epsilon,\eta}+\epsilon}} e^{-(C\delta+2\epsilon +D_{\epsilon,\eta})N} \mathbb{E}(\mathcal N_{\mathcal B})\\
&\le \sum_{N\ge J}(\# \mathcal P_\eta)^{\gamma_N N} e^{-(C\delta+2\epsilon +D_{\epsilon,\eta})N}e^{(C\delta+\epsilon +D_{\epsilon,\eta})N}<+\infty
\end{align*}
and for all  $J\ge L_{M_\delta}/\delta$, 
$$
\sum_{N\ge J} \sum_{\substack{\widehat{\mathcal B}\in\widetilde{ \mathscr{B}}_N\\ \sum_{n=1}^{N}H(W_{\widehat{\mathcal B}}^{(n)})<N\epsilon}} e^{-(C'_1\delta+2\epsilon)N} \mathbb{E}(\widehat{\mathcal N}_{\widehat{\mathcal B}})\le \sum_{N\ge J}(\# \mathcal P_\eta)^{\gamma_N N} e^{-(C'_1\delta+2\epsilon)N}e^{(C'_1\delta+\epsilon)N}<+\infty.
$$

Consequently, by the Borel-Cantelli Lemma, with probability~1, for all $j\ge 1$, for $N$ large enough,  for all $\mathcal B\in\mathscr{B}^j_N$ such that $d_N (\mathcal B)\le D_{\epsilon,\eta}+\epsilon$ one has $\mathcal N_{\mathcal B}\le e^{(C\delta+2\epsilon +D_{\epsilon,\eta})N}$, and for $N$ large enough, for all $\widehat{\mathcal B}\in \widehat{\mathscr{B}}_N$ such that $ \sum_{n=1}^{N}H(W_{\widehat{\mathcal B}}^{(n)})<N\epsilon$ one has $\widehat{\mathcal N}_{\widehat{\mathcal B}}\le e^{(C'_1\delta+2\epsilon)N}$. This, together with the fact that $\lim_{N\to +\infty} \gamma_N/N=0$ and the estimates provided in the above discussion for the diameters of the elements of any collection $\mathcal B$ or $\widehat{\mathcal B}$ is enough to show that if $\eta$ is small enough so that $1-\lambda_a\Lambda_a\delta(\eta)>0$, with probability~1, for any real number $s>s(\epsilon,\eta,\delta)=\max \big (\frac{D_{\epsilon,\eta}+C\delta+2\epsilon}{1-\lambda_a\Lambda_a\delta},\Lambda_a(C'_1\delta+2\epsilon)\big )$, one has, for $J$ large enough,  
\begin{align*}
&\sum_{N\ge J} \sum_{\substack{\mathcal B\in \mathscr{B}^j_N\\ d_N(\mathcal B)\le D_{\epsilon,\eta}+\epsilon}}\sum_{B\in \mathcal B} |Q_B|^s<+\infty\\
\text{ and }&\sum_{N\ge J} \sum_{\substack{\widehat{\mathcal B}\in \widehat{\mathscr{B}}_N\\\sum_{n=1}^{N}H(W_{\widehat{\mathcal B}}^{(n)})<N\epsilon}}\sum_{[U]\in \widetilde{\mathcal B}} |\pi([U])|^s<+\infty
\end{align*}
(we leave the detail of this simple calculation to the reader). Since the supremum of the diameters of the sets involved in the above sums tend to 0 as $J\to +\infty$, this implies that $\dim_H E\le s$ for all $E\in\{E_j:\, j\ge 1\}\cup\{\widehat E\}$. Remembering the expression $\eqref{constant C}$ for the constant $C$ and \eqref{C'1} for $C'_1$, we see that due to the relation $\delta=\delta(\eta)$ \eqref{deltaeta} between $\delta$ and $\eta$, taking $\epsilon\in(0,(\#\mathcal{I})^{-1})$ and $\eta=\epsilon^2$ such that $\delta(\eta)<1$ yields $s(\epsilon,\epsilon^2,\delta(\epsilon^2))=D_{\epsilon,\epsilon^2}+O(\epsilon)$. Consequently, denoting by $D$ the supremum of the Hausdorff dimensions of elements of $\mathcal M$ and letting $\epsilon$ tend to $0$, we get  $\dim_H K_\omega\le D$ (note that in fact the previous lines show that $\dim_H\widehat E= 0$).

\section{Proofs of Theorems~\ref{dimK2} and~\ref{dimKd''}} \label{secdim2}

\subsection{Proof of Theorem~\ref{dimK2}}\label{secdimK2} Recall that the result was obtained in \cite{ref10,ref1} for the deterministic case, and \cite{ref6}  for random Sierpi\'nski carpets. We will derive it in general from Theorem~\ref{dimKd}. To do so we adapt to our context the approach used in \cite{ref5} to prove that for the deterministic case, in dimension 2, the supremum of the Hausdorff dimensions of exponentially periodic Bernoulli measures supported on $K$ does not exceed that of the supremum of the Hausdorff dimensions of self-affine measures. For the random case,  the situation is  a little more involved due to the fact that one must consider a minimum in the definition of  each term of the sequence $(d_N)_{N\ge 1}$ (see \eqref{uN0}) associated with any  IMM of class $\mathcal M$. 

By the proof of Theorem~\ref{dimKd}, for all $\epsilon>0$ small enough and $\eta=\eta(\epsilon)=\epsilon^2$, one has 
\begin{equation}\label{(.)}
\dim_H K_\omega\le \sup\big\{ \dim_H(\mu_{\boldsymbol{p}})=\liminf_{N\to +\infty} d_N(\boldsymbol p): \boldsymbol{p} \in \mathscr{P}^{\ell,\epsilon,\eta}_{\mathcal{I}}\big \} +O(\epsilon).
\end{equation} 
Fix   $\boldsymbol{p}\in \mathscr{P}^{\ell,\epsilon,\eta}_{\mathcal{I}}$.  For each $N\ge 1$, we simply denote $s(N)$, which belongs to $\{1,2\}$, by $s$ and $D(N)$ by $D$; the components of $(D(N),s(N))_{N\ge 1}$ takes at most three values: $((\{1,2\}),1)$, $((\{1\},\{2\}),2)$ and $((\{2\},\{1\}),2)$. Set 
$$
\bar d_N(\boldsymbol{p})= \min (d_{1,N}(\boldsymbol{p}),d_{2,N}(\boldsymbol{p})), $$
where 
\begin{align*}
d_{1,N}(\boldsymbol{p})&=\frac{1}{N}\sum_{n=1}^{g_1(N)}H(W^{(n)})+\frac{1}{N}\sum_{n=g_1(N)+1}^{g_{s}(N)}h(\Pi^{D}_2p^{(n)})\\
d_{2,N}(\boldsymbol{p})&=\frac{1}{N}\sum_{n=1}^{g_{s}(N)}H(W^{(n)}).
\end{align*} 
Note that by definition of $d_N(\boldsymbol{p})$, one has $d_N(\boldsymbol{p})\le \bar d_N$. Also, note that $d_{1,N}$ and $d_{2,N}$ coincide when $s(N)=1$.

Suppose that $s(N)=2$ and write $D(N)=(\{k_1\},\{k_2\}$). Recall that $\widetilde W_i= \frac{\mathbf{1}_{\{c_i=1\}}}{\mathbb{P}(c_i=1)}$ for all $i\in\mathcal I$. Using the concavity of the functions 
\begin{equation*}
\widetilde h: p\in P_{\mathcal I}\mapsto H((p_i\widetilde W_i)_{i\in\mathcal I})=h(p)+\sum_{i\in\mathcal I}p_i\log (\mathbb{P}(c_i=1))
\end{equation*} 
and $h$, and writing $g_r$ for $g_r(N)$ ($r\in\{1,2\}$) we get
$$
d_{1,N}(\boldsymbol{p})\le \frac{g_1}{N}\widetilde h\left(\frac{1}{g_1}\sum_{n=1}^{g_1}p^{(n)}\right )+\frac{g_2-g_1} {N} h\left(\Pi^D_2\Big (\frac{1}{g_2-g_1} \sum_{n=g_1+1}^{g_2}p^{(n)}\Big)\right).
$$
Moreover, using the concavity of $h$ again, we get 
$$
\frac{g_2-g_1}{g_2}h\left(\frac{1}{g_2-g_1} \sum_{n=g_1+1}^{g_2}\Pi^D_2p^{(n)}\right)\le h\left(\Pi^D_2\Big (\frac{1}{g_2} \sum_{n=1}^{g_2}p^{(n)}\Big )\right )-\frac{g_1}{g_2} h\left(\Pi^D_2\Big( \frac{1}{g_1} \sum_{n=1}^{g_1}p^{(n)}\Big )\right ).
$$
The two above inequalities and the definition of the Lyapunov exponents  yield
\begin{equation}\label{comparaison}
d_{1,N}(\boldsymbol{p})\le T_{1,N}(\boldsymbol{p})+ T_{2,N}(\boldsymbol{p}) +o(1),
\end{equation}
where, using the notation  $\widehat{\boldsymbol{p}}_N=N^{-1}\sum_{n=1}^N p^{(n)}$, 
\begin{align*}
T_{1,N}(\boldsymbol{p})&=\frac{1}{\chi_{k_1}(\widehat{\boldsymbol{p}}_{g_1(N)})}\left (\widetilde h(\widehat{\boldsymbol{p}}_{g_1(N)})-  h\big (\Pi^{D(N)}_2\widehat{\boldsymbol{p}}_{g_1(N)}\big )\right )\\
T_{2,N}(\boldsymbol{p})&= \frac{1}{\chi_{k_2}(\widehat{\boldsymbol{p}}_{g_2(N)})} h\big ( \Pi^{D(N)}_2\widehat{\boldsymbol{p}}_{g_2(N)}\big ),
\end{align*}
and we remark that \eqref{comparaison} holds as well  when $s=1$ (actually, as an equality). Similarly,
$$
d_{2,N}(\boldsymbol{p})\le \frac{1}{\chi_{k_2}(\widehat{\boldsymbol{p}}_{g_2(N)})} \widetilde h( \widehat{\boldsymbol{p}}_{g_2(N)}) +o(1).
$$

If $D(N)$ takes infinitely many times the value$((\{1,2\}),1)$ along a subsequence $(N_j)_{j\ge 1}$, then by Theorem~\ref{thm-2.4}(3), $d_{1,N_j}(\boldsymbol{p})= d_{2,N_j}(\boldsymbol{p})=d(\mu_j)+o(1)$ where $\mu_j$ is the Mandelbrot measure associated with $((\widehat{\boldsymbol{p}}_{g_2(N_j)})_i\widetilde W_i)_{i\in\mathcal I}$ (note that in this situation $g_2(N_j)=g_1(N_j)$ and that the components of $\widehat{\boldsymbol{p}}_{g_2(N)}$ are positive, since $\boldsymbol{p}\in \mathscr{P}^{\ell,\epsilon,\eta}_{\mathcal{I}}$, so that a.s., conditional on $\mu_j\neq 0$, $\mu_j$ is fully supported on $K_\omega$). 

If $s(N)=2$ for $N$ large enough, fix $D=(\{k_1\},\{k_2\}) \in \{(\{1\},\{2\}), (\{2\},\{1\})\}$ such that $D(N)$ takes infinitely often the value $D$. Consider  $\underline d=\liminf_{M\to +\infty} \theta(M)$, where $\theta(M)=\frac{1}{\chi_{k_1}(\widehat{\boldsymbol{p}}_{M})}\left (\widetilde h(\widehat{\boldsymbol{p}}_{M})-  h(\Pi^{D}_2\widehat{\boldsymbol{p}}_{M})\right)$. Suppose first that there is a strictly  increasing  sequence $(N_j)_{j\ge 1}$ such that both $D(N_j)=D$ for all $j\ge 1$ and $\lim_{j\to +\infty}\theta(g_1(N_j))=\underline d$. Then, as $\liminf_{j\to +\infty} \theta(g_2(N_j)) \ge \underline d$, we deduce from \eqref{comparaison} that 
$$
\liminf_{j\to +\infty} d_{N_j}(\boldsymbol{p})\le\liminf_{j\to +\infty} d(j)
$$
where 
$$
d(j)=\min \Big (\theta(g_2(N_j))+  \frac{1}{\chi_{k_2}(\widehat{\boldsymbol{p}}_{g_2(N_j)})} h( \Pi^{D(N_j)}_2\widehat{\boldsymbol{p}}_{g_2(N_j)}), \frac{1}{\chi_{k_2}(\widehat{\boldsymbol{p}}_{g_2(N_j)})} \widetilde h( \widehat{\boldsymbol{p}}_{g_2(N_j)})\Big ).
$$
Moreover, $d(j)=\dim(\mu_j)+o(1)$, where $\mu_j$ is the same Mandelbrot measure as above. This is due to the fact that by definition of $g_1(N_j)$ and $g_2(N_j)$, one has $g_2(N_j)\chi_{k_1}(\widehat{\boldsymbol{p}}_{g_2(N_j)})\ge g_1(N_j)\chi_{k_1}(\widehat{\boldsymbol{p}}_{g_1(N_j)})=N_j+O(1)=g_2(N_j)\chi_{k_2}(\widehat{\boldsymbol{p}}_{g_2(N_j)})$, so that  either $\chi_{k_1}(\widehat{\boldsymbol{p}}_{g_2(N_j)})> \chi_{k_2}(\widehat{\boldsymbol{p}}_{g_2(N_j)})$, or $g_2(N_j)=g_1(N_j)+O(1)$ so that $\|\widehat{\boldsymbol{p}}_{g_2(N_j)}-\widehat{\boldsymbol{p}}_{g_1(N_j)}\|_\infty=o(1)$ and  if one denotes by $\widetilde \mu_j$ the Mandelbrot measure associated with $((\widehat{\boldsymbol{p}}_{g_1(N)})_i\widetilde W_i)_{i\in\mathcal I}$, $d(j)=\dim(\mu_j)+o(1)=\dim(\widetilde \mu_j)+o(1)$.

Finally, suppose that there is no sequence $(N_j)_{j\ge 1}$ as above. This implies that $D(N)$ is not stationary, so we can find a strictly  increasing  sequence $(N_j)_{j\ge 1}$ such that $D(N_j)$ and $D(N_j+1)$ are different for all $j\ge 1$. By construction, the difference between $g_1(N_j)$ and $g_2(N_j)$ is then  bounded independently of $j$, so $\lim_{j\to +\infty} \theta(g_2(N_j))-\theta(g_1(N_j))=0$, and the same argument as above yields $\liminf_{j\to +\infty} d_{N_j}(\boldsymbol{p})\le\liminf_{j\to +\infty}\max ( \dim(\mu_j), \dim(\widetilde \mu_j))$. 

The three cases distinguished above yield that $\liminf_{N\to +\infty} d_{N}(\boldsymbol{p})$ is bounded by the supremum of the Hausdorff dimensions of Mandelbrot measures fully supported on $K$. This holds for all $\boldsymbol{p}\in \mathscr{P}^{\ell,\epsilon,\eta}_{\mathcal{I}}$, hence letting $\epsilon$ tend to 0 in \eqref{(.)} yields the desired variational principle. 

Now, take a sequence $(p(j))_{j\ge 1}$ of positive elements of $P_{\mathcal I}$ such that if for $j\ge 1$ one denotes by $\mu_j$ the Mandelbrot measure associated with the random vectors $W(j)(v)=(p_i(j)\frac{\mathbf{1}_{\{c_i(v)=1\}}}{\mathbb{P}(c_i=1)})_{i\in\mathcal I}$, $v\in\mathcal I^*$, then $\mu_j$ is non degenerate and fully supported on $K_\omega$ conditional on $\{K_\omega\neq\emptyset\}$, and $\lim_{j\to+\infty}\dim(\mu_j)=\dim_HK_\omega$. Without loss of generality we can assume that for all $j\ge 1$ one has $\chi_1(p(j))\ge \chi_2(p(j))$. The set $P_{\mathcal I}$ being compact, without loss of generality again, we can also assume that $p(j)$ converges to a probability vector $p$ as $j\to+\infty$. Set $W(v)=(p_i\frac{\mathbf{1}_{\{c_i(v)=1\}}}{\mathbb{P}(c_i=1)})_{i\in\mathcal I}$ for all $v\in\mathcal I^*$, and consider the associated Mandelbrot measure $\mu$. The value of $\dim(\mu_j)$ provided Theorem~\ref{thm-2.4}(3) converges to $D(p)=\frac{1}{\chi_1(p)}H(W)+\big(\frac{1}{\chi_2(p)}-\frac{1}{\chi_1(p)}\big) \min (H(W),h(\Pi_2p))$ as $j\to+\infty$. Hence $D(p)=\dim_HK_\omega>0$ so $H(W)>0$ and $\mu$ is non degenerate. Moreover, due to the expression of $D(p)$, it is not hard to prove that when $D(p)$ attains its maximal value one necessarily has that $p$ is an interior point of $P_{\mathcal I}$ (this is due to the convexity of  $P_{\mathcal I}$  and the fact that the derivative of $t\ge 0\mapsto-t\log(t)$ at $0^+$ is infinite, which forbids the maximum of $D(\cdot)$ to be attained at a  the boundary point of $P_{\mathcal I}$), so that the associated Mandelbrot measure $\mu$ satisfies $\mathbb{P}(\mu\neq 0)=\mathbb{P}(K_\omega\neq\emptyset)$. Finally, since the assumption of  Proposition~\ref{mursnonchargés} holds for $\mu$, by Theorem~\ref{thm-2.4}(3) one has $\dim(\mu)=\dim_H K_\omega$ conditional on $\{K_\omega\neq\emptyset\}$.  

\subsection{Proof of  Theorem~\ref{dimKd''}}\label{secdimKd''}  In~\cite{ref6}, in the case of random Sierpi\'nski sponges,  after having established in this special context the Ledrappier-Young type formula provided by Theorem~\ref{thm-2.4}(3), one starts by identifying the unique couple $(C,W)$ which generates the Mandelbrot measure $\mu$ with maximal Hausdorff dimension on the attractor $K_\omega$. This dimension is expressed as the weighted pressure of some potential (in the terminology of weighted thermodynamic formalism~\cite{BFAJM}). Then one constructs an uncountable family of random coverings of $K_\omega$, each of which providing an upper bound for $\dim_HK_\omega$ expressed as the weighted pressure of some potential. The infimum of theses values  is then  directly  identified with the  dimension of $\mu$. As mentioned in the introduction, this approach can be extended to the more general class of sponges considered in Theorem~\ref{dimKd''}. Along the lines to follow, we reverse the point of view. We start from the fact that the supremum of the Hausdorff dimensions of IMMs supported on $K_\omega$ is an upper bound for $\dim_HK_\omega$; then from this supremum we quite easily recover the  family of upper bounds mentioned above, and considering their infimum we naturally exhibit a Mandelbrot measure of maximal Hausdorff dimension. For the uniqueness of $(C,W)$ to which can be associated a Mandelbrot measure of maximal Hausdorff dimension, we refer to the approach used in~\cite{ref6}, which still works in the present context.

By the proof of Theorem~\ref{dimKd} again, for $\epsilon>0$ small enough and $\eta=\eta(\epsilon)=\epsilon^2$,  one has $\dim_H K_\omega\le\sup\big \{ \liminf_{N\to +\infty}\widetilde d_N(\boldsymbol p): \boldsymbol{p} \in \mathscr{P}^{\ell,\epsilon,\eta}_{\mathcal{I}}\big \} +O(\epsilon)$, where $
\widetilde d_N(\boldsymbol{p})= \min\{N^{-1}H_{N,k}:\, g_1(N)\le k\le g_s(N)\}
$ was defined in \eqref{uN1}.

Fix the IMM in the class $\mathcal M$ associated with $\boldsymbol{p}\in \mathscr{P}^{\ell,\epsilon,\eta}_{\mathcal{I}}$. Note that since the linear parts $A_i$, $i\in\mathcal I$, are equal, for $N$ large enough $s(N)$ and $D(N)$ are independent of $N$ and~$\nu$,  and for all probability vectors~$p$, the exponents $\widetilde \chi_r(p)$ do not depend on $p$ and are given by $(\widetilde \chi_r)_{1\le r\le s}=(-\log (|a_{1,k_r}|)_{1\le r\le s}$, where the $|a_{1,k_r}|$, ${1\le r\le s}$, are the absolute values of the eigenvalues ordered in the increasing order and counted without multiplicity.  In particular, $g_r(N)/N\to 1/\widetilde \chi_r$ as $N\to +\infty$. Without loss of generality we assume that we are in the non-conformal case, so that $s\ge 2$. 

Fix $2\le r\le s$ as well as $\theta\in\big  [\frac{\widetilde \chi_{r}}{\widetilde \chi_{r-1}},1\big ]$. Then fix $g_{r-1}+1\le k\le g_r(N)$ such that $\theta_k= \frac{k}{g_r(N)}$ satisfies $|\theta_k-\theta|\le 1/g_r(N)$.  

Denoting $g_{r'}(N)$ by $g_{r'}$, and using similar concavity inequalities as in the previous section, we can write
\begin{align}
\nonumber H_{N,k}&=\sum_{n=1}^{\theta_k g_r}\widetilde h(p^{(n)})+ \sum_{n=\theta_k g_r +1}^{g_r} h(\Pi^D_r p^{(n)})+\sum_{r'=r+1}^s \sum_{n=g_{r'-1}+1}^{g_{r'}}h(\Pi^D_{r'} p^{(n)})\\
\label{ineqHNk}&\le \theta_k g_r\widetilde h(\widehat{\boldsymbol{p}}_{\theta_k g_r})+g_r h(\Pi^D_r \widehat{\boldsymbol{p}}_{g_r}) - \theta_k g_r h(\Pi^D_r \widehat{\boldsymbol{p}}_{\theta_k g_r})\\
\nonumber&\quad +\sum_{r'=r+1}^s \big (g_{r'} h(\Pi^D_{r'} \widehat{\boldsymbol{p}}_{g_{r'}})- g_{r'-1} h(\Pi^D_{r'} \widehat{\boldsymbol{p}}_{g_{r'-1}})\big ).
\end{align}

Similarly to what was done in the proof of Proposition~\ref{tau}, define for $j\in \mathcal I_r=\Pi_r(\mathcal I)$ and  $i\in \Pi_r^{-1}(\{j\})$ (we write $\Pi_r$ for $\Pi^D_r$ and $\mathcal I_r$ for $\mathcal I^D_r$)
$$
(V_{r})_{i,j}=
\begin{cases}
\displaystyle \frac{ (\widehat{\boldsymbol{p}}_{\theta_k g_r})_i\widetilde W_i}{(\Pi_r\widehat{\boldsymbol{p}}_{\theta_k g_r})_{j}}&\text{if }(\Pi_r\widehat{\boldsymbol{p}}_{\theta_k g_r})_{j}\neq 0\\
0&\text{otherwise}.
\end{cases}
$$
Setting for $q\ge 0$
$$
T_{(V_{r})_j}(q)=-\log \mathbb{E}\Big (\sum_{i\in \Pi_r^{-1}(\{j\})} (V_{r})^q_{i,j}\Big ),
$$
a calculation shows that
$$
\widetilde h(\widehat{\boldsymbol{p}}_{\theta_k g_r})=h(\Pi_r\widehat{\boldsymbol{p}}_{\theta_k g_r})+ \sum_{j\in \mathcal I_r}  (\Pi_r\widehat{\boldsymbol{p}}_{\theta_k g_r})_{j}T'_{(V_{r}),j}(1).
$$
Moreover, since $T_{(V_{r})_j}$ is concave and by construction $T_{(V_{r})_j}(1)=0$, we have 
$$
T'_{(V_{r}),j}(1)\le -T_{(V_{r}),j}(0)=\log(\mathbb{E}(N_{r,j})), $$
where 
$$
N_{r,j}=\#\{i\in \Pi_r^{-1}(\{j\}):\, c_i=1\}.
$$
Thus, setting 
\begin{align*}
R_N(r,\theta_k)&=  \theta_k g_r\sum_{j\in \mathcal I_r}  \big((\Pi_r\widehat{\boldsymbol{p}}_{\theta_k g_r})_{j}- (\Pi_r\widehat{\boldsymbol{p}}_{g_r})_{j}\big )\log(\mathbb{E}(N_{r,j}))\\
&\quad +\sum_{r'=r+1}^s \Big(g_{r'} \big (h(\Pi^D_{r'} \widehat{\boldsymbol{p}}_{g_{r'}}) - h(\Pi^D_{r'} \widehat{\boldsymbol{p}}_{g_{r}})\big )- g_{r'-1} \big (h(\Pi^D_{r'} \widehat{\boldsymbol{p}}_{g_{r'-1}})-h(\Pi^D_{r'} \widehat{\boldsymbol{p}}_{g_{r}})\big )\Big).
\end{align*}
we get from~\eqref{ineqHNk} that 
\begin{align*}
H_{N,k}&\le g_r\sum_{j\in \mathcal I_r}  (\Pi_r\widehat{\boldsymbol{p}}_{g_r})_{j} \theta_k\log(\mathbb{E}(\# N_{r,j})) +g_r  h(\Pi^D_{r'} \widehat{\boldsymbol{p}}_{g_r})+\sum_{r'=r+1 }^s  (g_{r'}-g_{r'-1}) h(\Pi^D_{r'} \widehat{\boldsymbol{p}}_{g_r}) \\
&\quad+ R_{N}(r,\theta_k).
\end{align*}

Denote by $\eta_{r,N}$ the Bernoulli product measure on ${\mathcal I}_r^{\mathbb{N}^+}$ associated with the probability vector $\Pi^D_r\widehat{\boldsymbol{p}}_{g_r(N)}$ and by $\varphi_{r}$ the potential defined over ${\mathcal I_r}^{\mathbb{N}^+}$ as being constant and equal to $\frac{1}{\widetilde \chi_r} \log(\mathbb{E}(N_{r,j}))$ over each cylinder $[j]$ of the first generation.  The previous inequality yields (using that $g_{r'}(N)/N\to 1/\widetilde \chi_{r'}$ as $N\to+\infty$)
\begin{align}
\label{prepressure}\widetilde d_N(\boldsymbol{p})&\le S(\theta, \eta_{r,N}) +\frac{R_{N}(r,\theta)}{N}+ o(1),
\end{align}
where for any $T_r$-invariant probability measure $\eta$ on $\mathcal I_r^{\mathbb{N}^+}$, 
$$
S(\theta, \eta)= \int \theta\varphi_{r}\,\mathrm{d}\eta+ \frac{h(\eta,T_r)}{\widetilde \chi_r}+ \sum_{r'=r+1}^s\Big (\frac{1}{\widetilde \chi_{r'}}- \frac{1}{\widetilde \chi_{r'-1}}\Big ) h(\Pi^D_{r,r'}\eta,T_{r'}),
$$
and $\Pi_{r,r'}=\Pi_{r'-1,r'}\circ\cdots\circ \Pi_{r,r+1}$. Using the terminology of \cite{BFAJM}, set $\vec{\boldsymbol{\gamma}}_r=(\frac{1}{\widetilde \chi_r}, \frac{1}{\widetilde \chi_{r+1}}- \frac{1}{\widetilde \chi_{r}},\ldots,  \frac{1}{\widetilde \chi_{s}}- \frac{1}{\widetilde \chi_{s-1}})$ and 
\begin{equation}\label{wtherm}
P_r^{\vec{\boldsymbol{\gamma}}_r}(\theta\varphi_{r},T_r)=\sup\big\{S(\theta, \eta): \, \text{$\eta$ is a $T_r$-invariant probability measure on $\mathcal{I}_r^{\mathbb{N}^+}$}\big\};
\end{equation}
this supremum is called the $\vec{\boldsymbol{\gamma}}_r$-weighted topological pressure of $\theta\varphi_{r}$. It is attained at a unique fully supported Bernoulli product measure $\eta_{\theta,r}$ on $\mathcal I_r^{\mathbb{N}^+}$ (see~\cite{BFAJM}) generated by a probability vector that we denote by $p_{\theta,r}$. 

We thus deduce from \eqref{prepressure} that  
$$
\widetilde d_N(\boldsymbol{p})\le P^{\vec{\boldsymbol{\gamma}}_r}(\theta\varphi_r,T_r) +\frac{R_{N}(r,\theta)}{N}+ o(1).
$$
The term $\frac{R_{N}(r,\theta)}{N}$ can easily be written under the form $\sum_{p=1}^P u_p(\lfloor \alpha_p N\rfloor)-u_p(\lfloor \beta_p N\rfloor) +\delta_N$, where for each $p$ one has $\lim_{N\to +\infty} u_p(N)-u_p(N-1)=0$ and $(\alpha_p, \beta_p)\in(\mathbb{R}_+^*)^2$, and $\lim_{N\to +\infty} \delta_N=0$. According to a slight extension (see \cite[Lemma 5.4]{FHJMPA}) of a combinatorial lemma first considered by Kenyon and Peres in~\cite{ref18} in the study of deterministic Sierpi\'nski sponges, this implies that $\liminf_{N\to +\infty} \frac{1}{N}R_{N}(r,\theta)\le 0$. Consequently,
$
\liminf_{n\to +\infty} \widetilde d_N(\boldsymbol{p})\le P_r^{\vec{\boldsymbol{\gamma}}_r}(\theta\varphi,T_r)$ for all $2\le r\le s$ and $\theta\in \big [\frac{\widetilde \chi_{r}}{\widetilde \chi_{r-1}},1\big ]$. Thus
\begin{equation}\label{UPS}
\liminf_{n\to +\infty} \widetilde d_N(\boldsymbol{p})\le\inf_{2\le r\le s}\inf_{\theta\in [\frac{\widetilde \chi_{r}}{\widetilde \chi_{r-1}},1]} P_r^{\vec{\boldsymbol{\gamma}}_r}(\theta\varphi_{r},T_r).
\end{equation}
For each $2\le r\le s$, by continuity of  $P_r:\theta\in(0,\infty)\mapsto P_r^{\vec{\boldsymbol{\gamma}}_r}(\theta\varphi_{r},T_r)$, the infimum $\inf_{\theta\in [\frac{\widetilde \chi_{r}}{\widetilde \chi_{r-1}},1]} P_r^{\vec{\boldsymbol{\gamma}}_r}(\theta\varphi_{r},T_r)$ is a minimum. Let $2\le r_0\le s$ and $\theta_{r_0}\in[\frac{\widetilde \chi_{r_0}}{\widetilde \chi_{r_0-1}},1]$ be such that the right hand side of \eqref{UPS} equals  $P_{r_0}^{\vec{\boldsymbol{\gamma}}_{r_0}}(\theta_{r_0}\varphi_{r_0},T_{r_0})=P_{r_0}(\theta_{r_0})$. 

We can associate to each $(\theta,r)$ a Mandelbrot measure  $\nu_{\theta,r}$ by defining, for $j\in \mathcal I_{r}$, $i\in\Pi_{r}^{-1}(\{j\})$ and $v\in \mathcal{I}^*$, 
$$
W^{\theta,r}_i(v)=(p_{\theta,r})_j \frac{\mathbf{1}_{\{c_i(v)=1\}}}{\mathbb{E}(\# N_{r,j})}=(p_{\theta,r})_j \frac{\mathbb{P}(c_i=1)}{\mathbb{E}(\# N_{r,j})} \frac{\mathbf{1}_{\{c_i(v)=1\}}}{\mathbb{P}(c_i=1)}.
$$
The Mandelbrot measure $\mu_{\theta_{r_0},r_0}$ is non degenerate (this is justified below), and since the components of $W=W^{\theta_{r_0},r_0}_i$ have positive expectations (equal to $(p_{\theta_{r_0},r_0})_j \frac{\mathbb{P}(c_i=1)}{\mathbb{E}(\# N_{r_0,j})}$ with the previous notation), one has both that $\mu_{\theta_{r_0},r_0}$ is fully supported on $K_\omega$, conditional on $K_\omega\neq\emptyset$, and that Proposition~\ref{mursnonchargés} applies to $\mu_{\theta_{r_0},r_0}$. Moreover,  Theorem~\ref{thm-2.4}(3) implies that it is exact dimensional,  with dimension $P_{r_0}(\theta_{r_0})$ (this value is justified below as well).  Consequently, letting $\epsilon$ tend to $0$ in the inequality $\dim_H K_\omega\le\sup\big \{ \liminf_{N\to +\infty}\widetilde d_N(\boldsymbol p): \boldsymbol{p} \in \mathscr{P}^{\ell,\epsilon,\eta}_{\mathcal{I}}\big \} +O(\epsilon)$ yields the desired result in terms of realizing the supremum in Theorem~\ref{dimKd''} as a maximum attained by choosing $\mu_{\theta_{r_0},r_0}$. 

Now let us justify that $\nu_{\theta_{r_0},r_0}$ is non degenerate and that $\dim(\mu_{\theta_{r_0},r_0})=P_{r_0}(\theta_{r_0})$.

We first make some observations based on the thermodynamic formalism. 

\noindent 
$(i)$ For $2\le r\le s-1$ one has $P_r(1)=P_{r+1}\big (\frac{\widetilde \chi_{r+1}}{\widetilde \chi_r}\big )$. This is obtained by using the relativized thermodynamic formalism~(see \cite{LeWa77}, and \cite[Theorem 3.1]{BFAJM}) and by conditioning on $(\Pi_{r,r+1})_*\eta$ in seeking for the  measure $\eta$ at which $P_r(1)=P_r^{\vec{\boldsymbol{\gamma}}_r}(\varphi_{r},T_r)$ is attained in \eqref{wtherm}). As a result, if $j\in\mathcal I_{r+1}$ and $i\in\mathcal I_{r}$ are related by $j=\Pi_{r,r+1}(i)$, one has $(p_{1,r})_i=\big (p_{\frac{\widetilde \chi_{r+1}}{\widetilde \chi_r},r+1}\big)_j\frac{\mathbb{E}(N_{r,i})}{\mathbb{E}(N_{r+1,j})}$. Also, it is easily checked that $\big (p_{\frac{\widetilde \chi_{r+1}}{\widetilde \chi_r},r+1}\big)_j\frac{\mathbb{E}(N_{r,i})}{\mathbb{E}(N_{r+1,j})}= \big (\Pi_r\mathbb{E}(W^{\frac{\widetilde \chi_{r+1}}{\widetilde \chi_r},r})\big)_i$.  Thus, $(p_{1,r})_i=\big (p_{\frac{\widetilde \chi_{r+1}}{\widetilde \chi_r},r+1}\big)_j\frac{\mathbb{E}(N_{r,i})}{\mathbb{E}(N_{r+1,j})}$. 

\noindent
$(ii)$ For $2\le r\le s$, $P_r'(\theta)$ exists and equals $\frac{1}{\widetilde \chi_r}\sum_{j\in \mathcal I_r}(p_{\theta,r})_j\log(\mathbb{E}(N_{r,j}))$ (this is a special case of \cite[Proposition 4.1]{BFAJM}). Moreover, it is direct to see that $P_s(1)$ is attained by the Bernoulli product measure on $\mathcal I_s^{\mathbb{N}^+}$ associated with $p_{1,s}=\big(\frac{\mathbb{E}(N_{s,j})}{\mathbb{E}(\#\mathcal I_\omega)}=\frac{\mathbb{E}(N_{s,j})}{\sum_{j'\in\mathcal I_s}\mathbb{E}(N_{s,j'}))}\big)_{j\in\mathcal I_s}$. 

Next, we remark that due to the definition of $(r_0,\theta_{r_0})$, we have either $\theta_{r_0}\in \big [\frac{\widetilde \chi_{r_0}}{\widetilde \chi_{r_0-1}},1\big )$ and  $P'_{r_0}(\theta_{r_0})\ge 0$, or $\theta_{r_0}=1$ and in  this case either $r_0\le s-1$  and by observation $(i)$ we can change $(\theta_{r_0},r_0)$ to $\big(\frac{\widetilde \chi_{r_0+1}}{\widetilde\chi_{r_0}},r_0+1\big )$ which makes it possible to initially assume that $\theta_{r_0}\in \big [\frac{\widetilde \chi_{r_0}}{\widetilde \chi_{r_0-1}},1\big )$, or $r_0=s$ and $P'_s(1)\le 0$. Moreover, 
$$
H(W)= h(p_{\theta_{r_0},r_0})+\sum_{j\in \mathcal I_{r_0}}(p_{\theta_0,r_0})_j\log(\mathbb{E}(N_{r_0,j})).
$$
Thus, by the observation $(ii)$, if $\theta_{r_0}\in \big [\frac{\widetilde \chi_{r_0}}{\widetilde \chi_{r_0-1}},1\big )$, one has $H(W)\ge  h(p_{\theta_{r_0},r_0})>0$, and if $r_0=s$ and $\theta_s=1$, $W=W^{1,s}$ so that $\frac{H(W)}{\widetilde \chi_s}=\frac{h(p_{1,s})+\sum_{j\in\mathcal I_s}(p_{1,s})_j\log \mathbb{E}(N_{s,j})}{\widetilde \chi_s}=P_s(1)\ge \dim_HK_\omega>0$ (conditional on $\{K_\omega\neq\emptyset\}$). Consequently, $\nu_{\theta_{r_0},r_0}$ is non degenerate. 

Now let us determine $\dim (\mu_{\theta_{r_0},r_0})$. 

Suppose that $\theta_{r_0}\in \big[\frac{\widetilde \chi_{r_0}}{\widetilde \chi_{r_0-1}},1\big)$ and  $P'_{r_0}(\theta_{r_0})\ge 0$. To see that the value provided by Theorem~\ref{thm-2.4}(3) for $\dim (\mu_{\theta_{r_0},r_0})$ is indeed $P_{r_0}(\theta_{r_0})$, due to the Ledrappier-young  type formula for $\dim (\mu_{p_{\theta_{r_0},r_0}})$, as well as the expression of $P_{r_0}(\theta_{r_0})$ in terms of $S(\theta_{r_0},\eta_{\theta_{r_0},r_0})$ and the previous paragraph which yields $H(W)\ge  h(p_{\theta_{r_0},r_0})$, we only need to prove that $H(W)\le h(p)$, where $p=(p_i)_{i\in\mathcal{I}_{r_0-1}}$ is the $\Pi_{r_0-1}$-projection of $\mathbb{E}(W)$, that is $p_i= (\Pi_{r_0-1}\mathbb{E}(W))_i$ for $i\in \mathcal{I}_{r_0-1}$. 

Note that $p_{\theta_{r_0},r_0}=\Pi_{r_0-1,r_0} p$. Consequently, if $P'_{r_0}(\theta_{r_0})=0$, the desired property comes from the inequalities $h(p)\ge h(p_{\theta_{r_0},r_0})=H(W)$. If $P'_{r_0}(\theta_{r_0})> 0$, then $\theta_{r_0}=\frac{\widetilde \chi_{r_0}}{\widetilde \chi_{r_0-1}}$.  If $r_0=2$, the inequality $H(W)\le h(p)$ is obvious by \eqref{HW}. If $r_0\ge 3$, by observation $(i)$ above, setting $j=\Pi_{r_0-1,r_0}(i)$, one has $p_i=\big (p_{\frac{\widetilde \chi_{r_0}}{\widetilde \chi_{r_0-1}},r_0}\big)_{j}\frac{\mathbb{E}(N_{r_0-1,i})}{\mathbb{E}(N_{r_0,j})}=(p_{1,r_0-1})_i$. Also, by observation $(ii)$, $P'_{i_0-1}(1)=\frac{1}{\widetilde\chi_{r_0-1}}\sum_{i\in\mathcal I_{r_0-1}}(p_{1,r_0-1})_i\log(\mathbb{E}(N_{r_0-1,i}))\le 0$. Noting, moreover, that for $i'\in\Pi_{r_0-1}^{-1}(\{i\})$, we have $W_{i'}=\big (p_{\frac{\widetilde \chi_{r_0}}{\widetilde \chi_{r_0-1}},r_0}\big)_{j}\frac{\mathbf{1}_{\{c_{i'}>0\}}}{\mathbb{E}(N_{r_0,j})}=p_i \frac{\mathbf{1}_{\{c_{i'}>0\}}}{\mathbb{E}(N_{r_0-1,i})}$, we get  $H(W)=h(p)+\sum_{i\in\mathcal I_{r_0-1}}p_i\log(\mathbb{E}(N_{r_0-1,i}))$. Finally,  $H(W)\le h(p)$. 

If $r_0=s$, $\theta_s=1$ and $P'_s(1)\le 0$, then $W=W^{1,s}$ implies that $H(W)=h(p_{1,s})+ \widetilde\chi_s P'_s(1)$ and  $P'_s(1)\le 0$ yields $H(W)\le h(p_{1,s})=h(\Pi_s\mathbb{E}(W))$, so we directly see that the Ledrappier-Young type formula yields $\dim (\mu_{1,s})=\frac{H(W)}{\widetilde \chi_s}$; also, $P_s(1)=\frac{H(W)}{\widetilde \chi_s}$ by definition of $P_s(1)$.

\section{Proof of Theorem~\ref{dimPKd}}\label{secdimPK} 

We continue to work, for each $\boldsymbol{p}=(p^{(n)})_{n\in\mathbb{N}^+}\in P_{\mathcal I}^{ \mathbb{N}^+}$ with the sequence of weights 
$$
W_{\boldsymbol{p}}^{(n)}=(p^{(n)}_i\widetilde W_i)_{i\in\mathcal I}, \ n\ge 1, \text{where }\widetilde W_i=\frac{\mathbf{1}_{\{c_i=1\}}}{\mathbb{P}(c_i=1)}.
$$ 
Recall that for all $N\ge 1$, $
\widetilde d_N(\boldsymbol{p})=\frac{1}{N}\min_{g_1(N)\le k\le  g_s(N)} H_{N,k}
$ was defined in~\eqref{uN1}.

Denote by $H_{\max}$ and $H_{\min}$ respectively the maximum and the minimum of the function 
$
\widetilde h: p\in P_{\mathcal I}\mapsto H((p_i\widetilde W_i)_{i\in\mathcal I})=h(p)+\sum_{i\in\mathcal I} p_i\log(\mathbb{P}(c_i=1)).
$ 
One has $H_{\max}=\log(\mathbb{E}(N))$ and the maximum is uniquely reached, at the point $p_{\max}=\left(\frac{\mathbb{P}(c_i=1)}{\mathbb{E}(\#\mathcal I_\omega))}\right)_{i\in\mathcal I}$, and $H_{\min}= \min_{i\in\mathcal I}\log(\mathbb{P}(c_i=1))$. Let $\lambda=8 \frac{H_{\max}-2H_{\min}}{H^2_{\max}}$.

Recall that $\Lambda'_a$ is a positive constant such that $g_1(N)\ge \Lambda'_a N$ for all $\boldsymbol{p}\in P_{\mathcal I}^{ \mathbb{N}^+}$ and $N\ge 1$. 

Fix $\ell\in \mathscr{L}$. For $\epsilon\in \left (0, \min (\lambda^{-1},\Lambda'_a,(\#\mathcal{I})^{-1} )\right)$, set $\eta=\eta(\epsilon)=\epsilon^2$. As in the study of the upper bound for $\dim_H K_\omega$, set $P_{\mathcal I}(\eta)=\{(p_i)_{i\in\mathcal I}:\, p_i\ge \eta, \, \forall\, i\in\mathcal I\}$ and fix a finite $(\#\mathcal I)\eta$-covering  $\left \{B(q,(\#\mathcal I)\eta)\right \}_{q\in \mathcal P_\eta}$ of $P_{\mathcal I}$, where $\mathcal P_\eta\subset  P_{\mathcal I}(\eta)$.  Recall also that we defined 
\begin{align*}
\mathscr{P}^{\ell,\eta}_{\mathcal{I}}&=\left \{(p^{(n)})_{n\in\mathbb{N}^+}\in P_{\mathcal{I}}^{\mathbb{N}^+}:\, \forall\, m\ge 1,\ \exists \, q \in \mathcal P_\eta,\ p^{(L_{m-1}+1)}=\cdots=p^{(L_m)}=q\right \}.
\end{align*}

For $N\in\mathbb{N}^+$ such that $N\epsilon\ge 1$, let 
\begin{align}\label{QIl}
\mathscr{Q}^{\ell,\epsilon,\eta,\lfloor\Lambda_a N\rfloor}_{\mathcal{I}}=\left \{\boldsymbol{p}\in \mathscr{P}^{\ell,\eta}_{\mathcal{I}}:\, \forall\,  \lfloor N\epsilon\rfloor \le M\le \lfloor\Lambda_aN\rfloor  ,\, \sum_{n=1}^{M}H(W_{\boldsymbol{p}}^{(n)})\ge  -M\epsilon\right \}.
\end{align}
We are going to prove the following proposition, which is enough to get Theorem~\ref{dimPKd}. 
\begin{proposition}\label{propdimP} For $\epsilon>0$ and $N\ge 1$, set $\Delta(\epsilon,N)=\sup\left\{\widetilde d_N(\boldsymbol{p}): \, \boldsymbol{p}\in \mathscr{Q}^{\ell,\epsilon,\eta(\epsilon),\lfloor\Lambda_a N\rfloor}_{\mathcal{I}}\right \}$ and  $\Delta(\epsilon)=\limsup_{N\to +\infty}\Delta(\epsilon,N)$. With probability~1, conditional on $K_\omega\neq\emptyset$, one has 
$$
\dim_P K_\omega\le \Delta:= \lim_{\epsilon\to 0} \Delta(\epsilon).
$$
Moreover, for all $\epsilon>0$ there exists $\boldsymbol{q}_\epsilon\in P_{\mathcal I}^{\mathbb{N}^+}$, such that $\mu_{\boldsymbol{q}_\epsilon}$ is of type $\ell$ and a.s. fully supported on $K_\omega$, and for which conditional on $K_\omega\neq\emptyset$,  $\dim_P(\mu_{\boldsymbol{q}_\epsilon})\ge \Delta-\epsilon$. Also, in the deterministic case, one can find $\boldsymbol{q}\in P_{\mathcal I}^{\mathbb{N}^+}$ of type $\ell$ such that $\mu_{\boldsymbol{q}}$ is fully supported on $K$ and $\dim_P(\mu_{\boldsymbol{q}})=\Delta$. 
 \end{proposition}

Before proving the proposition, we  establish two lemmas. Their proofs can be skipped at first reading. 

\begin{lemma}\label{lemperturbation} Let $N\ge \epsilon^{-1}$. If $\boldsymbol{p}=(p^{(n)})_{n\in\mathbb{N}^+}\in \mathscr{Q}^{\ell,\epsilon,\eta,\lfloor\Lambda_a N\rfloor}_{\mathcal{I}}$, set $\boldsymbol{p}_\epsilon=(p_\epsilon^{(n)})_{n\in\mathbb{N}^+}$, where 
$$
p_\epsilon^{(n)}=
\begin{cases}
p_{\max}&\text{if }  1\le n\le   \lfloor N\epsilon\rfloor\\
(1-\lambda\epsilon)p^{(n)}+\lambda\epsilon p_{\max}&\text{if } \lfloor N\epsilon\rfloor+1\le n\le  \lfloor\Lambda_aN\rfloor  \text{ and } H(W_{\boldsymbol{p}}^{(n)})\le H_{\max}/2\\
p^{(n)}&\text{otherwise}.
\end{cases}
$$
Then, $\sum_{n=1}^{M}H(W_{\boldsymbol{p}_\epsilon}^{(n)})\ge  M\epsilon$ for all $1\le M\le \lfloor\Lambda_aN\rfloor $. And the same holds if one redefines $p_\epsilon^{(n)}=p_{\max}$ for those $n$ belonging to the same interval  $[L_{m-1}+1, L_m]$ as $ \lfloor N\epsilon\rfloor$. 
\end{lemma}
Note that the modification of $\boldsymbol{p}_\epsilon$ in the last assertion is considered so that $p^{(n)}$ is independent of $n$ in intervals of the form $[L_{k-1}+1,L_k]$. 

\begin{proof}[Proof of Lemma~\ref{lemperturbation}] We note that by concavity of the mapping $\widetilde h$ and the fact that $\lambda\epsilon\in (0,1)$, when $\lfloor N\epsilon\rfloor+1\le n\le  g_s(N)$ and $H(W_{\boldsymbol{p}}^{(n)})=\widetilde h(p^{(n)})\le H_{\max}/2$, one has 
\begin{align*}
H(W_{\boldsymbol{p}_\epsilon}^{(n)})=\widetilde h(p^{(n)}_\epsilon)&=\widetilde h((1-\lambda\epsilon)p^{(n)}+\lambda\epsilon p_{\max})\\
&\ge \widetilde h(p^{(n)})+ \lambda\epsilon (\widetilde h(p_{\max})-\widetilde h(p^{(n)}))\\
&\ge  \widetilde h(p^{(n)})+ \lambda\epsilon H_{\max}/2= H(W_{\boldsymbol{p}}^{(n)})+\lambda\epsilon H_{\max}/2.
\end{align*}

It follows that for all $1\le M\le \lfloor\Lambda_aN\rfloor$, one has $\sum_{n=1}^{M}H(W_{\boldsymbol{p}_\epsilon}^{(n)})\ge \sum_{n=1}^{M}H(W_{\boldsymbol{p}}^{(n)})$. So if $M\le \lfloor N\epsilon\rfloor$ or $\sum_{n=1}^{M}H(W_{\boldsymbol{p}}^{(n)})\ge M\epsilon$, there is nothing to prove. If $M\ge \lfloor N\epsilon\rfloor+1$ and $\sum_{n=1}^{M}H(W_{\boldsymbol{p}}^{(n)})< M\epsilon$, denote $S_M=\{1\le n\le M:\,   H(W_{\boldsymbol{p}}^{(n)})\le H_{\max}/2\}$.   One has 
$$
M\epsilon>  \sum_{n=1}^{M}H(W_{\boldsymbol{p}}^{(n)}) \ge (M-\#S_M)H_{\max}/2 +(\#S_M) H_{\min},
$$
hence $\#S_M\ge    M \frac{(H_{\max}-2\epsilon)}{H_{\max}-2H_{\min}}$ (note that    $H_{\min}<0$). Now 
\begin{align*}
\sum_{n=1}^{M}H(W_{\boldsymbol{p}_\epsilon}^{(n)})&\ge \sum_{n=1}^{\lfloor N\epsilon\rfloor}\mathbf{1}_{S_M^c}(n)H_{\max}+ \sum_{n=1}^{\lfloor N\epsilon\rfloor}\mathbf{1}_{S_M}(n)(H(W_{\boldsymbol{p}}^{(n)})+\lambda\epsilon H_{\max}/2)\\
&\quad +\sum_{n=\lfloor N\epsilon\rfloor+1}^M\mathbf{1}_{S_M^c}(n)(H(W_{\boldsymbol{p}}^{(n)}))+\sum_{n=\lfloor N\epsilon\rfloor+1}^M\mathbf{1}_{S_M}(n)(H(W_{\boldsymbol{p}}^{(n)})+\lambda\epsilon H_{\max}/2)\\
&= \sum_{n=1}^{\lfloor N\epsilon\rfloor}\mathbf{1}_{S_M^c}(n)(H_{\max}-H(W_{\boldsymbol{p}}^{(n)})) +\sum_{n=1}^M  H(W_{\boldsymbol{p}}^{(n)})+ (\#S_M) \lambda\epsilon H_{\max}/2\\
&\ge \sum_{n=1}^M  H(W_{\boldsymbol{p}}^{(n)})+ (\#S_M) H_{\max}/2\ge -M\epsilon+ M\lambda\epsilon \frac{H_{\max}(H_{\max}-2\epsilon)}{2(H_{\max}-2H_{\min})}\ge M\epsilon; 
\end{align*}                                                                                                                                                                                                                                                                                                                                                                                                                                                                                                                                                                                                                                                                                                                                                                                                                                                                                                                                                                                                                                                                                                                                                                                                                                                                                                                                                                                                                                                                                                                                                                                                                                                                                                                                                                                                                                                                                                                                                                                                                                                                                                                                                                                                                                                                                                                                                                                                                                                                                                                                                                                                                                                                                                                                                                                                                                                                                                                                                                                                                                                                                                                                                                                                                                                                                                                                                                                                                                                                                                                                                                                                                                                                                                                                                                                                                                                                                                                                                                                                                                                                                                                                                                                                                                                                                                                                                                                                                                                                                                                                                                                                                                                                                                                                                                                                                                                                                                                                                                                                                                                                                                                                                                                                                                                                                                                                                                                                                                                                                
indeed our choice of $\epsilon$  implies that $\epsilon\le H_{\max}/4$, so that 
$$
\lambda \frac{H_{\max}(H_{\max}-2\epsilon)}{2(H_{\max}-2H_{\min})}\ge \lambda \frac{H^2_{\max}}{4(H_{\max}-2H_{\min})}\ge 2.
$$ 
by definition of $\lambda$.
\end{proof}

The statement of the second lemma requires two last definitions. Recall the definition \eqref{epMmd} of the sets of the form $E_{\boldsymbol{p}}(M, m,\delta)$ (we add the subscript $\boldsymbol{p}$ to indicate the dependence in $\boldsymbol{p}$). Recall also that for $N\in\mathbb{N}^+$, we defined $m(N)$ the greatest integer such that $L_{m(N)}\le N-1$, and for any fixed $\delta\in(0,1)$ we can consider an integer $M_\delta $  such that  $\ell_{m+1}\le \delta L_m$ for all $m\ge M_\delta$.  Observe that $E_{\boldsymbol{p}}(M_\delta, m,\delta)$ only depends on $(p^{(n)})_{1\le n\le L_m}$, so that in the lemma below the sets $E_{\boldsymbol{p}}(M_\delta , m(M),\delta)$, $ \lfloor N\epsilon\rfloor \le M\le \lfloor\Lambda_aN\rfloor$, depend only on $(p^{(n)})_{1\le n\le  \lfloor\Lambda_aN\rfloor }$.

For $N\ge 1$ define 
$$
\mathscr{P}^{\ell,\eta, \lfloor\Lambda_aN\rfloor }_{\mathcal{I}}=\left \{(p^{(n)})_{1\le n\le\lfloor\Lambda_aN\rfloor }: \boldsymbol{p}\in \mathscr{P}^{\ell,\eta}_{\mathcal{I}}\right \}.
$$
and 
$$
\widetilde{\mathscr{P}}^{\ell,\epsilon,\eta, \lfloor\Lambda_aN\rfloor}_{\mathcal{I}}=\left \{\boldsymbol{p}\in \mathscr{P}^{\ell,\eta, \lfloor\Lambda_aN\rfloor }_{\mathcal{I}}: \, \exists\,   \lfloor N\epsilon\rfloor \le M\le \lfloor\Lambda_aN\rfloor  ,\, \sum_{n=1}^{M}H(W_{\boldsymbol{p}}^{(n)})<  -M\epsilon\right \}.
$$

\begin{lemma}\label{elim}Recall the constant $C'_1$ defined in \eqref{C'1}. Fix $\delta\in(0,\epsilon/C'_1)$. With probability~1, for $N$ large enough, 
\begin{equation}\label{intvide}
\Sigma_\omega\cap \Big (\bigcup_{\boldsymbol{p}\in \widetilde {\mathscr{P}}^{\ell,\epsilon,\eta, \lfloor\Lambda_aN\rfloor }_{\mathcal{I}}} E_{\boldsymbol{p}}(M_\delta, m(g_s(N)), \delta)\Big )=\emptyset.
\end{equation}
\end{lemma}
\begin{proof}
For $M\ge L_{M_\delta}/\delta$, set 
$$
\widetilde {\mathscr{P}}^{\ell,\epsilon,\eta, \lfloor\Lambda_aN\rfloor}_{\mathcal I,M}=\left\{\boldsymbol{p}\in \widetilde {\mathscr{P}}^{\ell,\epsilon,\eta, \lfloor\Lambda_aN\rfloor}: \sum_{n=1}^{M}H(W_{\boldsymbol{p}}^{(n)})<  -M\epsilon\right\}.
$$
Remark~\ref{alternativeupperbound} yields that for  $M\ge L_{M_\delta}/\delta$ and $\boldsymbol{p}\in \widetilde {\mathscr{P}}^{\ell,\epsilon,\eta, \lfloor\Lambda_aN\rfloor}_{\mathcal I,M}$, if $\widehat{\mathcal N}_{\boldsymbol{p},M}$ stands for the cardinality of those $U_1\in\mathcal{I}^{M}$ such that $[U_1]\cap E_{\boldsymbol{p}}(M_\delta, m(M),\delta)\neq\emptyset$, then $\mathbb{E}(\widehat{\mathcal N}_{\boldsymbol{p},M})\le e^{C'_1\delta M+\sum_{n=1}^{M}H(W_{\boldsymbol{p}}^{(n)})}\le e^{(C'_1\delta-\epsilon)M}$; note also that  $\widetilde{\mathscr{P}}_{\mathcal I}^{\ell,\epsilon,\eta, \lfloor\Lambda_aN\rfloor}\le \#(\mathcal P_\eta)^{\gamma_NN}$ with $\gamma_N \to 0$ as $N\to+\infty$  (this is obtained as \eqref{cardPlepsiloni}). It follows that if $\lfloor N\epsilon\rfloor \ge L_{M_\delta}/\delta$ and we denote by $\mathscr{C}_N$ the set of cylinders which are of some generation $M\in  [\lfloor N\epsilon\rfloor ,\lfloor\Lambda_aN\rfloor]$ and which meet $E_{\boldsymbol{p}}(M_\delta, m(M), \delta)$ for some $\boldsymbol{p}\in \widetilde {\mathscr{P}}_{\mathcal I,M}^{\ell,\epsilon,\eta, \lfloor\Lambda_aN\rfloor}$, we have 
$$
\mathbb{E} (\#\mathscr{C}_N)\le \#(\mathcal P_\eta)^{\gamma_NN}\sum_{M=\lfloor N\epsilon\rfloor }^{\lfloor\Lambda_aN\rfloor} e^{(C'_1\delta -\epsilon)M},
$$ 
from which it follows that $\mathbb{E}(\sum_{N:\, \lfloor N\epsilon\rfloor \ge L_{M_\delta}/\delta}\#\mathscr{C}_N)<+\infty$ due to the  assumption $C'_1\delta -\epsilon<0$ and the property of $(\gamma_N)_{N\ge 1}$. Subsequently, almost surely, for $N$ large enough one has $\#\mathscr{C}_N=0$, that is $\mathscr{C}_N=\emptyset$. Since $\bigcup_{\boldsymbol{p}\in \widetilde {\mathscr{P}}^{\ell,\epsilon,\eta, \lfloor\Lambda_aN\rfloor }_{\mathcal{I}}} E_{\boldsymbol{p}}(M_\delta, m(g_s(N)), \delta)$ is covered by the elements of $\mathscr{C}_N$, we get ~\eqref{intvide}.
\end{proof}
Before proving Proposition~\ref{propdimP}, we need to extend Definition~\ref{HN}.
\begin{definition}\label{HN'} Recall the notations of Definition~\ref{HN}, all associated with a fixed $\boldsymbol{p}\in P_{\mathcal I}^{\mathbb{N}^+}$. If $\boldsymbol{q}\in P_{\mathcal I}^{\mathbb{N}^+}$, for $N\ge 1$, define
\begin{align*}
H_{N,k}^{\boldsymbol{p}}(\boldsymbol{q})=\sum_{n=1}^k H({W_{\boldsymbol{q}}^{(n)}})+\sum_{n=k+1}^{g_s(N)} h(\Pi_{r_n}q^{(n)})\quad (0\le k\le g_s(N)),
\end{align*}
where $r_n$ is the index $r$ such that $g_{r-1}(N)+1\le n\le g_{r}(N)$.

Also, set 
\begin{align*}
d^{\boldsymbol{p}}_N(\boldsymbol{q})&=\frac{1}{N}\min \Big (\min_{g_1(N)\le k\le  g_s(N)-1} H^{\boldsymbol{p}}_{N,k}(\boldsymbol{q}),\min_{N'\ge g_s(N)} \sum_{n=1}^{N'} H(W_{\boldsymbol{q}}^{(n)})\Big )\\
\text{and}\quad
\widetilde d^{\boldsymbol{p}}_N(\boldsymbol{q})&=\frac{1}{N}\min_{g_1(N)\le k\le  g_s(N)} H^{\boldsymbol{p}}_{N,k}(\boldsymbol{q}).
\end{align*}
In particular, $d_N(\boldsymbol{p})$ and $\widetilde d_N(\boldsymbol{p})$ equal $d^{\boldsymbol{p}}_N(\boldsymbol{p})$ and $\widetilde d^{\boldsymbol{p}}_N(\boldsymbol{p})$ respectively.
\end{definition}

\begin{proof}[Proof of Proposition~\ref{propdimP}] Let $\delta=\delta(\eta)$ as in~\eqref{deltaeta} and note that if $\epsilon$ is small enough then $\delta(\eta)<\epsilon/C'_1(<1)$. Fix $N\ge M_\delta/(\delta\Lambda'_a)$. Recall the inequality~\eqref{constant C'} in Remark~\ref{alternativeupperbound}, namely $\mathbb{E}(\widetilde{N}_{\boldsymbol{p},N})\le e^{(C''_1\delta  +\widetilde d_N(\boldsymbol{p}))N}$ valid for any $\boldsymbol{p}\in \mathscr P_{\mathcal I}^{\ell,\eta}$ and $\widetilde{N}_{\boldsymbol{p},N}$, the number  of sets $B_{N}(\boldsymbol{i})$ in $\mathcal F^{D}_N(g)$, with $\boldsymbol{i}\in E_{\boldsymbol{p}}(M_\delta,m({g_s(N)}),\delta)$. Denote this collection of sets $B_{N}(\boldsymbol{i})$ by $\mathcal B_N(\boldsymbol{p})$. It only depends on $(p^{(n)})_{1\le N\le \lfloor\Lambda_aN\rfloor}$. 

We deduce from Lemma~\ref{elim} that with probability~1, conditional on $K_\omega\neq\emptyset$, for $N$ large enough, one has (recall \eqref{QIl})
$$
K_\omega\subset \bigcup_{\boldsymbol{p}\in {\mathscr{Q}}_{\mathcal I}^{\ell,\epsilon,\eta, \lfloor\Lambda_aN\rfloor}}\bigcup_{B\in \mathcal B_N(\boldsymbol{p})}Q_B.
$$
Each $Q_B$ in the above union is a parallelepiped of sides lengths  smaller than or equal to $e^{\lambda_a\Lambda_a\delta N} e^{-N}$, so there exists a constant $C(d)$ such that $Q_B$ is contained in a union of at most $C(d)e^{\lambda_a\Lambda_a\delta N d}$ cubes of sides lengths $e^{-N}$. Moreover, the expectation of the total number of parallelepipeds $Q_B$ occurring in the above union is bounded by 
\begin{align*}
\sum_{\boldsymbol{p}\in {\mathscr{Q}}_{\mathcal I}^{\ell,\epsilon,\eta, \lfloor\Lambda_aN\rfloor}}\mathbb{E}(\widetilde{N}_{\boldsymbol{p},N})\le \big (\# \mathscr{Q}_{\mathcal I}^{\ell,\epsilon,\eta, \lfloor\Lambda_a N\rfloor}\big )  e^{(C''_1\delta +\Delta(\epsilon,N))N}\le \#(\mathcal P_\eta)^{\gamma_NN}e^{(C''_1\delta + \Delta(\epsilon,N))N}.
\end{align*}
This implies that with probability~1, conditional on $K_\omega\neq\emptyset$, for $N$ large enough, $K_\omega$ is covered by at most $C(d)e^{\lambda_a\Lambda_a\delta N d}(\#\mathcal P_\eta)^{\gamma_NN}e^{(C''_1\delta) N+ \Delta(\epsilon,N)N+\epsilon N}$ cubes of sides lengths $e^{-N}$. Consequently,
\begin{align*}
\overline \dim_BK_\omega&\le \limsup_{N\to +\infty} \gamma_N \log(\# \mathcal P_\eta)+(C''_1+\lambda_a\Lambda_a d)\delta(\eta(\epsilon))+\Delta(\epsilon,N)+\epsilon\\
&=(C''_1+\lambda_a\Lambda_a d)\delta(\eta(\epsilon))+\limsup_{N\to +\infty} \Delta(\epsilon,N)+\epsilon.
\end{align*}
Since $C''_1$ (see \eqref{C''1}) does not depend on $\epsilon$ and $\delta(\eta(\epsilon))$ tends to 0 as $\epsilon\to 0$, we deduce that $\dim_PK_\omega\le  \overline \dim_BK_\omega\le \Delta$ as desired.  

It remains to exhibit, for each $\gamma>0$, an inhomogeneous Mandelbrot measure of type~$\ell$ whose packing dimension is larger than $\Delta-\gamma$, and show that in the deterministic case  one can take $\gamma=0$.  

Suppose now that $\epsilon$ is also strictly smaller than $2 H_{\max} $ and small enough so that the conclusions of Lemma~\ref{elim} hold with $\delta(\eta(\epsilon))$. Consider also an increasing sequence of integers $(N_j)_{j\in\mathbb{N}^+}$,  as well as a sequence $(\boldsymbol{p}_j)_{j\in\mathbb{N}^+}\in \mathscr{P}^{\ell,\eta}_{\mathcal I}$ such that for each $j\ge 1$ one has $\boldsymbol{p}_j\in {\mathscr{Q}}_{\mathcal I}^{\ell,\epsilon,\eta(\epsilon), \lfloor\Lambda_aN_j\rfloor}$ and $\widetilde d_{N_j}(\boldsymbol{p}_j)\ge \Delta(\epsilon)(1-\epsilon)$. 

To each $\boldsymbol{p}_{j}$ are associated the objects $(\gamma_k(N_j))_{1\le k\le d}$, $D(N_j)$, $s=s(N_j)$, $g(N_j)=(g_1(N_j),\ldots g_s(N_j))$ and the partition $\mathcal{F}^{D}_{N_j}(g)$ at scale $N_j$ as in Section~\ref{pfthmdim}. In particular $ \gamma_k(N_j)\chi_k((\widehat{\boldsymbol{p}_{j}})_{\gamma_k(N_j)})\sim N_j$ as $j\to+\infty$ for $1\le k\le d$. 

We denote by $m_j$ the unique integer $m$ such that $L_{m-1}+1\le g_{s(N_j)}(N_j)\le L_m$ (remember that $g_{s(N_j)}(N_j)$ is associated with $\boldsymbol{p}_j$), and without loss of generality we can assume that $L_{m_{j-1}}\le \log (\lfloor\epsilon N_{j}\rfloor)\le \frac{\epsilon}{2H_{\max}-\epsilon} \lfloor\epsilon N_{j}\rfloor$ for all $j\ge 2$. This implies in particular that 
\begin{equation}\label{control2}
\text{for all }M\ge \lfloor\epsilon N_{j}\rfloor, \  M\epsilon- L_{m_{j-1}}H_{\max}\ge (M-L_{m_{j-1}})\epsilon/2.
\end{equation}  

For each $j\ge 1$, we denote by $\boldsymbol{p}_{\epsilon,j}$ the sequence $(\boldsymbol{p}_j)_{\epsilon}$ constructed from  $\boldsymbol{p_j}$ in Lemma~\ref{lemperturbation}.

We then define a sequence $\boldsymbol{q}_\epsilon$ as follows: 
$$
\boldsymbol{q}_\epsilon^{(n)}=
\begin{cases}
\boldsymbol{p}^{(n)}_{\epsilon,1} \text{ if $1\le n\le L_{m_1}$}\\
\boldsymbol{p}^{(n)}_{\epsilon,j} \text{ if $j\ge 2$ and $L_{m_{j-1}}+1\le n\le L_{m_j}$}.
\end{cases}
$$
We denote by $\mu_{\boldsymbol{q}_\epsilon}$ the Mandelbrot measure  constructed from  $\boldsymbol{q}_\epsilon$ and random vectors of generation $n-1$ identically distributed with $W^{(n)}_{\boldsymbol{q}_\epsilon}$ for all $n\ge 1$. It is of type $\ell$. Let us check that this measure is not degenerate. By construction, for all $j\ge 2$ and $L_{m_{j-1}}+1\le M\le L_{m_j}$, one has 
\begin{align*}
\sum_{n=1}^MH(W_{{\boldsymbol{q}_\epsilon}}^{(n)})&= \sum_{n=1}^{L_{m_{j-1}}}H(W_{{\boldsymbol{q}_\epsilon}}^{(n)})+ \sum_{n=L_{m_{j-1}}+1}^{M}H(W_{\boldsymbol{p}_{\epsilon,j}}^{(n)})\\
&\ge \sum_{n=1}^{L_{m_{j-1}}}H(W_{{\boldsymbol{q}_\epsilon}}^{(n)})+
\begin{cases}
(M-L_{m_{j-1}})H_{\max}&\text{if } M\le  \lfloor\epsilon N_{j}\rfloor\\
M\epsilon- L_{m_{j-1}}H_{\max} &\text{otherwise}
\end{cases}\\
&\ge  \Big (\sum_{n=1}^{L_{m_{j-1}}}H(W_{{\boldsymbol{q}_\epsilon}}^{(n)})\Big )+  (M-L_{m_{j-1}})\epsilon/2,
\end{align*}
where we used that $H_{\max} \ge \epsilon/2$ and \eqref{control2}. Since for $1\le M\le L_{m_1}$ one has  $\sum_{n=1}^{M}H(W_{{\boldsymbol{q}_\epsilon}}^{(n)})=\sum_{n=1}^{M}H(W_{\boldsymbol{p}_{\epsilon,1}}^{(n)})\ge M\epsilon\ge M\epsilon/2$, we deduce by recursion on the integer $j$ such that $L_{m_{j-1}}+1\le n\le L_{m_j}$ that for all $M\ge 1$ one has  that $\sum_{n=1}^{M}H(W_{{\boldsymbol{q}_\epsilon}}^{(n)})\ge M\epsilon/2$, hence by Theorem~\ref{sufcond} the measure $\mu_{\boldsymbol{q}_\epsilon}$ is positive and fully supported on $K_\omega$, conditional on $K_\omega\neq\emptyset$ (by construction the components of any vector $\boldsymbol{q}_\epsilon^{(n)}$ are positive). Similar arguments as above show that $\sum_{n=m_j+1}^MH(W_{{\boldsymbol{q}_\epsilon}}^{(n)})\ge 0$ for all $M\ge L_{m_j}+1$. In particular, $L_{m_j-1}+1\le g_s(N_j)\le \widetilde {g}_s(N_j)\le L_{m_j}$, hence $d^{\boldsymbol{p}_{j}}_{N_j}(\boldsymbol{q}_\epsilon)=\widetilde d^{\boldsymbol{p}_{j}}_{N_j}(\boldsymbol{q}_\epsilon)+o(1)$ (recall  Definition~\ref{HN'}). Note also that  the components of $\boldsymbol{q_\epsilon}$ are uniformly bounded away from 0, so that Proposition~\ref{mursnonchargés} applies to $\mu_{\boldsymbol{q}_\epsilon}$.

What is left to prove is that as $\epsilon\to 0$, conditional on $\mu_{\boldsymbol{q}_\epsilon}\neq 0$, we have $\dim_P(\mu_{\boldsymbol{q}_\epsilon})\to \Delta$ as $\epsilon\to 0^+$. Indeed, one has $\dim_P(\mu_{\boldsymbol{q}_\epsilon})\ge \Delta(\epsilon)(1-\epsilon)+O(\epsilon)$; to see this, the idea  is to use a computation similar to that used to prove   Theorem~\ref{thm-2.4}(2) via Propositions~\ref{tau} and~\ref{SN'} and Theorem~\ref{dimHmu}(2), by considering the partitions $\mathcal{F}^{D}_{N_j}(g)$, $j\ge 1$ (remember that $\mathcal{F}^{D}_{N_j}(g)$ is associated with $\boldsymbol{p}_j$), and estimating from above $\mathbb{E}\Big (\sum_{B\in \mathcal{F}^{D}_{N_j}(g)}\mu_{\boldsymbol{q}_\epsilon}(B)^q\Big )$ for $q$ close to~$1^+$.  Due to the assumption $L_{m_{j-1}}\le \log (\lfloor\epsilon N_{j}\rfloor)$ on the growth of $N_j$, this yields that with probability~1, conditional on $\mu_{\boldsymbol{q}_\epsilon}\neq 0$,  for $\mu_{\boldsymbol{q}_\epsilon}$-almost every $z$, $\liminf_{j\to +\infty}\frac{\log(\mu_{\boldsymbol{q}_\epsilon}(Q_{N_j}(z)))}{-N_j}\ge \liminf_{j\to +\infty}d^{\boldsymbol{p}_{j}}_{N_j}(\boldsymbol{q}_\epsilon)= \liminf_{j\to +\infty}\widetilde d^{\boldsymbol{p}_{j}}_{N_j}(\boldsymbol{q}_\epsilon)$. Moreover, the relation between $\boldsymbol{p}_{\epsilon,j}$ and $\boldsymbol{p}_{j}$, as well as the constraint $L_{m_{j-1}}\le \log (\lfloor\epsilon N_{j}\rfloor)$ imply that $|\widetilde d^{\boldsymbol{p}_{j}}_{N_j}(\boldsymbol{q}_\epsilon)-\widetilde d_{N_j}(\boldsymbol{p}_j)|=O(\epsilon)$ and for all $1\le k\le d$ (recall \eqref{lambdakN} and that the $\gamma_k(N_j)$ are associated with ${\boldsymbol{p}_{j}}$)
$$
\gamma_k(N_j)\chi_k((\widehat{\boldsymbol{q}_\epsilon})_{\gamma_k(N_j)})=\gamma_k(N_j)\chi_k((\widehat{\boldsymbol{p}_{j}})_{\gamma_k(N_j)})+O(\epsilon)N_j=N_j(1+O(\epsilon)).
$$ 
This implies that for $\mu_{\boldsymbol{q}_\epsilon}$-almost every $z$, $Q_{N_j}(z)$ is a parallelepiped whose sides lengths are $e^{-N_j(1+O(\epsilon))}$, and $\liminf_{j\to +\infty} \frac{\log(\mu_{\boldsymbol{q}_\epsilon}(Q_{N_j}(z)))}{-N_j}\ge \liminf_{j\to +\infty} \widetilde d_{N_j}(\boldsymbol{p}_j)(1+ O(\epsilon))\ge \Delta(\epsilon)(1-\epsilon)+O(\epsilon)$. Consequently, Lemma~\ref{calcdim}(3) yields $\lim_{\epsilon\to 0^+} \dim_P(\mu_{\boldsymbol{q}_\epsilon})\ge\Delta$ (for IMMs we know that the packing dimension exists). 

For the deterministic case, we do not have to take care of the non degeneracy of the measure we construct, since we simply consider an inhomogeneous Bernoulli measure. This makes it possible to consider a decreasing sequence $(\epsilon_j)_{j\ge 1}$ converging to 0 and require at the beginning of the above argument that the  sequence $(\boldsymbol{p}_j)_{j\in\mathbb{N}^+}\in \mathscr{P}^{\ell,\eta}_{\mathcal I}$ is such that for each $j\ge 1$ one has $\boldsymbol{p}_j\in {\mathscr{Q}}_{\mathcal I}^{\ell,\epsilon_j,\eta(\epsilon_j), \lfloor\Lambda_aN_j\rfloor}$ and $\widetilde d_{N_j}(\boldsymbol{p}_j)\ge \Delta(\epsilon_j)(1-\epsilon_j)$. Then we  consider $\boldsymbol{p}_{\epsilon_j,j}$ instead of $\boldsymbol{p}_{\epsilon,j}$, construct $\boldsymbol{q}_\epsilon$ as above but from the collection $\{\boldsymbol{p}_{\epsilon_j,j}\}_{j\ge 1}$ instead of $\{\boldsymbol{p}_{\epsilon,j}\}_{j\ge 1}$, and it results that $\dim_P(\mu_{\boldsymbol{q}_\epsilon})=\Delta$. 
\end{proof}

\section{Appendix}\label{appendix}

\noindent
\textbf{\textit{An inequality on the moments of a sum of independent and centered random variables.}}

\begin{lemma}[\cite{vBE}] \label{lemvBE} For all $h\in(1,2]$, for all integers $m\ge 1$, if $Z_1,\ldots, Z_m$ are independent and centered real random variables.  Then   $\mathbb{E}\big(\big|\sum_{i=1}^mZ_i\big|^h\big )\le 2^{h}\sum_{i=1}^m\mathbb{E}(|Z_i|^h)$. 
\end{lemma}

\noindent
\textbf{\textit{Dimensions of a measure.} }Recall that if $\mu$ is a positive and finite Borel measure on $\mathbb{R}^d$, then its lower Hausdorff dimension and upper Hausdorff dimensions are respectively defined as 
\begin{align*}
\underline{\dim}_H(\mu)&=\inf\{\dim_H E:\, E \text{ is Borel and } \mu(E)>0\}\\ 
\text{and}\quad\overline{\dim}_H(\mu)&=\inf\{\dim_H E:\, E \text{ is Borel  and }\mu(\mathbb{R}^d\setminus E)=0\}, 
\end{align*} 
In case of equality of these dimensions, their common value is simply denoted by $\dim_H(\mu)$ and called the Hausdorff dimension of $\mu$. The lower packing  dimension $\underline{\dim}_P(\mu)$ and upper packing  dimensions $\overline{\dim}_H(\mu)$ of $\mu$ are define similarly by replacing $\dim_H$ by $\dim_P$, as well as the packing dimension of $\mu$, defined as their common value whenever they coincide, and denoted $\dim_P(\mu)$. 

Defining the lower local and upper local dimensions of $\mu$ at any point $x\in\mathrm{supp}(\mu)$ respectively as as 
$$
\underline \dim(\mu,x)=\liminf_{r\to 0^+} \frac{\log\big (\mu(B(x,r))\big )}{\log(r)}\text{ and } \overline \dim(\mu,x)=\limsup_{r\to 0^+} \frac{\log\big (\mu(B(x,r))\big )}{\log(r)},
$$
one has the characterizations~(see \cite{FLR} for instance): 
\begin{align*}
&\underline{\dim}_H(\mu)=\mathrm{ess\, inf}_\mu \, \underline \dim(\mu,\cdot), \ \underline{\dim}_H(\mu)=\mathrm{ess\, sup}_\mu \, \underline \dim(\mu,\cdot),\\ 
&\underline{\dim}_P(\mu)=\mathrm{ess\, inf}_\mu \, \overline \dim(\mu,\cdot),\ \overline{\dim}_P(\mu)=\mathrm{ess\,sup}_\mu \, \overline \dim(\mu,\cdot),
\end{align*}
and one says that $\mu$ is exact dimensional if $\underline{\dim}_H(\mu)=\overline{\dim}_P(\mu)$, and denote the common value by $\dim(\mu)$.

The following lemma and its proofs are elementary. They are in spirit of \cite[Proposition 2.3]{ref32} (which only deals with Hausdorff dimension), though different. They exploit the characterization of lower and upper Hausdorff or packing dimensions recalled in Section~\ref{thmdim} as well as the characterization of packing dimension as modified box-counting dimension (see~\cite{ref22}).

\begin{lemma}\label{calcdim}Let $\mu$ be a positive and finite Borel measure supported on $[0,1]^d$. Let $(\mathcal G_N)_{N\ge 1}$ a sequence of finite families of closed parallelepipeds included in $[0,1]^d$ and such that for all $N\ge 1$ two elements of $\mathcal G_N$ are equal or have disjoint interior. 

Suppose that for each $N\ge 1$ and each $Q\in\mathcal G_N$ one has $\mu(\partial Q)=0$ and  the elements of $\mathcal G_N$ form a covering of $\supp(\mu)$. In particular, $\mu$-almost every $z\in \supp(\mu)$ is contained in a unique element $Q_N(z)$ of $\mathcal{F}_N$ for all $N\ge 1$. 
 
Let $\epsilon_1>0$, $\epsilon_2\in (0,1)$, $\delta_2\ge \delta_1\ge 0$ and $\Delta_2\ge \Delta_1\ge 0$. Let $(N_j)_{j\ge 1}$ be an increasing sequence of integers. 
\begin{enumerate}
\item Suppose that  for $\mu$-almost every $z$ one has $\liminf_{N\to+\infty} \frac{\log(\mu(Q_N(z)))}{-N}\ge \delta_1$ and for~$N$ large enough the sides lengths of $Q_N(z)$ are larger than $e^{-N(1+\epsilon_1)}$. Then,  $\underline\dim_H(\mu)\ge \frac{\delta_1}{1+\epsilon_1}$.

\item Suppose that for $\mu$-almost every $z$ one has $\liminf_{j\to+\infty} \frac{\log(\mu(Q_{N_j}(z)))}{-N_j}\le \delta_2$ and  the  sides lengths of $Q_{N_j}(z)$ are smaller than $e^{-N_j(1-\epsilon_2)}$. Then, $\overline\dim_H(\mu)\le \frac{\delta_2}{1-\epsilon_2}$.

\item  Suppose that for $\mu$-almost every $z$ one has $\limsup_{j\to+\infty} \frac{\log(\mu(Q_{N_j}(z)))}{-N_j}\ge \Delta_1$  and the  sides lengths of $Q_{N_j}(z)$ are larger than $e^{-N_j(1+\epsilon_1)}$. Then, $\underline\dim_P(\mu)\ge \frac{\Delta_1}{1+\epsilon_1}$.

\item Suppose that for $\mu$-almost every $z$ one has $\limsup_{N\to+\infty} \frac{\log(\mu(Q_N(z)))}{-N}\le \Delta_2$  and for~$N$ large enough the sides lengths of $Q_N(z)$ are smaller than $e^{-N(1-\epsilon_2)}$.  Then,  $\overline\dim_P(\mu)\le \frac{\Delta_2}{1-\epsilon_2}$. 
\end{enumerate}
\end{lemma}

\newpage
\section*{Appendix: glossary of notation}
\begin{center} 
\begin{tabular}{|l l  |}
\hline 
\footnotesize $\mathbb{N}^+$ &\footnotesize Set of positive integers\\
\footnotesize $\mathcal I$ & \footnotesize \text{Finite set of cardinality $\ge 2$ }  \\
\footnotesize $\mathcal I^*$ & \footnotesize Set of finite words over the alphabet $\mathcal I$\\
\footnotesize $(\Sigma,T)=(\mathcal I^{\mathbb{N}^+},T)$& \footnotesize One-sided full shift over the alphabet $\mathcal I$\\
\footnotesize $T$&\footnotesize shift operation on $\Sigma$\\
\footnotesize $(f_i)_{i\in\mathcal I}$ & \footnotesize Contracting self-affine IFS \\
\footnotesize $K$&\footnotesize Attrator of $(f_i)_{i\in\mathcal I}$\\
\footnotesize $\pi$&\footnotesize Coding map from $\Sigma $ to $K$\\
\footnotesize $(a_{i,k})_{1\le k\le d}$ & \footnotesize diagonal coefficients of the linear part of $f_i$\\
\footnotesize $\Lambda_a$, $\Lambda'_a$, $\lambda_a$ & \footnotesize Constants depending on $\#\mathcal I$ and the $(a_{i,k})_{\substack{1\le k\le d\\i\in\mathcal I}}$ (see \eqref{LambdaLambda'} and \eqref{lyap}) \\ 
\footnotesize $(c_i)_{i\in\mathcal I}$& \footnotesize Random vector taking values in $\{0,1\}^{\mathcal I}$\\
\footnotesize $\mathcal I_\omega$& \footnotesize $\{i\in\mathcal I:\, c_i(\omega)=1\}$\\
\footnotesize $\Sigma_\omega$& \footnotesize Boundary of the Galton-Watson tree constructed in $\mathcal I^*$ via fractal\\
\footnotesize  &\footnotesize percolation by using independent copies of $(c_i)_{i\in\mathcal I}$ indexed by $\mathcal{I}^*$\\
\footnotesize $K_\omega$& \footnotesize Image of $\Sigma_\omega$ by $\pi$\\
\footnotesize $\mathbb R^A$&\footnotesize  Linear subspace  of  the Euclidean space $\mathbb{R}^d$ generated by $(e_k)_{k\in A}$,\\
\footnotesize & \footnotesize where $\emptyset\neq A\subset \{1,\ldots,d\}$ and $(e_k)_{1\le k\le d}$ is the canonical basis of $\mathbb{R}^d$\\ 
\footnotesize $\pi^A$& \footnotesize Orthogonal projection from $\mathbb R^d$ to $\mathbb R^A$\\ 
\footnotesize $\mathcal{P}_{\mathcal I}$& \footnotesize Set of probability vectors indexed by $\mathcal I$\\
\footnotesize $\chi_k(p)$& \footnotesize $k$-th Lyapunov exponent associated with $p\in\mathcal P_{\mathcal I}$ and $(f_i)_{i\in\mathcal I}$ (see \eqref{chip})\\
\footnotesize $h(p)$& \footnotesize Entropy $-\sum_{j\in\mathcal J} p_j\log(p_j)$ of the probability vector $p=(p_j)_{j\in\mathcal J}$\\
\footnotesize $H(W)$&\footnotesize  ``Entropy'' of the non negative random vector $W=(W_i)_{i\in\mathcal I}$ (see \eqref{HW})\\ 
\footnotesize $(D_r)_{r=1}^s$& \footnotesize Decreasing family of sets of principal directions in $\mathbb R^d$ related to some\\
\footnotesize &\footnotesize  Lyapunov exponents defined as above (see Sections~\ref{preliminaries} and~\ref{thmdim}) \\
\footnotesize $(\Pi^D_r:\mathcal I\to \mathcal I^D_r)_{r=1}^s$& \footnotesize Family of mappings from $\mathcal I$ to some of its subsets $\mathcal I^D_r$ \\
 \footnotesize &\footnotesize associated to some $(D_r)_{r=1}^s$ (see Section~\ref{preliminaries})\\
\footnotesize $\Pi^D_rp$& \footnotesize Probability vector indexed by $\mathcal I^D_r$ obtained by  projecting $p\in\mathcal{P}_{\mathcal I}$ via $\Pi^D_r$\\
\footnotesize &\footnotesize (see Section~\ref{preliminaries})\\
\footnotesize $( \mathcal I_r^{\mathbb{N}^+},T_r)$&\footnotesize One-sided full shift over the alphabet $\mathcal I_r=\mathcal I^D_r$\\
\footnotesize $(g_r)_{r=1}^s$&\footnotesize  Increasing sequence of integers associated to some $(D_r)_{r=1}^s$ \\
\footnotesize &\footnotesize (see Section~\ref{thmdim})\\
\footnotesize $\widetilde {g_s}$ (also denoted $\tilde g_s$)&\footnotesize  Integer defined from $g_s$ via \eqref{Ntilde}\\
\footnotesize $H_{n,k}$, $d_N$ and $\widetilde d_N$& \footnotesize See Definition~\ref{HN}\\
\footnotesize $\mathcal F^D(g)$ and $\mathcal F^D_N(g)$&\footnotesize See \eqref{partition} and Section~\ref{pfthmdim}\\
\hline
\end{tabular}
\end{center}

\noindent
\textbf{Acknowledgements.} The authors thank an anonymous referee for their extremely careful reading, which made it possible to substantially improve the manuscript. 


\begin{thebibliography}{50}

\bibitem{ref1}
K.~Bara\'nski,
\newblock Hausdorff dimension of the limit sets of some planar geometric
  constructions.
\newblock {\em Adv. Math.}, 210 (2007), 215\--245.

\bibitem{ref31}
B.~B\'ar\'any, M.~Hochman, and A.~Rapaport,
\newblock Hausdorff dimension of planar self-affine sets and measures,
\newblock {\em Invent. Math.}, 216 (2019), 601--659.

\bibitem{BK} B.~B\'ar\'any and A. K\"aenm\"aki, Ledrappier-Young formula and exact dimensionality of self-affine measures, {\em Adv. Math.}, 318 (2017), 88--129.

\bibitem{BRS}
B.~B\'ar\'any, M.~Rams, and K.~Simon,
\newblock On the dimension of self-affine sets and measures with overlaps.
\newblock {\em Proc. Amer. Math. Soc.}, 144 (2016), 4427--4440.

\bibitem{ref39}
B.~B\'ar\'any, M.~Rams, and K.~Simon,
\newblock On the dimension of triangular self-affine sets.
\newblock {\em Ergod. Th. and Dynam. Sys.}, 39 (2019), 1751--1783.

\bibitem{vBE} B.~von Bahr, C.~G.~Esseen,  Inequalities for the $r$-th absolute moment of a sum of random variables, $1\le r\le 2$. {\it Ann. Math. Stat.}, {36} (1965), No. 1, 299--303.

\bibitem{ref4}
J.~Barral,
\newblock Generalized vector multiplicative cascades,
\newblock {\em Adv. in Appl. Prob.}, 33 (2002), 874--895.

\bibitem{BFNL}J.~Barral and D.-J.~Feng,  Non-uniqueness of ergodic measures with full Hausdorff dimension on Gatzouras-Lalley carpet, {\em Nonlinearity}, 24 (2011), 2563--2567. 

\bibitem{BFAJM}J.~Barral and D.-J.~Feng,
\newblock Weighted thermodynamic formalism on subshifts and applications.
\newblock {\em Asian J. Math.}, 16 (2012), 319--352.

\bibitem{ref2}
J.~Barral and D.-J.~Feng,
\newblock Projections of random {M}andelbrot measures.
\newblock {\em Adv. Math.}, 325 (2018), 640--718.

\bibitem{ref6}
J.~Barral and D.-J.~Feng,
\newblock Dimensions of random statistically self-affine Sierpi\'nski sponges in
  $\mathbb{R}^k$.
\newblock {\em J. Math. Pures Appl.}, 149 (2021), 254--303.

\bibitem{BM} J. Barral and M.~Mensi, Gibbs measures on self-affine Sierpi\'nski carpets and their singularity spectrum, {\em Ergod. Th. $\&$ Dynam. Sys.}, 27 (2007), 1419--1443. 



\bibitem{ref3}
T.~Bedford,
\newblock Crinkly curves, {M}arkov partitions and box dimension in self-similar
  sets. {P}h.{D}. {T}hesis.
\newblock 1984.

\bibitem{BenNasr} F.~Ben Nasr.
 Dimension de Hausdorff de certains fractals al\'eatoires.  {\it  S\'em. Th\'eor. Nombres Bordeaux}, {\bf 4} (1992), 129--140.


\bibitem{Br} L. Breiman, \textit{Probability}, Classics in Applied Mathematics, 7, SIAM, 1992.  

\bibitem{ref5}
T.~Das and D.~Simmons,
\newblock The Hausdorff and dynamical dimensions of self-affine sponges: a
  dimension gap result.
\newblock {\em Invent. Math.}, 210 (2017), 85--134.

\bibitem{DekGri} F.~M.~Dekking and G.~R.~Grimmett, Superbranching processes and projections
of random Cantor sets, {\em Probab.\ Theory Related Fields},  78 (1988), 335--355.

\bibitem{Dur}
R.~Durrett,  
\textit{Probability: Theory and Examples, 5th ed.} Cambridge Series in Statistical and Probabilistic Mathematics, 2019.

\bibitem{ref37}
R.~Durrett and T.-M.~Liggett,
\newblock Fixed points of the smoothing transformation.
\newblock {\em Z. Wahrsch. Verw. Gebiete}, 64 (1983), 275--301.

\bibitem{ref28}
K.~J.~Falconer,
\newblock The Hausdorff dimension of self-affine fractals.
\newblock {\em Math. Proc.  Camb. Phil. Soc}, 103 (1988), 339--350.

\bibitem{Fal} K.~J.~Falconer,
 Projections of random Cantor sets. {\em J.~Theoret.\ Probab.,} 2 (1989), 65--70.


\bibitem{ref22}
K.~J.~Falconer,
\newblock {\em Fractal Geometry: Mathematical Foundations and Applications}.
\newblock Wiley, 1990.

\bibitem{FJ} K.~J.~Falconer and X.~Jin,  Exact dimensionality and projections of random self-similar measures and sets,
{\em J. London Math. Soc.}, {90} (2014), 388--412.

\bibitem{FLR} A.-H.~Fan, K.-S.~Lau and H.~Rao, Relationships between different dimensions of a measure, {\em Monatsh. Math.}, 135 (2002), 191--201.

\bibitem{FengDuke} D.-J.~Feng,  Dimension of invariant measures for affine iterated function systems, {\em Duke Math. J.}, 172 (2023), no. 4, 701--774. 

\bibitem{FHJMPA}
D.-J.~Feng and W.~Huang,
\newblock {\em J. Math. Pure. Appl.}, 106 (2016), 411--452.

\bibitem{FengWang} D.-J.~Feng and Y.~Wang, A class of self-affine sets and self-affine measures, {\em J. Fourier Anal. Appl.}, 11 (2005), 107--124.


\bibitem{Fra-Kol}
J.~Fraser and I.~Kolossv\'ary,
\newblock The assouad dimension of self-affine measures on sponges.
\newblock {\em Ergod. Th. $\&$ Dynam. Sys.}, 43 (2023), 2841--2862.

\bibitem{ref10}
D.~Gatzouras and S.~Lalley,
\newblock Hausdorff and box dimensions of certain self-affine fractals.
\newblock {\em Indiana Univ. Math. J.}, 41 (1992), 533--568.

\bibitem{ref9}
D.~Gatzouras and S.~Lalley,
\newblock Statistically self-affine sets: Hausdorff and box dimensions.
\newblock {\em J. Theoret. Probab.}, 7 (1994), 437--468.

\bibitem{GMW}S.~Graf, R.~D.~Mauldin and S.~C.~Williams,  The exact Hausdorff dimension in random recursive constructions. {\it Mem. Amer. Math. Soc.}, 71 (1988), no. 381.

\bibitem{Heurteaux1998}
Y.~Heurteaux, 
Estimation de la dimension inf\'erieure et de la dimension sup\'erieure des mesures,
\textit{Ann. Inst. H. Poincar\'e}, 34 (1998), 309--338. 


\bibitem{Hochman1} M.~Hochman,
On self-similar sets with overlaps and inverse theorems for entropy, {\em Ann. Math.}, 180 (2014), 773--822. 

\bibitem{Hochman2} M.~Hochman,
On self-similar sets with overlaps and inverse theorems for entropy in $\mathbb{R}^d$, Memoirs of the A.M.S 1287, 2020.

\bibitem{ref40}
M.~Hochman and A.~Rapaport,
\newblock Hausdorff dimension of planar self-affine sets and measures with
  overlaps.
\newblock {\em J. Eur. Math. Soc.}, 24 (2022), 2361--2441.

 \bibitem{Hutchinson}J.~E.~Hutchinson,
 Fractals and self-similarity.
 {\it Indiana Univ. Math. J.}, {30} (1981),  713--747.

\bibitem{JPS} T.~Jordan, M.~Pollicott and K.~Simon,  Hausdorff dimension for randomly perturbed self affine attractors, {\em Comm. Math. Phys.}, {270} (2007),
 519--544.
 
 \bibitem{Kaenmaki} A.~K\"{a}enm\"{a}ki,  On natural invariant measures on generalised iterated function systems. {\em Ann. Acad. Sci. Fenn. Math.},  29 (2004),
    419--458.
    
\bibitem{K87} J.-P.~Kahane, Multiplications al\'eatoires et
  dimensions de Hausdorff, {\em  Ann. Inst. H. Poincar\'e Probab. Statist.}, {23} (1987), 289--296.

\bibitem{ref26}
J.-P.~Kahane and J.~Peyri{\`e}re,
\newblock Sur certaines martingales de {B}. {M}andelbrot.
\newblock {\em Adv. Math.}, 22 (1976), 131--145.

\bibitem{ref18}
R.~Kenyon and Y.~Peres,
\newblock Measures of full dimension on affine-invariant sets.
\newblock {\em Ergod. Th. and Dyn. Sys.}, 16 (1996), 307--323.

\bibitem{Kol}
I.~Kolossv\'ary.
\newblock ${L}^q$-spectrum of self-affine sponges,
\newblock {\em J. London Math. Soc.}, 108 (2023), 666--701.

\bibitem{LeWa77} F.~Ledrappier and P.~Walters,
A relativised variational principle for continuous
transformations. \textit{J. Lond. Math. Soc.}, {16} (1977), 568--576.


\bibitem{Lyons}
R.~Lyons, A simple path to Biggins' martingale convergence for branching random walk,
In \textit{Classical and Modern Branching Processes} (IMA Vol. Math. Appl. 84), eds. K. B. Athreya and P. Jagers, Springer, New York, 1997, pp. 217--222.


\bibitem{ref25}
B.~B.~Mandelbrot.
\newblock Intermittent turbulence in self-similar cascades, divergence of high
  moments and dimension on the carrier,
\newblock {\em J. Fluid Mech.}, 62 (1974), 331--358.

\bibitem{Mand-Orsay}
B.~B.~Mandelbrot.
\newblock Intermittent turbulence and fractal dimension: Kurtosis and the
  spectral exponent 5/3 + {B},
\newblock In {\em Turbul. Navier Stokes Equat., Proc. Conf. Univ. Paris-Sud,
  Orsay 1975}, volume 565 of {\em Lect. Notes Math.}, pages 121--145, 1976.


\bibitem{ref12}
C.~McMullen.
\newblock Hausdorff dimension of general {S}ierpi\'nski carpets,
\newblock {\em Nagoya Math. J.}, 96 (1984), 1--9.

\bibitem{Moran} P.~A.~P.~Moran, Additive functions of intervals and Hausdorff measure, {\em Math. Proc. Camb. Phil.
Soc.} 42 (1946), 15--23.

\bibitem{MSh} I.~Morris and P.~Shmerkin, On equality of Hausdorff and affinity dimensions, via self-affine measures on positive subsystems, 
    {\em Trans.  Amer. Math. Soc.},  371 (2019) 1547-1582. 

\bibitem{MSert} I.~Morris and C.~Sert, A variational principle relating self-affine measures to self-affine sets, arXiv:2303.03437.

\bibitem{Ngai} S.~M.~Ngai, A dimension result arising from the
  $L^q$-spectrum of a measure, {\em Proc.\ Amer.\ Math.\ Soc.}
  125 (1997), 2943--2951.

\bibitem{Olsen} L.~Olsen, Self-affine multifractal Sierpi\'nski sponges, {\em Pacific J. Math.}, 183 (1998), 1--57. 

\bibitem{SimonOrgov} V.~Orgov\'anyi and K.~Simon, Projections of random Menger sponges, {\em Asian J. Math.}, 27 (2023), 893--936. 

\bibitem{ref32}
J.~Peyri{\`e}re,
\newblock Mesures singuli{\`e}res associ{\'e}es {\`a} des d{\'e}coupages
  al{\'e}atoires d'un hypercube.
\newblock {\em Colloquium Math.}, 51 (1987), 267--276.

\bibitem{RamSi3} M.~Rams, K.~Simon.  The geometry of fractal percolation. In: Geometry and analysis of fractals, 303--323, Springer Proc. Math. Stat., 88, Springer, Heidelberg, 2014.

\bibitem{Rapaport1} A.~Rapaport, On self-affine measures associated to strongly irreducible and proximal systems, {\em Adv. Math.},  449 (2024), 109734.

\bibitem{Rapaport2} 	A.~Rapaport, Dimension of diagonal self-affine sets and measures via non-conformal partitions,  arXiv:2309.03985.

\bibitem{Solomyak} B.~Solomyak, Measure and dimension for some fractal families, {\em Math. Proc. Camb. Phil. Soc.},  124  (1998), 531--546.


\bibitem{WW}
E.~C.~Waymire and S.~C.~Williams, A cascade decomposition theory with applications to Markov and exchangeable cascades, 
\textit{Trans. Amer. Math. Soc.} 348 (1996), 585--632. 

\end{thebibliography}
\end{document}